\documentclass[11pt,leqno]{article}

\usepackage[T1]{fontenc}
\usepackage{lmodern}
\usepackage{microtype}
\usepackage{etoolbox}
\usepackage{needspace}
\usepackage[a4paper,margin=1.12in]{geometry}
\usepackage{amsmath,amssymb,amsthm,mathtools,mathrsfs}
\usepackage{enumitem}
\usepackage{booktabs,array,tabularx,longtable}
\usepackage{xcolor}
\usepackage{hyphenat}
\usepackage{tikz}
\usetikzlibrary{arrows.meta,calc}
\definecolor{refcitepink}{cmyk}{0,1,0,0}
\usepackage[colorlinks=true,linkcolor=blue,citecolor=refcitepink,
  urlcolor=black,linkbordercolor={1 0 0},citebordercolor={0 1 0},
  urlbordercolor={0 1 1},pdfborder={0 0 1}]{hyperref}
\usepackage[noabbrev]{cleveref}
\makeatletter
\let\IF@original@linkurl\hyper@linkurl
\def\hyper@linkurl#1#2{%
  \begingroup
  \hypersetup{pdfborder={0 0 1}}%
  \IF@original@linkurl{#1}{#2}%
  \endgroup
}
\makeatother

\allowdisplaybreaks[2]
\numberwithin{equation}{section}
\setlist{nosep,leftmargin=2em}

\newtheorem{theorem}{Theorem}[section]
\newtheorem{lemma}[theorem]{Lemma}
\newtheorem{proposition}[theorem]{Proposition}
\newtheorem{corollary}[theorem]{Corollary}
\newtheorem{frobeniuslemma}[theorem]{Lemma}
\theoremstyle{definition}
\newtheorem{definition}[theorem]{Definition}
\newtheorem{convention}[theorem]{Convention}

\theoremstyle{remark}

\BeforeBeginEnvironment{comment}{\Needspace{7\baselineskip}}

\crefname{theorem}{theorem}{theorems}
\crefname{lemma}{lemma}{lemmas}
\crefname{proposition}{proposition}{propositions}
\crefname{corollary}{corollary}{corollaries}
\crefname{definition}{definition}{definitions}
\crefname{convention}{convention}{conventions}
\crefname{remark}{remark}{remarks}

\newcommand{\A}{\mathcal A}
\newcommand{\K}{\mathcal K}

\newcommand{\Lus}{\mathrm{Lus}}
\newcommand{\Dr}{\mathrm{Dr}}
\newcommand{\ULus}{U_{\A}^{\Lus}}
\newcommand{\UDr}{U_{\A}^{\Dr}}
\newcommand{\UDrn}[1]{U_{\A}^{\Dr,(#1)}}
\newcommand{\Bstar}{B^{\ast}}
\newcommand{\Cstar}{C^{\ast}}

\newcommand{\steptag}[1]{\stepcounter{equation}\tag{#1}}

\tikzset{
  dynkin vertex/.style={circle,draw=black,fill=white,
    inner sep=0pt,minimum size=5.6pt,line width=.45pt},
  dynkin bond/.style={draw=black,line width=.45pt},
  dynkin label/.style={font=\scriptsize,inner sep=1.5pt}
}
\newcommand{\DynkinNode}[3]{%
  \node[dynkin vertex,label={[dynkin label]below:$#3$}] (#1) at #2 {};%
}
\newcommand{\DynkinNodeAbove}[3]{%
  \node[dynkin vertex,label={[dynkin label]above:$#3$}] (#1) at #2 {};%
}
\newcommand{\DynkinDouble}[2]{%
  \draw[dynkin bond] ([yshift=1.05pt]#1.east)--
    ([yshift=1.05pt]#2.west);%
  \draw[dynkin bond] ([yshift=-1.05pt]#1.east)--
    ([yshift=-1.05pt]#2.west);%
}
\newcommand{\DynkinTriple}[2]{%
  \draw[dynkin bond] ([yshift=1.55pt]#1.east)--
    ([yshift=1.55pt]#2.west);%
  \draw[dynkin bond] (#1.east)--(#2.west);%
  \draw[dynkin bond] ([yshift=-1.55pt]#1.east)--
    ([yshift=-1.55pt]#2.west);%
}
\newcommand{\DynkinQuadruple}[2]{%
  \draw[dynkin bond] ([yshift=2.25pt]#1.east)--
    ([yshift=2.25pt]#2.west);%
  \draw[dynkin bond] ([yshift=.75pt]#1.east)--
    ([yshift=.75pt]#2.west);%
  \draw[dynkin bond] ([yshift=-.75pt]#1.east)--
    ([yshift=-.75pt]#2.west);%
  \draw[dynkin bond] ([yshift=-2.25pt]#1.east)--
    ([yshift=-2.25pt]#2.west);%
}
\newcommand{\DynkinArrow}[2]{%
  \draw[-{Latex[length=2.7pt,width=3.2pt]},line width=.42pt]
    ($(#1)!.46!(#2)$)--($(#1)!.54!(#2)$);%
}

\title{Integral Forms of Twisted Quantum Affine Algebras and Quantum Frobenius }
\author{Ziyang Jiang and Naihuan Jing}
\date{}
\hypersetup{
  pdftitle={Integral Forms of Twisted Quantum Affine Algebras and Quantum Frobenius},
  pdfauthor={Ziyang Jiang and Naihuan Jing}
}

\begin{document}
\maketitle

\begin{abstract}
For twisted quantum affine algebras, we compare Lusztig's restricted integral form
with the integral form generated by group-like toral elements and
divided powers of real Drinfeld current modes.
The two forms coincide over \(\mathbb Z[q,q^{-1}]\) outside type
\(A_{2n}^{(2)}\).
In type \(A_{2n}^{(2)}\), the smallest ordinary localization over which they coincide is
\(\mathbb Z[q,q^{-1}][\Phi_{4r}(q)^{-1}\mid r\geq1]\) when \(n=1\),
and \(\mathbb Z[q,q^{-1}][\Phi_{8r}(q)^{-1}\mid r\geq1]\)
when \(n\geq2\), where $\Phi_k(q)$ are cyclotomic polynomials.
In particular, no localization obtained by inverting finitely many
nonzero Laurent polynomials suffices in this family.

\end{abstract}



\section*{Introduction}

Quantum affine algebras admit two complementary presentations.  The
Drinfeld--Jimbo definition is attached directly to an affine Cartan datum
and is expressed in Chevalley--Serre generators \cite{Dr1,Ji}, whereas
Drinfeld's new realization is organized by loop modes and current generators
\cite{Dr2}.  The current realization is particularly well adapted to the study of 
finite-dimensional representations \cite{CP} and the construction of
level‑one vertex representations \cite{J1}, while the Drinfeld--Jimbo
presentation is the natural setting for Lusztig's divided powers, braid
symmetries, and specialization at roots of unity. Explicit formulas for the comparison between the Drinfeld--Jimbo and
Drinfeld current presentations were given by Jing
\cite{J2} and, for twisted types including
\(A_{2n}^{(2)}\), were also recorded by Jing and Misra
\cite{JM}. These works stated the comparison as an
isomorphism; a complete proof over \(\mathbb C(q)\) was later obtained by Damiani 
\cite{D2, D3}.
A braid group construction of it appeared in \cite{JZ}.

The integral case is more interesting. Although the comparison isomorphism identifies the two algebras over
\(\mathbb C(q)\), it need not identify their prescribed
\(\mathbb Z[q,q^{-1}]\)-forms.  A denominator
invisible over the rational function field can change the specialized
algebra, and a generating set adapted to current modes need not generate
the Chevalley divided-power lattice.  The comparison also depends on exactly
which current generators are placed in the integral form; this point is part
of the definition, not a harmless normalization choice.

There is a useful classical precedent.  The works of Chevalley and Kostant
initiated the divided-power integral theory for finite-dimensional
semisimple Lie algebras \cite{C,Ko}.  Garland treated untwisted loop algebras
\cite{G}, and Mitzman constructed integral bases for affine Kac--Moody
algebras in the Chevalley presentation \cite{M}; see also \cite{Pr}. Akasaka~\cite{A} constructed an integral PBW basis and an integral
crystal basis of $U_q^+(A_2^{(2)})$, but did not determine the least
ordinary localization over which the integral forms in Drinfeld presentation and Lusztig's restricted
integral forms coincide. At the classical specialization \(q\mapsto1\), Damiani--Paolini proved in
rank one that the integral algebra generated by the divided powers of
the Drinfeld generators of \(A_2^{(2)}\) is properly contained in the
corresponding Chevalley--Mitzman integral form \cite{DP}. Paolini
subsequently extended this classical comparison to all affine types
and identified \(A_{2n}^{(2)}\) as the exceptional family \cite{P}.  These classical
results predict where an obstruction may occur, while they do not determine
its cyclotomic support in the generic quantum algebra.

The comparison of the integral forms for quantum twisted affine algebras is substantially complex and new ideas are needed. We pursue a structural
reduction for the comparison problem in this paper. We use local
\(q\)-Weyl contractions and braid transport to establish sufficiency, while
modified quantum Frobenius and parity separation detect necessity. This approach does not require a global computation of all commutation relations, which were extremely complicated in quantum situation.

The untwisted simply-laced case provides an important quantum precedent.
In this setting, the real-root comparison and the recovery of the affine
simple-root divided powers established by Beck--Chari--Pressley show that
no additional coefficient obstruction occurs
\cite{BCP}. Chari--Jing subsequently
recorded the resulting equality between the corresponding Drinfeld and
Lusztig integral forms \cite[Prop. 1.2]{CJ}.

The {\it goal} of this paper is to compare the ''bare real-current core'' generated as an \(\mathcal A\)-algebra by the group-like toral elements \(K_\delta^{\pm1}\), \(K_i^{\pm1}\),
and the divided powers of the Drinfeld current modes \(x_{i,r}^{\pm}\) with the full Lusztig's divided-power form. We denote the bare real-current core by
\(\UDr(\mathfrak g)\). No imaginary generator or Heisenberg generators, no additional real-root vector, and no toral quantum
binomial is inserted into this definition. 
This choice separates the genuine real-current obstruction from artifacts of
an enlarged generating set.

Our first result is that no obstruction occurs outside twisted type \(A_{2n}^{(2)}\).
For
\[
 A_{2n-1}^{(2)},\quad D_{n+1}^{(2)},\quad E_6^{(2)},\quad D_4^{(3)},
\]
the bare real-current core equals Lusztig's integral form over \(\A\).  The proof does not use information of a global PBW basis.  An ordinary \(A_1^{(1)}\) vertex supplies a
direct integral starting point for each finite color.  Current translations
and finite braid operators preserve the core; an explicit extended Weyl element
then carries one starting root \(\alpha_{i^\ast}\)
to \(\alpha_0\). Lifting that extended Weyl group element by a braid operator transports \(E_{i^\ast}\) to \(E_0\). Finite braid stability then places the affine simple divided powers in the current integral form, which therefore contains the full
Lusztig form.

Type \(A_{2n}^{(2)}\) stands out in the family and splits sharply according to rank.  Let
\[
 \Bstar=\A[\Phi_{8r}(q)^{-1}\mid r\geq1],\qquad
 \Cstar=\A[\Phi_{4r}(q)^{-1}\mid r\geq1],
\]
where $\Phi_k(q)\in \Bbb Z[q]$ are the cyclotomic polynomials of degree $\phi(k)$.

For \(A_2^{(2)}\), a single \(q\)-Weyl cell produces explicit coefficients
\[
 d_a=\prod_{h=1}^{a}(q^{2h}+1)(q^{4h}+1)
\]
with \(d_aE_0^{(a)},d_aF_0^{(a)}\in\UDr\).  Their irreducible support is
contained in the \(\Phi_{4r}\)'s, so \(\Cstar\) is sufficient.  For
\(A_{2n}^{(2)}\), \(n\geq2\), the relevant coefficients are instead
\[
 c_a=\prod_{h=1}^{a}(q^{4h}+1).
\]
We first create \(c_aE_{\eta_n}^{(a)}\) inside a special three-root cell.
Denominator-free braid transport carries this integrality along a precisely
normalized chain from \(\eta_n\) to \(\eta_1=\alpha_0\).  Since
the nonunit factors of \(c_a\) are exactly supported on \(\Phi_{8r}\), the ring
\(\Bstar\) is sufficient.

At roots of unity we use Lusztig's modified quantum Frobenius formalism
\cite{L}.  Sufficiency alone does not identify a least ring.  For this we introduce the
coefficient ideals
\[
 \mathfrak c_{n,r}:=(\UDrn{n}:_{\A}E_0^{(r)}).
\]
At each relevant cyclotomic prime ideal, a modified quantum Frobenius
map preserves the root degree of every nonzero block. We equip the scaled root lattice with the parity given by the coefficient
of the scaled affine simple root $r\alpha_0$ modulo $2$.
With respect to this parity, every nonzero image of a real-current block
is even, whereas the image of $E_0^{(r)}1_0$ is nonzero and odd. This forces the relevant coefficient ideal to be contained in that
cyclotomic prime ideal. The standard Frobenius theorem applies at
\((\Phi_{8r}(q))\) for \(r\geq2\), while a second scaled datum treats
\((\Phi_{4r}(q))\) in rank one for odd \(r\geq3\).

Two endpoint roots require special treatment: \((\Phi_8)\) in every rank and \((\Phi_4)\) in rank one.
The standard Frobenius size condition fails there, so it would be circular to
invoke the same theorem.  We instead form a common triangular quotient of the
specialized modified form.  Root-of-unity Gaussian vanishing shows that triangular straightening
preserves badness, where a word is bad if some associated weight
\(\lambda_j(w)\) lies outside \(X^\bullet\), or if it contains a
divided-power letter \(E_i^{(a)}\) or \(F_i^{(a)}\) with
\(\ell_i\nmid a\).  Consequently, the span \(J\) of the evaluations
of all bad words agrees with the span of the evaluations of the
triangularly ordered bad words.  The resulting subspace \(J\) is
root graded, and concatenation of words shows that it is a two-sided
ideal. The quotient kills the evaluations of all bad words, preserves \(E_01_0\), and carries the same parity
separation.  Thus the endpoint primes are forced as well.
The detailed Gaussian, straightening, and half-ideal arguments are
collected in Appendix~\ref{app:endpoint-technical}, while
Section~\ref{sec:endpoint-obstructions} retains the quotient
construction and its role in the parity argument.

Combining the constructive and obstructive directions gives the following
classification.  Among all ordinary localizations \(B=S^{-1}\A\subset\K\),
\[
 B\UDr(A_2^{(2)})=B\ULus(A_2^{(2)})
 \quad\Longleftrightarrow\quad
 \Phi_{4r}(q)\in B^{\times}\ \text{for every }r\geq1,
\]
and, for \(n\geq2\),
\[
 B\UDr(A_{2n}^{(2)})=B\ULus(A_{2n}^{(2)})
 \quad\Longleftrightarrow\quad
 \Phi_{8r}(q)\in B^{\times}\ \text{for every }r\geq1.
\]
Consequently the least rings are \(\Cstar\) in rank one and \(\Bstar\) in
higher rank, whereas \(\A\) itself is least in every other twisted family. In particular, for every \(n\geq1\), no localization of \(\A\)
obtained by inverting only finitely many elements is sufficient to
make the two integral forms of type \(A_{2n}^{(2)}\) coincide.

Both exceptional coefficient rings have the same classical shadow.  For
\(m>1\), one has \(\Phi_m(1)=p\) when \(m\) is a power of the prime \(p\),
and \(\Phi_m(1)=1\) otherwise.  Consequently, base change along
\(q\mapsto1\) sends both \(\Cstar\) and \(\Bstar\) to
\(\mathbb Z[1/2]\), which is consistent with the classical
conclusions obtained by Damiani--Paolini \cite{DP} and Paolini \cite{P}. The classical strictness
results are therefore a consistency check, but they cannot imply the
individual coefficient-ideal obstructions proved here.

The paper is organized as follows.  Section~\ref{sec:common} fixes the comparison isomorphism between the Drinfeld and
Drinfeld--Jimbo realizations and places the two integral forms in a common
ambient algebra.  The remaining argument is divided into two parts.  Part~
\ref{part:nonexceptional} consists of
\cref{sec:nonexceptional-data,sec:real-current-start,%
sec:finite-Weyl-endpoint} and proves equality in the
four non-\(A_{2n}^{(2)}\) twisted families.  Part~\ref{part:a2n} begins with
the type \(A_{2n}^{(2)}\) data in \cref{sec:a2n-root-data}.
\Cref{sec:rank-one-sufficiency,sec:quantum-dp-scaled-roots} prove sufficiency
in rank one and higher rank, respectively; \cref{sec:coefficient-ideals}
develops the coefficient-ideal and parity-separation machinery; and
\cref{sec:scaled-root-datum,sec:endpoint-obstructions} prove
cyclotomic necessity for \(r\geq2\) and \(r=1\), respectively, including the
two exceptional endpoints.  Finally, Section~\ref{sec:least-ordinary-localization}
assembles the implications, proves minimality, and shows that
in type \(A_{2n}^{(2)}\), \(n\geq1\), no localization obtained
by inverting finitely many nonzero Laurent polynomials makes
the two integral forms equal.  Appendix~\ref{app:endpoint-technical}
contains the technical proofs for the endpoint quotient used in
Section~\ref{sec:endpoint-obstructions}.
\section{Conventions, two integral forms, and minimality}
\label{sec:common}

\subsection{Types, parameters, and the fixed comparison isomorphism}

Throughout the paper, set
\[
  \mathcal A=\mathbb Z[q,q^{-1}],
  \qquad
  \mathcal K=\mathbb Q(q).
\]
Let \(\mathfrak g\) be an irreducible twisted affine Cartan type in
\[
 \mathcal T:=
 \left\{
 A_{2n}^{(2)}\ (n\geq1),\quad
 A_{2n-1}^{(2)}\ (n\geq3),\quad
 D_{n+1}^{(2)}\ (n\geq2),\quad
 E_6^{(2)},\quad D_4^{(3)}
 \right\}.
\]
The low-rank isomorphism \(A_3^{(2)}\cong D_3^{(2)}\) is understood through
the case \(n=2\) of the family \(D_{n+1}^{(2)}\).  Write
\(I=\{0\}\sqcup I_0\) for the set of affine nodes.  We work throughout with the derived quantum affine
algebra: the central element is retained, whereas the degree element is
omitted.


Let \(A=(a_{ij})_{i,j\in I}\) be the affine Cartan matrix, and choose
relatively prime positive integer symmetrizers \((d_i)_{i\in I}\) satisfying
\[
d_i a_{ij}=d_j a_{ji}
\qquad(i,j\in I).
\]
Outside type \(A_{2n}^{(2)}\), the \(d_i\) are those used in
\cref{eq:nonexceptional-type-table}.  In type \(A_{2n}^{(2)}\), they are
normalized by
\[
 (d_0,d_1,\ldots,d_n)=(4,2,\ldots,2,1)\quad(n\geq2),
 \qquad
 (d_0,d_1)=(4,1)\quad(n=1).
\]
Set
\begin{equation}\label{eq:local-quantum-parameters}
q_i:=q^{d_i}
\qquad(i\in I).
\end{equation}
We use the corresponding integral invariant form
\[
 (\alpha_i\mid\alpha_j):=d_i a_{ij};
\]
thus \((\alpha_i\mid\alpha_i)=2d_i\), and the Cartan action and all
root parameters are expressed through this form.
Thus the local parameters in type \(A_{2n}^{(2)}\) are
\((q_0,q_1,\ldots,q_n)=(q^4,q^2,\ldots,q^2,q)\) for \(n\geq2\),
and \((q_0,q_1)=(q^4,q)\) for \(n=1\).

Let \(\Delta^{\mathrm{re}}\) be the set of affine real roots.
For \(i\in I_0\), define
\begin{equation}\label{eq:allowed-current-modes}
\nu_i
:=
\min\bigl\{
m>0 \bigm| m\delta-\alpha_i\in\Delta^{\mathrm{re}}
\bigr\},
\qquad
I_i:=\nu_i\mathbb Z\subseteq\mathbb Z.
\end{equation}
Thus \(I_i\) is the set of allowed current modes of color \(i\). See Figures~\ref{fig:nonexceptional-dynkin} and~\ref{fig:a2n-dynkin} for numbering of Dynkin nodes .

In the four non-\(A_{2n}^{(2)}\) types, \(\nu_i=d_i\);
in type \(A_{2n}^{(2)}\), \(\nu_i=1\).
For an indeterminate \(v\), use the symmetric quantum integers,
factorials, and binomial coefficients
\begin{equation}\label{eq:quantum-number-convention}
 [r]_v:=\frac{v^r-v^{-r}}{v-v^{-1}},\qquad
 [a]_v!:=\prod_{r=1}^a[r]_v,\qquad
 \begin{bmatrix}a\\b\end{bmatrix}_v
 :=\frac{[a]_v!}{[b]_v![a-b]_v!},
\end{equation}
where \(r\in\mathbb Z\), \(0\leq b\leq a\), and \([0]_v!:=1\).

\medskip
\noindent\textbf{The Drinfeld--Jimbo presentation.}
The algebra \(U_q^{\mathrm{DJ}}(\mathfrak g)\) is the unital
\(\mathcal K\)-algebra generated by
\{\(E_i,F_i,K_i^{\pm1}\)| \(i\in I\)\}, subject to
\begin{subequations}\label{eq:dj-defining-relations}
\begin{align}
 &K_iK_i^{-1}=K_i^{-1}K_i=1,\qquad K_iK_j=K_jK_i,
 \label{eq:dj-torus}\\
 &K_iE_jK_i^{-1}=q_i^{a_{ij}}E_j,\qquad
 K_iF_jK_i^{-1}=q_i^{-a_{ij}}F_j,
 \label{eq:dj-cartan-action}\\
 &[E_i,F_j]=\delta_{ij}\frac{K_i-K_i^{-1}}{q_i-q_i^{-1}}
 \qquad(i,j\in I),
 \label{eq:dj-mixed-commutator}\\
 &\sum_{p=0}^{1-a_{ij}}(-1)^p
 E_i^{(p)}E_jE_i^{(1-a_{ij}-p)}=0
 \qquad(i\neq j),
 \label{eq:dj-positive-serre}\\
 &\sum_{p=0}^{1-a_{ij}}(-1)^p
 F_i^{(p)}F_jF_i^{(1-a_{ij}-p)}=0
 \qquad(i\neq j).
 \label{eq:dj-negative-serre}
\end{align}
\end{subequations}
Here the Chevalley divided powers are
\begin{equation}\label{eq:chevalley-divided-power-definition}
 E_i^{(a)}:=\frac{E_i^a}{[a]_{q_i}!},\qquad
 F_i^{(a)}:=\frac{F_i^a}{[a]_{q_i}!}
 \qquad(i\in I,\ a\geq0).
\end{equation}
Write \(Q:=\bigoplus_{i\in I}\mathbb Z\alpha_i\) for the root lattice.
For \(\mu=\sum_{i\in I}m_i\alpha_i\in Q\), put
\(K_\mu:=\prod_{i\in I}K_i^{m_i}\); in particular, writing
\(\delta=\alpha_0+\sum_{i\in I_0}r_i\alpha_i\), we have
\(K_\delta=K_0\prod_{i\in I_0}K_i^{r_i}\).
Put
\[
Q_+:=\sum_{i\in I}\mathbb N\alpha_i.
\]
Let \(U_q^+\) denote the \(\mathcal K\)-subalgebra of
\(U_q^{\mathrm{DJ}}(\mathfrak g)\) generated by the \(E_i\),
\(i\in I\).  Its root-lattice grading is written as
\[
U_q^+
=
\bigoplus_{\gamma\in Q_+}(U_q^+)_{\gamma},
\qquad
\deg(E_i)=\alpha_i.
\]

\medskip
\noindent\textbf{The Drinfeld presentation.}
Set \(\nu_{ij}:=\max\{\nu_i,\nu_j\}\), and write
\([u,v]_a:=uv-avu\).
For \(r\in\mathbb Z\setminus\{0\}\), define
\[
 b_{ij}(r):=
 \begin{cases}
  0,&\nu_{ij}\nmid r,\\[2pt]
  \displaystyle
  \frac{[2r]_q\bigl(q^{2r}+(-1)^{r-1}+q^{-2r}\bigr)}{r},
   &\mathfrak g=A_{2n}^{(2)},\ i=j=n,\\[7pt]
  \displaystyle
  \frac{[(r/\nu_{ij})a_{ij}]_{q_i}}{r/\nu_{ij}},
   &\text{otherwise}.
 \end{cases}
\]
The algebra \(U_q^{\mathrm{Dr}}(\mathfrak g)\) is the unital
\(\mathcal K\)-algebra generated by
\[
 \{K_\delta^{\pm1},\quad K_i^{\pm1},\quad
 x_{i,r}^{\pm},\quad h_{i,s}
 \qquad |i\in I_0,\ r\in I_i,\
       s\in I_i\setminus\{0\}\}
\]
with the relations listed below.
Use the zero-mode convention
\begin{equation}\label{eq:dr-zero-mode-convention}
 x_{i,r}^{\pm}=h_{i,r}=0\qquad(r\notin I_i),
\end{equation}
and define
\begin{equation}\label{eq:dr-cartan-series}
 \sum_{a\geq0}\phi_{i,\pm a}^{\pm}u^a
 =K_i^{\pm1}\exp\!\left(
  \pm(q_i-q_i^{-1})\sum_{a>0}h_{i,\pm a}u^a
 \right),
 \qquad \phi_{i,t}^{\pm}=0\quad\text{if }\pm t<0.
\end{equation}
For \(\varepsilon\in\{1,-1\}\), write
\(x_{i,r}^{1}=x_{i,r}^{+}\), \(x_{i,r}^{-1}=x_{i,r}^{-}\).
The notation \(\operatorname{Sym}_{r_1,\ldots,r_a}\) means the sum over
all permutations of these mode indices. For \(m=-a_{ij}\in\{2,3\}\), let
\[
\begin{aligned}
 P_{ij}^{\varepsilon}(u,v;s)
 &:={}
 x_{j,s}^{\varepsilon}x_{i,u}^{\varepsilon}x_{i,v}^{\varepsilon}
 -[2]_{q^m}x_{i,u}^{\varepsilon}x_{j,s}^{\varepsilon}x_{i,v}^{\varepsilon}
 +x_{i,u}^{\varepsilon}x_{i,v}^{\varepsilon}x_{j,s}^{\varepsilon}.
\end{aligned}
\]
In the following list, \(i,j\in I_0\),
\(\varepsilon,\eta\in\{1,-1\}\), and \(z\) ranges over the generators.
In the first five relations, input modes belong to their respective
\(I_i\), and Heisenberg modes are nonzero; all mode indices in the
remaining relations range over \(\mathbb Z\), using the zero convention
above. In the last two relations, \(M=1-a_{ij}\) and \(m=-a_{ij}\),
respectively. The defining relations are
\begingroup
\allowdisplaybreaks[1]
\begin{subequations}\label{eq:dr-defining-relations}
\begin{align}
 &\begin{aligned}
 [K_\delta,z]&=0,\qquad
 K_\delta K_\delta^{-1}=K_\delta^{-1}K_\delta=1,\\
 K_iK_i^{-1}&=K_i^{-1}K_i=1,\qquad K_iK_j=K_jK_i,
 \end{aligned}
 \label{eq:dr-torus}\\[3pt]
 &K_ix_{j,r}^{\varepsilon}K_i^{-1}
 =q_i^{\varepsilon a_{ij}}x_{j,r}^{\varepsilon},
 \qquad [K_i,h_{j,s}]=0,
 \label{eq:dr-cartan-action}\\[3pt]
 &[x_{i,r}^{+},x_{j,s}^{-}]
 =\delta_{ij}\frac{
 K_\delta^{-s}\phi_{i,r+s}^{+}-K_\delta^{-r}\phi_{i,r+s}^{-}
 }{q_i-q_i^{-1}},
 \label{eq:dr-mixed-current}\\[3pt]
 &[h_{i,r},x_{j,s}^{\varepsilon}]
 =\varepsilon b_{ij}(r)
 K_\delta^{(r-\varepsilon|r|)/2}x_{j,r+s}^{\varepsilon},
 \label{eq:dr-heisenberg-current}\\[3pt]
 &[h_{i,r},h_{j,s}]
 =\delta_{r+s,0}b_{ij}(r)
 \frac{K_\delta^r-K_\delta^{-r}}{q_j-q_j^{-1}},
 \label{eq:dr-heisenberg}\\[3pt]
 &\begin{aligned}
 &[x_{i,r+\varepsilon\nu_{ij}}^{\varepsilon},
   x_{j,s}^{\varepsilon}]_{q_i^{a_{ij}}}
 +[x_{j,s+\varepsilon\nu_{ij}}^{\varepsilon},
   x_{i,r}^{\varepsilon}]_{q_i^{a_{ij}}}=0
 &&(i\neq j,\ a_{ij}<0),
 \end{aligned}
 \label{eq:dr-distinct-current}\\[3pt]
 &\operatorname{Sym}_{r_1,r_2}
 [x_{i,r_1+\varepsilon\nu_i}^{\varepsilon},
  x_{i,r_2}^{\varepsilon}]_{q_i^2}=0
 \quad\bigl((\mathfrak g,i)\neq(A_{2n}^{(2)},n)\bigr),
 \label{eq:dr-same-current}\\[3pt]
 &\begin{aligned}
 &\operatorname{Sym}_{r_1,r_2}\left(
 [x_{n,r_1+2\varepsilon}^{\varepsilon},x_{n,r_2}^{\varepsilon}]_{q^2}
 -q^4[x_{n,r_1+\varepsilon}^{\varepsilon},
       x_{n,r_2+\varepsilon}^{\varepsilon}]_{q^{-6}}
 \right)=0\\[-2pt]
 &\hspace{18em}(\mathfrak g=A_{2n}^{(2)}),
 \end{aligned}
 \label{eq:dr-exceptional-quadratic}\\[3pt]
 &\operatorname{Sym}_{r_1,r_2,r_3}
 \bigl[[x_{n,r_1+\varepsilon\eta}^{\varepsilon},
        x_{n,r_2}^{\varepsilon}]_{q^{2\eta}},
        x_{n,r_3}^{\varepsilon}\bigr]_{q^{4\eta}}=0
 \quad(\mathfrak g=A_{2n}^{(2)}),
 \label{eq:dr-exceptional-cubic}\\[3pt]
 &\begin{aligned}
 &\operatorname{Sym}_{r_1,\ldots,r_M}
 \sum_{p=0}^{M}(-1)^p\begin{bmatrix}M\\p\end{bmatrix}_{q_i}
 x_{i,r_1}^{\varepsilon}\cdots x_{i,r_p}^{\varepsilon}
 x_{j,s}^{\varepsilon}
 x_{i,r_{p+1}}^{\varepsilon}\cdots x_{i,r_M}^{\varepsilon}=0\\[-2pt]
 &\hspace{18em}(i\neq j,\ a_{ij}\in\{0,-1\}),
 \end{aligned}
 \label{eq:dr-current-serre}\\[3pt]
 &\operatorname{Sym}_{r_1,r_2}
 \sum_{p=0}^{m-1}q^{m-1-2p}
 P_{ij}^{\varepsilon}
 \bigl(r_1+\varepsilon(m-1-p),\,r_2+\varepsilon p;\,s\bigr)=0
 \quad(m\in\{2,3\}).
 \label{eq:dr-multiple-edge-serre}
\end{align}
\end{subequations}
\endgroup

\medskip
\noindent\textbf{The comparison map and braid action.}
Over \(\K\), we fix the comparison isomorphism
\begin{equation}\label{eq:comparison-isomorphism}
\Psi_{\mathfrak g}\colon
U_q^{\mathrm{Dr}}(\mathfrak g)
\xrightarrow{\;\sim\;}
U_q^{\mathrm{DJ}}(\mathfrak g),
\end{equation}
whose values on the generators are specified below.
We use Beck's braid operators \(T_i\) \cite[Section~1.3]{B}.
The standard finite and extended braid lifts preserve the Lusztig
integral form:
\begin{equation}
 \label{eq:braid-integral-stability}
 T_w^{\pm1}\bigl(U_{\mathcal A}^{\mathrm{Lus}}(\mathfrak g)\bigr)
 =U_{\mathcal A}^{\mathrm{Lus}}(\mathfrak g);
\end{equation}
see \cite[Section~41.1.2]{L}.
For \(i\neq j\), let \(m=-a_{ij}\). The braid operators $T_i$ act on the positive Chevalley generators by
\begin{equation}
\label{eq:braid-generator-formulas}
\begin{aligned}
 T_i(E_i)&=-F_iK_i,
 & T_i^{-1}(E_i)&=-K_i^{-1}F_i,\\
 T_i(E_j)
 &=\sum_{r=0}^{m}(-1)^{m+r}q_i^{-r}
   E_i^{(m-r)}E_jE_i^{(r)},\\
 T_i^{-1}(E_j)
 &=\sum_{r=0}^{m}(-1)^{m+r}q_i^{-r}
   E_i^{(r)}E_jE_i^{(m-r)}.
\end{aligned}
\end{equation}

For a reduced expression \(w=s_{i_1}\cdots s_{i_t}\tau\), with
\(\tau\) a length-zero diagram automorphism, set
\begin{equation}\label{eq:standard-braid-lift-definition}
 T_w:=T_{i_1}\cdots T_{i_t}T_\tau,
 \qquad
 T_\tau(E_i)=E_{\tau(i)},\quad
 T_\tau(F_i)=F_{\tau(i)},\quad
 T_\tau(K_i)=K_{\tau(i)}.
\end{equation}
For ordinary Weyl-group elements, \(\tau=1\).
Weyl-group products act from right to left, and
\(T_i(K_\mu)=K_{s_i(\mu)}\).
We use the semilinear antiautomorphism \(\Omega\) given by
\begin{equation}
 \label{eq:Omega}
 \Omega(q)=q^{-1},\qquad
 \Omega(E_i)=F_i,\qquad
 \Omega(F_i)=E_i,\qquad
 \Omega(K_\mu)=K_{-\mu}.
\end{equation}
It is easy to check that 
\begin{equation}
 \label{eq:Omega-commutes-braid}
 \Omega T_w=T_w\Omega.
\end{equation}

Thus \(T_i^{\pm1}(F_j)=\Omega(T_i^{\pm1}(E_j))\), which gives the formulas for the braid action on the generators.

\medskip
\noindent\textbf{Real root vectors.}
For \(i\in I_0\), let \(\omega_i^\vee\) be the fundamental
coweight, and put
\begin{equation}\label{eq:fundamental-translation-definition}
 \lambda_i:=\nu_i\omega_i^\vee,\qquad
 t_{\lambda_i}(\alpha)
 :=\alpha-\langle\lambda_i,\alpha\rangle\delta,
 \qquad T_{\lambda_i}:=T_{t_{\lambda_i}}.
\end{equation}
Here \(\langle\omega_i^\vee,\alpha_j\rangle=\delta_{ij}\) for
\(j\in I_0\), and \(\langle\lambda_i,\delta\rangle=0\).
Order \(I_0\) increasingly in the non-\(A_{2n}^{(2)}\) types
and decreasingly in type \(A_{2n}^{(2)}\), and denote this order by
\((j_1,\ldots,j_{|I_0|})\).
For each \(j\in I_0\), choose a reduced expression for
\(t_{\lambda_j}\) ending in \(s_j\); this is possible because
\(t_{\lambda_j}(\alpha_j)=\alpha_j-\nu_j\delta<0\).
Concatenate these expressions in the chosen order and move all
length-zero diagram automorphisms to the right. This gives a reduced
expression
\begin{equation}\label{eq:fixed-translation-word}
 t_{\lambda_{j_1}}\cdots t_{\lambda_{j_{|I_0|}}}
 =s_{i_1}\cdots s_{i_N}\tau,
 \qquad N=\sum_{i\in I_0}\ell(t_{\lambda_i}).
\end{equation}
In type \(A_{2n}^{(2)}\), choose the initial block
\begin{equation}\label{eq:a2n-initial-translation-word}
 (s_0s_1\cdots s_n)^n=t_{\lambda_n}.
\end{equation}
Extend the indices to \(\mathbb Z\) by \(i_{r+N}=\tau(i_r)\), and set
\begin{equation}\label{eq:real-root-enumeration}
 w_r:=
 \begin{cases}
  s_{i_1}\cdots s_{i_{r-1}},&r\geq1,\\
  s_{i_0}\cdots s_{i_{r+1}},&r\leq0,
 \end{cases}
 \qquad \rho_r:=w_r(\alpha_{i_r}).
\end{equation}
The roots \(\rho_r\), \(r\in\mathbb Z\), enumerate
\(\Delta_+^{\mathrm{re}}\) without repetition. We define
\begin{equation}\label{eq:real-root-vector-definition}
 E_{\rho_r}:=
 \begin{cases}
  T_{i_1}\cdots T_{i_{r-1}}(E_{i_r}),&r\geq1,\\
  T_{i_0}^{-1}\cdots T_{i_{r+1}}^{-1}(E_{i_r}),&r\leq0,
 \end{cases}
 \qquad F_{\rho_r}:=\Omega(E_{\rho_r}),
\end{equation}
with empty products equal to the identity. These are straightforward generalization of untwisted cases (Damiani
\cite{D1}). 
Note that the initial block in type \(A_{2n}^{(2)}\) is the reduced block therein. 
In particular,
\begin{equation}\label{eq:translation-real-root-vectors}
\begin{aligned}
 E_{m\nu_i\delta+\alpha_i}&=T_{\lambda_i}^{-m}(E_i)
 &&(m\geq0),\\
 E_{m\nu_i\delta-\alpha_i}&=T_{\lambda_i}^{m}T_i^{-1}(E_i)
 &&(m>0).
\end{aligned}
\end{equation}

Choose the bipartite sign function \(o:I_0\to\{\pm1\}\) satisfying
\begin{equation}\label{eq:current-sign-convention}
\begin{gathered}
 o(i)o(j)=-1\qquad(i\neq j,\ a_{ij}\neq0),\\
 o(i)=1\quad(\mathfrak g\neq A_{2n}^{(2)},\ a_{ij}=-2),\qquad
 o(i)=(-1)^{n-i}\quad(\mathfrak g=A_{2n}^{(2)}).
\end{gathered}
\end{equation}
Then the comparison isomorphism can be stated as follows.
\begin{equation}\label{eq:comparison-current-images}
\begin{aligned}
 \Psi_{\mathfrak g}(K_\delta)&=K_\delta,
 &\Psi_{\mathfrak g}(K_i)&=K_i,\\
 \Psi_{\mathfrak g}(x_{i,\nu_i r}^{+})
 &=o(i)^rT_{\lambda_i}^{-r}(E_i),
 &\Psi_{\mathfrak g}(x_{i,\nu_i r}^{-})
 &=o(i)^rT_{\lambda_i}^{r}(F_i)
 \qquad(r\in\mathbb Z).
\end{aligned}
\end{equation}
For complete proofs, see 
\cite{D2, D3} (also see \cite{JZ}).  
The coefficients in the defining relations and in the formulas for the comparison map belong to
\(\mathbb Q(q)\), so the comparison homomorphism is first defined over
\(\mathcal K\).  Since \(\mathbb C(q)/\mathbb Q(q)\) is faithfully
flat, its kernel and cokernel already vanish over \(\mathbb Q(q)\).  Thus
\eqref{eq:comparison-isomorphism} is indeed an isomorphism of
\(\mathcal K\)-algebras.

We henceforth identify the Drinfeld generators with their images under
\(\Psi_{\mathfrak g}\) and write \(U_q(\mathfrak g)\) for the common algebra.
All root--coroot pairings are denoted by \(\langle h_i,\gamma\rangle\).
For \(m\geq1\), let \(\Phi_m(q)\) denote the \(m\)-th cyclotomic polynomial.

Whenever a modified form is used, we fix a regular realization
\((X,Y,I,\langle\, ,\,\rangle)\) of the same Cartan datum, solely to
define the weight idempotents, modified blocks, and modified braid
action. The ordinary quantum affine algebra remains the derived
subalgebra, so no degree element is added to either the ordinary integral
form.

For \(i\in I_0\), \(r\in I_i\), and \(a\geq0\), define
\begin{equation}\label{eq:current-divided-power-definition}
 (x_{i,r}^{\pm})^{(a)}
 :=\frac{(x_{i,r}^{\pm})^a}{[a]_{q_i}!},
\end{equation}
with the factorials of \eqref{eq:quantum-number-convention}.

Writing
\[
\delta=\alpha_0+\sum_{i\in I_0}r_i\alpha_i,
\]
the affine toral generator is recovered by
\begin{equation}
\label{eq:K0-recovery}
K_0
=
K_\delta\prod_{i\in I_0}K_i^{-r_i}.
\end{equation}
\begin{lemma}\cite[Lemma~3.2]{B}\label{lem:simple-to-simple}
Let $w$ lie in the (extended) affine Weyl group and let $i,j$ be nodes with
$w(\alpha_i)=\alpha_j$.  Then $ws_i$ and $s_jw$ are length additive and
\[
   T_wT_i=T_jT_w,\qquad T_w(E_i)=E_j,\qquad T_w(F_i)=F_j .
\]
\end{lemma}

\subsection{The two integral forms}
Throughout this paper, an \emph{integral form} of a
\(\mathcal K\)-algebra \(U\) means a unital
\(\mathcal A\)-subalgebra \(U_{\mathcal A}\subseteq U\)
for which the natural map
\(\mathcal K\otimes_{\mathcal A}U_{\mathcal A}\to U\)
is an isomorphism.
Freeness as an \(\mathcal A\)-module is not required
in this definition.
We use the divided powers defined in
\eqref{eq:chevalley-divided-power-definition}.

\begin{definition}[Lusztig restricted integral form]
\label{def:Lusztig-restricted-integral-form}
Inside \(U_q^{\mathrm{DJ}}(\mathfrak g)\), define
\[
U_{\mathcal A}^{\mathrm{Lus}}(\mathfrak g)
:=
\mathcal A
\Bigl\langle
E_i^{(a)},F_i^{(a)},K_i^{\pm1}
\ \Bigm|\
i\in I,\ a\geq0
\Bigr\rangle.
\]
\end{definition}

\begin{definition}[Bare real-current core or integral form in Drinfeld presentation]
\label{def:bare-real-current-core}
Inside the fixed ambient algebra
\(U_q^{\mathrm{Dr}}(\mathfrak g)\), define, uniformly in all types,
\begin{equation}
\label{eq:bare-real-current-core}
U_{\mathcal A}^{\mathrm{Dr}}(\mathfrak g)
:=
\mathcal A
\Bigl\langle
K_\delta^{\pm1},K_i^{\pm1},
(x_{i,r}^{\pm})^{(a)}
\ \Bigm|\
i\in I_0,\ r\in I_i,\ a\geq0
\Bigr\rangle.
\end{equation}
\end{definition}

\begin{lemma}
\label{lem:forward-inclusion}
For every \(\mathfrak g\in\mathcal T\),
\[
 \UDr(\mathfrak g)
 \subseteq U_{\mathcal A}^{\mathrm{Lus}}(\mathfrak g).
\]
\end{lemma}

\begin{proof}
The toral generators in \eqref{eq:bare-real-current-core} already
belong to \(U_{\mathcal A}^{\mathrm{Lus}}(\mathfrak g)\).
For \(i\in I_0\) and \(r\in I_i\), write \(r=\nu_i k\), where
\(k\in\mathbb Z\). Equation~\eqref{eq:comparison-current-images}
gives, for every \(a\geq0\),
\begin{equation}
\label{eq:current-divided-powers-braid-formula}
\begin{aligned}
 (x_{i,r}^{+})^{(a)}
 &=o(i)^{ka}T_{\lambda_i}^{-k}\bigl(E_i^{(a)}\bigr),\\
 (x_{i,r}^{-})^{(a)}
 &=o(i)^{ka}T_{\lambda_i}^{k}\bigl(F_i^{(a)}\bigr).
\end{aligned}
\end{equation}
Since \(o(i)\in\{\pm1\}\) and the extended braid operators preserve
\(U_{\mathcal A}^{\mathrm{Lus}}(\mathfrak g)\), both divided powers
belong to \(U_{\mathcal A}^{\mathrm{Lus}}(\mathfrak g)\).
This proves the inclusion.
\end{proof}

\begin{lemma}[\(q\)-Weyl normal ordering and the highest contraction]
\label{lem:q-Weyl-highest-contraction}
Let \(D\) be an \(\mathcal A\)-subalgebra of a \(\mathcal K\)-algebra, let
\(Q\in\mathcal A^\times\), and suppose that \(A,B,Z\) satisfy
\[
 Z=AB-BA,
 \qquad
 AZ=Q^{-2}ZA,
 \qquad
 ZB=Q^{-2}BZ.
\]
For \(m\geq0\), set
\[
 A_Q^{(m)}=\frac{A^m}{[m]_Q!},
 \qquad
 B_Q^{(m)}=\frac{B^m}{[m]_Q!}.
\]
Then, for all \(M,N\geq0\),
\begin{equation}
 \label{eq:q-Weyl-normal-ordering}
 A_Q^{(M)}B_Q^{(N)}
 =\sum_{t=0}^{\min(M,N)}
 Q^{-t(M+N-2t)-\binom t2}
 B_Q^{(N-t)}\frac{Z^t}{[t]_Q!}A_Q^{(M-t)}.
\end{equation}
If \(A_Q^{(m)},B_Q^{(m)}\in D\) for every \(m\geq0\), then
\[
 \frac{Z^a}{[a]_Q!}\in D
 \qquad(a\geq0).
\]
\end{lemma}

\begin{proof}
The three defining relations first imply, by induction on \(N\), that
\begin{equation}
 AB^N=B^NA+Q^{-(N-1)}[N]_Q B^{N-1}Z
 \qquad(N\geq1).
 \label{eq:q-Weyl-first-commutation}
\end{equation}
The case \(N=1\) is precisely \(Z=AB-BA\).  Assume that
\eqref{eq:q-Weyl-first-commutation} holds for \(N\).  Multiplication on the
right by \(B\) gives
\begin{align*}
 AB^{N+1}
 &=\bigl(B^NA+Q^{-(N-1)}[N]_QB^{N-1}Z\bigr)B\\
 &=B^{N+1}A+B^NZ+Q^{-(N-1)}[N]_QB^{N-1}ZB\\
 &=B^{N+1}A+\bigl(1+Q^{-(N+1)}[N]_Q\bigr)B^NZ\\
 &=B^{N+1}A+Q^{-N}[N+1]_QB^NZ.
\end{align*}
Here the third equality uses \(ZB=Q^{-2}BZ\), and the last one uses
\[
 1+Q^{-(N+1)}[N]_Q=Q^{-N}[N+1]_Q.
\]
This completes the induction.  Dividing
\eqref{eq:q-Weyl-first-commutation} by \([N]_Q!\) yields
\begin{equation}
 AB_Q^{(N)}=B_Q^{(N)}A+Q^{-(N-1)}B_Q^{(N-1)}Z.
 \label{eq:q-Weyl-divided-commutation}
\end{equation}

We next prove \eqref{eq:q-Weyl-normal-ordering} by induction on \(M\).
The cases \(M=0\) and \(M=1\) follow from the definition and
\eqref{eq:q-Weyl-divided-commutation}, respectively.  Put
\[
 c_{M,N,t}:=Q^{-t(M+N-2t)-\binom t2}.
\]
Assume that
\begin{equation}
 A_Q^{(M)}B_Q^{(N)}
 =\sum_{s=0}^{\min(M,N)}
 c_{M,N,s}B_Q^{(N-s)}\frac{Z^s}{[s]_Q!}A_Q^{(M-s)}.
 \label{eq:q-Weyl-induction-hypothesis}
\end{equation}
Since
\[
 A_Q^{(M+1)}B_Q^{(N)}
 =\frac{1}{[M+1]_Q}
 A\bigl(A_Q^{(M)}B_Q^{(N)}\bigr),
\]
fix an index \(s\) in \eqref{eq:q-Weyl-induction-hypothesis}.  Using
\eqref{eq:q-Weyl-divided-commutation} together with
\[
 AZ^s=Q^{-2s}Z^sA,
 \qquad
 AA_Q^{(M-s)}=[M-s+1]_QA_Q^{(M-s+1)},
 \qquad
 \frac{Z^{s+1}}{[s]_Q!}
 =[s+1]_Q\frac{Z^{s+1}}{[s+1]_Q!},
\]
we obtain
\begin{align*}
 A\left(
 B_Q^{(N-s)}\frac{Z^s}{[s]_Q!}A_Q^{(M-s)}
 \right)
 &={Q^{-2s}[M-s+1]_Q}
 B_Q^{(N-s)}\frac{Z^s}{[s]_Q!}A_Q^{(M-s+1)}\\
 &\quad+{Q^{-(N-s-1)}[s+1]_Q}
 B_Q^{(N-s-1)}\frac{Z^{s+1}}{[s+1]_Q!}A_Q^{(M-s)}.
\end{align*}

The second summand is omitted when \(s = N\). Fix the standard term
\[
 B_Q^{(N-t)}\frac{Z^t}{[t]_Q!}A_Q^{(M+1-t)}
\]
in the desired formula for \(M+1\).  By the preceding identity, it can arise
only from the first
summand with \(s=t\), or from the second summand with \(s=t-1\).  Its
coefficient is therefore
\begin{equation}
 \frac{
 Q^{-2t}[M-t+1]_Qc_{M,N,t}
 +Q^{-(N-t)}[t]_Qc_{M,N,t-1}}
 {[M+1]_Q},
 \label{eq:q-Weyl-two-contributions}
\end{equation}
where an out-of-range term is understood to be zero.  Directly from the
definition of \(c_{M,N,t}\),
\[
 Q^{-2t}c_{M,N,t}=Q^{-t}c_{M+1,N,t},
 \qquad
 Q^{-(N-t)}c_{M,N,t-1}=Q^{M-t+1}c_{M+1,N,t}.
\]
Consequently \eqref{eq:q-Weyl-two-contributions} equals
\[
 c_{M+1,N,t}
 \frac{Q^{-t}[M-t+1]_Q+Q^{M-t+1}[t]_Q}{[M+1]_Q}
 =c_{M+1,N,t}.
\]
The last equality is the symmetric quantum-integer identity
\[
 [a+b]_Q=Q^{-b}[a]_Q+Q^a[b]_Q
\]
with \(a=M-t+1\) and \(b=t\).  If \(M<N\), the new endpoint
\(t=M+1\) arises only from the second summand with \(s=M\), and its
coefficient is
\[
 Q^{-(N-M-1)}c_{M,N,M}=c_{M+1,N,M+1}.
\]
If \(N\leq M\), the upper endpoint remains \(N\).  Thus every boundary
term is accounted for, and the induction proves
\eqref{eq:q-Weyl-normal-ordering}.

It remains to prove the integrality assertion.  We use strong induction on
\(a\).  The case \(a=0\) is immediate.  Let \(a>0\), and assume that
\[
 \frac{Z^t}{[t]_Q!}\in D
 \qquad(0\leq t<a).
\]
Setting \(M=N=a\) in \eqref{eq:q-Weyl-normal-ordering} gives
\[
 A_Q^{(a)}B_Q^{(a)}
 =\sum_{t=0}^{a}
 Q^{-t(2a-2t)-\binom t2}
 B_Q^{(a-t)}\frac{Z^t}{[t]_Q!}A_Q^{(a-t)}.
\]
For every \(t<a\), the corresponding term belongs to \(D\): the middle
factor belongs to \(D\) by the induction hypothesis, the two divided-power
factors belong to \(D\) by assumption, and the displayed power of \(Q\) is
a unit of \(\mathcal A\).  The left-hand side also lies in \(D\).  After
subtracting the terms indexed by \(t<a\), the remaining term is
\[
 Q^{-\binom a2}\frac{Z^a}{[a]_Q!}\in D,
\]
because \(A_Q^{(0)}=B_Q^{(0)}=1\).  Finally,
\(Q^{\binom a2}\in\mathcal A^\times\), and multiplication by this unit
gives \(Z^a/[a]_Q!\in D\).  This completes the strong induction.
\end{proof}

A similar result without the integrality assertion is given in
\cite[Theorem~3.3(1)]{DiP}, with different notation
and normalization. The lemma isolates the algebraic operation used in both exceptional
sufficiency arguments.  In Sections~\ref{sec:rank-one-sufficiency} and
\ref{sec:quantum-dp-scaled-roots}, two integral real-root families form a
\(q\)-Weyl cell, and the highest contraction produces the missing affine
direction.

\subsection{Ordinary localizations and the main theorem}

If \(M\) is an \(\mathcal A\)-submodule of the common ambient
\(\mathcal K\)-algebra and \(B=S^{-1}\mathcal A\subseteq\mathcal K\), write
\(BM\) for the image of
\[
 B\otimes_{\mathcal A}M\longrightarrow U_q^{\mathrm{DJ}}(\mathfrak g),
 \qquad b\otimes m\longmapsto bm.
\]
Localizing the inclusion
\(M\hookrightarrow U_q^{\mathrm{DJ}}(\mathfrak g)\) gives
\[
 S^{-1}M\hookrightarrow
 S^{-1}U_q^{\mathrm{DJ}}(\mathfrak g)
 =U_q^{\mathrm{DJ}}(\mathfrak g),
\]
because localization is flat and every element of \(S\) acts invertibly on
the ambient \(\mathcal K\)-algebra.  Thus the displayed multiplication map
is injective, and we identify \(BM\) with
\(B\otimes_{\mathcal A}M\).
The following proposition shows that if the two localized integral
forms coincide, then their coefficient rings must also coincide.

\begin{proposition}
Let \(B_1,B_2\subseteq\mathcal K\) be ordinary localizations of
\(\mathcal A\). If
\[
B_1U_{\mathcal A}^{\mathrm{Dr}}(\mathfrak g)
=
B_2U_{\mathcal A}^{\mathrm{Lus}}(\mathfrak g)
\]
inside \(U_q^{\mathrm{DJ}}(\mathfrak g)\), then
\[
B_1=B_2.
\]

\end{proposition}

\begin{proof}
Let
\[
\varepsilon:
U_q^{\mathrm{DJ}}(\mathfrak g)\longrightarrow\mathcal K
\]
be the counit, and let \(1_U\) denote the identity element of
\(U_q^{\mathrm{DJ}}(\mathfrak g)\). Since both integral forms contain
\(\mathcal A1_U\) and the counit takes each of their defining
generators into \(\mathcal A\), we have
\[
\varepsilon\bigl(
U_{\mathcal A}^{\mathrm{Dr}}(\mathfrak g)
\bigr)
=
\mathcal A
=
\varepsilon\bigl(
U_{\mathcal A}^{\mathrm{Lus}}(\mathfrak g)
\bigr).
\]
Since \(\varepsilon\) is \(\mathcal K\)-linear, it follows that
\[
\varepsilon\bigl(
B_1U_{\mathcal A}^{\mathrm{Dr}}(\mathfrak g)
\bigr)
=
B_1,
\qquad
\varepsilon\bigl(
B_2U_{\mathcal A}^{\mathrm{Lus}}(\mathfrak g)
\bigr)
=
B_2.
\]
Applying \(\varepsilon\) to the assumed equality therefore gives
\[
B_1
=
\varepsilon\bigl(
B_1U_{\mathcal A}^{\mathrm{Dr}}(\mathfrak g)
\bigr)
=
\varepsilon\bigl(
B_2U_{\mathcal A}^{\mathrm{Lus}}(\mathfrak g)
\bigr)
=
B_2.
\]
\end{proof}

Thus allowing the two integral forms to be localized over different coefficient rings does not yield any additional
comparisons.
\begin{definition}[Sufficient ordinary localization and least sufficient ring]
\label[definition]{def:sufficient-localization}
Let \(S\subset\mathcal A\setminus\{0\}\) be a multiplicatively closed set
containing \(1\), and put \(B=S^{-1}\mathcal A\subset\mathcal K\).  We call
\(B\) a \emph{sufficient ordinary localization} for \(\mathfrak g\) if
\[
 B\UDr(\mathfrak g)
 =BU_{\mathcal A}^{\mathrm{Lus}}(\mathfrak g).
\]
If the poset of sufficient ordinary localizations, ordered by inclusion as
subrings of \(\mathcal K\), has a least element, we denote it by
\(B_{\min}(\mathfrak g)\).
\end{definition}

Set
\begin{equation}
 \label{eq:Bstar-Cstar}
\Bstar:=\mathcal A[\Phi_{8r}(q)^{-1}\mid r\geq1],
 \qquad
 \Cstar:=\mathcal A[\Phi_{4r}(q)^{-1}\mid r\geq1].
\end{equation}

\begin{theorem}[Uniform least-ring theorem]
\label{thm:uniform-least-ring}
For every \(\mathfrak g\in\mathcal T\), the least element in
Definition~\ref{def:sufficient-localization} exists, and
\[
 B_{\min}(\mathfrak g)=
 \begin{cases}
 \mathcal A,
   &\mathfrak g\notin\{A_{2m}^{(2)}\mid m\geq1\},\\
 \Cstar,
   &\mathfrak g=A_2^{(2)},\\
 \Bstar,
   &\mathfrak g=A_{2n}^{(2)},\ n\geq2.
 \end{cases}
\]
More precisely, for every ordinary localization
\(B=S^{-1}\mathcal A\subset\mathcal K\),
\[
 B\UDr(A_2^{(2)})
 =BU_{\mathcal A}^{\mathrm{Lus}}(A_2^{(2)})
 \quad\Longleftrightarrow\quad
 \Phi_{4r}(q)\in B^\times\quad(r\geq1),
\]
whereas, for \(n\geq2\),
\[
 B\UDr(A_{2n}^{(2)})
 =BU_{\mathcal A}^{\mathrm{Lus}}(A_{2n}^{(2)})
 \quad\Longleftrightarrow\quad
 \Phi_{8r}(q)\in B^\times\quad(r\geq1).
\]
\end{theorem}
The following theorem shows that adjoining \(D^{\pm1}\) does not change
the ordinary localizations over which the two integral forms coincide.
\begin{theorem}
\label{thm:degree-operator}
Let \(\widehat U_q^{\mathrm{DJ}}(\mathfrak g)\) be obtained from
\(U_q^{\mathrm{DJ}}(\mathfrak g)\) by adjoining the degree element
\(D^{\pm1}=q^{\pm d}\), where
\[
 \begin{aligned}
  DE_iD^{-1}&=q^{\delta_{i0}}E_i,
  &
  DF_iD^{-1}&=q^{-\delta_{i0}}F_i,\\
  Dx_{i,r}^{\pm}D^{-1}&=q^r x_{i,r}^{\pm},
  &
  DK_iD^{-1}&=K_i,
  \qquad
  DK_\delta D^{-1}=K_\delta .
 \end{aligned}
\]
For \(\bullet\in\{\mathrm{Dr},\mathrm{Lus}\}\), let
\[
 \widehat U_{\mathcal A}^{\bullet}(\mathfrak g)
 :=
 \bigl\langle
 U_{\mathcal A}^{\bullet}(\mathfrak g),D^{\pm1}
 \bigr\rangle_{\mathcal A}.
\]
Then, for every ordinary localization
\(B=S^{-1}\mathcal A\subset\mathcal K\),
\[
 B\widehat U_{\mathcal A}^{\mathrm{Dr}}(\mathfrak g)
 =
 B\widehat U_{\mathcal A}^{\mathrm{Lus}}(\mathfrak g)
 \quad\Longleftrightarrow\quad
 BU_{\mathcal A}^{\mathrm{Dr}}(\mathfrak g)
 =
 BU_{\mathcal A}^{\mathrm{Lus}}(\mathfrak g).
\]
Consequently, adjoining \(D^{\pm1}\) does not change any conclusion of
\cref{thm:uniform-least-ring}; in particular, the least sufficient rings
remain \(\mathcal A\), \(\Cstar\), and \(\Bstar\) in the respective cases
listed there.
\end{theorem}

\begin{proof}
Let \(\tau=\operatorname{Ad}(D)\).  The defining relations give
\[
 \tau\bigl(E_i^{(a)}\bigr)
 =q^{a\delta_{i0}}E_i^{(a)},
 \qquad
 \tau\bigl(F_i^{(a)}\bigr)
 =q^{-a\delta_{i0}}F_i^{(a)},
 \qquad
 \tau\bigl((x_{i,r}^{\pm})^{(a)}\bigr)
 =q^{ar}(x_{i,r}^{\pm})^{(a)},
\]
and \(\tau\) fixes all toral generators.  Since the displayed powers of
\(q\) belong to \(\mathcal A^\times\), both integral forms are preserved
by \(\tau\) and \(\tau^{-1}\).  Hence
\[
 \widehat U_{\mathcal A}^{\bullet}(\mathfrak g)
 =
 \bigoplus_{m\in\mathbb Z}
 U_{\mathcal A}^{\bullet}(\mathfrak g)D^m,
 \qquad
 \bullet\in\{\mathrm{Dr},\mathrm{Lus}\},
\]
and therefore
\[
 B\widehat U_{\mathcal A}^{\bullet}(\mathfrak g)
 =
 \bigoplus_{m\in\mathbb Z}
 BU_{\mathcal A}^{\bullet}(\mathfrak g)D^m.
\]
Equality of the degree-zero components proves the forward implication;
the reverse implication follows by extending an equality in degree zero
to every \(D\)-degree.
\end{proof}
The remainder of the paper is devoted entirely to the proof of
Theorem~\ref{thm:uniform-least-ring}.

\part{The non-\texorpdfstring{$A_{2n}^{(2)}$}{A2n(2)} twisted types}
\label{part:nonexceptional}

\section{Root data and the equality theorem in the nonexceptional types}
\label{sec:nonexceptional-data}

Throughout Part~\ref{part:nonexceptional}, fix
\[
 \mathfrak g\in
 \left\{
 A_{2n-1}^{(2)}\ (n\geq3),\quad
 D_{n+1}^{(2)}\ (n\geq2),\quad
 E_6^{(2)},\quad D_4^{(3)}
 \right\}.
\]
We use the following numbering of the nodes and the parameters
\eqref{eq:local-quantum-parameters}. Write \(\delta=\alpha_0+\theta\).  The following table records the finite root
data and fixes the finite node \(i^*\) and the Weyl-group element \(w\) used
in Section~\ref{sec:finite-Weyl-endpoint}.
\begin{equation}
\label{eq:nonexceptional-type-table}
\begin{array}{c|c|c|c}
 \mathfrak g&(d_0;d_1,\ldots,d_n)&i^*&w\\
 \hline
 A_{2n-1}^{(2)}
 &(1;2,1,\ldots,1)
 &n
 &(s_{n-1}\cdots s_2)s_1(s_2\cdots s_{n-1})\\
 D_{n+1}^{(2)}
 &(1;1,2,\ldots,2)
 &1
 &s_n\cdots s_2\\
 E_6^{(2)}
 &(1;1,1,2,2)
 &1
 &s_1s_2s_3s_4s_2s_3s_2\\
 D_4^{(3)}
 &(1;1,3)
 &1
 &s_1s_2
\end{array}
\end{equation}

The corresponding affine Dynkin diagrams are displayed in
\cref{fig:nonexceptional-dynkin}.

\begin{figure}[htbp]
\centering
\renewcommand{\arraystretch}{2.15}
\begin{tabular}{c@{\qquad}c}
\(A_5^{(2)}\ (n=3)\) &
\begin{tikzpicture}[x=1cm,y=1cm,baseline=(current bounding box.center)]
  \DynkinNode{a51}{(0,0)}{1}
  \DynkinNode{a52}{(1.05,0)}{2}
  \DynkinNode{a53}{(2.10,0)}{3}
  \DynkinNodeAbove{a50}{(1.05,.82)}{0}
  \DynkinDouble{a51}{a52}\DynkinArrow{a51}{a52}
  \draw[dynkin bond] (a52)--(a53) (a52)--(a50);
\end{tikzpicture}
\\
\(A_{2n-1}^{(2)}\ (n\geq4)\) &
\begin{tikzpicture}[x=1cm,y=1cm,baseline=(current bounding box.center)]
  \DynkinNode{a1}{(0,0)}{1}
  \DynkinNode{a2}{(1.05,0)}{2}
  \DynkinNode{am}{(3.15,0)}{n-1}
  \DynkinNode{an}{(4.20,0)}{n}
  \DynkinNodeAbove{a0}{(3.15,.82)}{0}
  \DynkinDouble{a1}{a2}\DynkinArrow{a1}{a2}
  \draw[dynkin bond] (a2)--(1.65,0);
  \node at (2.10,0) {\(\cdots\)};
  \draw[dynkin bond] (2.55,0)--(am);
  \draw[dynkin bond] (am)--(an) (am)--(a0);
\end{tikzpicture}
\\
\(D_3^{(2)}\ (n=2)\) &
\begin{tikzpicture}[x=1cm,y=1cm,baseline=(current bounding box.center)]
  \DynkinNode{d31}{(0,0)}{1}
  \DynkinNode{d32}{(1.05,0)}{2}
  \DynkinNode{d30}{(2.10,0)}{0}
  \DynkinDouble{d31}{d32}\DynkinArrow{d32}{d31}
  \DynkinDouble{d32}{d30}\DynkinArrow{d32}{d30}
\end{tikzpicture}
\\
\(D_{n+1}^{(2)}\ (n\geq3)\) &
\begin{tikzpicture}[x=1cm,y=1cm,baseline=(current bounding box.center)]
  \DynkinNode{d1}{(0,0)}{1}
  \DynkinNode{d2}{(1.05,0)}{2}
  \DynkinNode{dn}{(3.15,0)}{n}
  \DynkinNode{d0}{(4.20,0)}{0}
  \DynkinDouble{d1}{d2}\DynkinArrow{d2}{d1}
  \draw[dynkin bond] (d2)--(1.65,0);
  \node at (2.10,0) {\(\cdots\)};
  \draw[dynkin bond] (2.55,0)--(dn);
  \DynkinDouble{dn}{d0}\DynkinArrow{dn}{d0}
\end{tikzpicture}
\\
\(E_6^{(2)}\) &
\begin{tikzpicture}[x=1cm,y=1cm,baseline=(current bounding box.center)]
  \DynkinNode{e0}{(0,0)}{0}
  \DynkinNode{e1}{(1.05,0)}{1}
  \DynkinNode{e2}{(2.10,0)}{2}
  \DynkinNode{e3}{(3.15,0)}{3}
  \DynkinNode{e4}{(4.20,0)}{4}
  \draw[dynkin bond] (e0)--(e1)--(e2);
  \DynkinDouble{e2}{e3}\DynkinArrow{e3}{e2}
  \draw[dynkin bond] (e3)--(e4);
\end{tikzpicture}
\\
\(D_4^{(3)}\) &
\begin{tikzpicture}[x=1cm,y=1cm,baseline=(current bounding box.center)]
  \DynkinNode{g0}{(0,0)}{0}
  \DynkinNode{g1}{(1.05,0)}{1}
  \DynkinNode{g2}{(2.35,0)}{2}
  \draw[dynkin bond] (g0)--(g1);
  \DynkinTriple{g1}{g2}\DynkinArrow{g2}{g1}
\end{tikzpicture}
\end{tabular}
\caption{The non-\(A_{2n}^{(2)}\) twisted affine Dynkin diagrams.  The low-rank isomorphism
\(A_3^{(2)}\cong D_3^{(2)}\) is represented by the \(D_3^{(2)}\) row.}
\label{fig:nonexceptional-dynkin}
\end{figure}
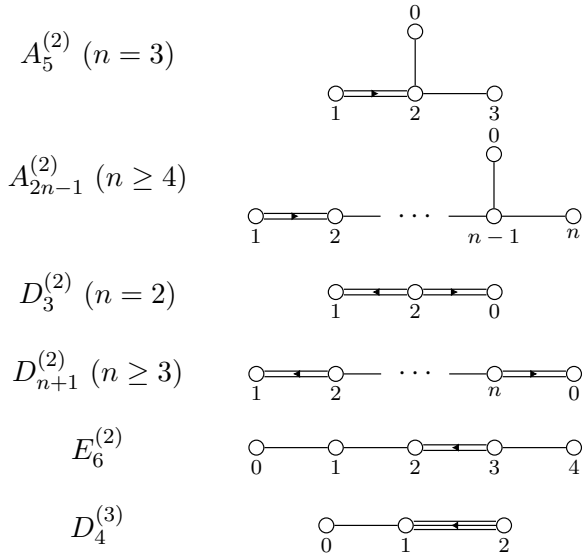

\begin{theorem}[equality in the nonexceptional types]
\label{thm:nonexceptional-equality}
For each of the four families above, after fixing the comparison isomorphism
\eqref{eq:comparison-isomorphism},
\[
 \UDr(\mathfrak g)
 =U_{\mathcal A}^{\mathrm{Lus}}(\mathfrak g)
 \qquad\text{inside }U_q^{\mathrm{DJ}}(\mathfrak g).
\]
\end{theorem}
\noindent
Sections~\ref{sec:real-current-start}
and~\ref{sec:finite-Weyl-endpoint}
are devoted to the proof of
Theorem~\ref{thm:nonexceptional-equality}.

\section{Braid stability and the
\texorpdfstring{$\beta_0$}{beta0}-root divided powers}
\label{sec:real-current-start}
Use the translations in \eqref{eq:fundamental-translation-definition};
for these four types, \(\nu_i=d_i\).
Writing \((\lambda\mid\alpha):=\langle\lambda,\alpha\rangle\)
for the canonical pairing with the finite root lattice, we have
\begin{equation}\label{eq:translation-lattice-pairing}
 (\lambda_i\mid\alpha_j)=d_i\delta_{ij}.
\end{equation}
The sign and zero-mode conventions are
\eqref{eq:current-sign-convention} and
\eqref{eq:dr-zero-mode-convention}.

The current formulas in the proof of
Lemma~\ref{lem:forward-inclusion}, together with
\eqref{eq:Omega-commutes-braid}, give
\begin{equation}
\label{eq:Omega-current-core}
\Omega\bigl((x_{i,d_ir}^{\pm})^{(a)}\bigr)
=(x_{i,-d_ir}^{\mp})^{(a)},
\qquad
\Omega\bigl(U_{\mathcal A}^{\mathrm{Dr}}(\mathfrak g)\bigr)
=U_{\mathcal A}^{\mathrm{Dr}}(\mathfrak g).
\end{equation}

\begin{lemma}
\label{lem:translation-stability}
For every \(h\in I_0\),
\[
T_{\lambda_h}
\bigl(U_{\mathcal A}^{\mathrm{Dr}}(\mathfrak g)\bigr)
=
U_{\mathcal A}^{\mathrm{Dr}}(\mathfrak g).
\]
\end{lemma}

\begin{proof}
Under the comparison isomorphism \(\Psi_{\mathfrak g}\), for
\(\varepsilon\in\{\pm1\}\),
\begin{equation}
\label{eq:translation-action}
\begin{aligned}
T_{\lambda_h}^{\varepsilon}(K_\delta)
  & =K_\delta,\\
T_{\lambda_h}^{\varepsilon}(K_j)
  & =K_jK_\delta^{-\varepsilon\delta_{hj}d_h},\\
T_{\lambda_h}^{\varepsilon}
 \bigl((x_{j,s}^{+})^{(a)}\bigr)
  & =
o(h)^{a\delta_{hj}}
 \bigl(x_{j,s-\varepsilon\delta_{hj}d_h}^{+}\bigr)^{(a)},\\
T_{\lambda_h}^{\varepsilon}
 \bigl((x_{j,s}^{-})^{(a)}\bigr)
  & =
o(h)^{a\delta_{hj}}
 \bigl(x_{j,s+\varepsilon\delta_{hj}d_h}^{-}\bigr)^{(a)}.
\end{aligned}
\end{equation}
If \(s\in d_j\mathbb Z\), then every mode occurring on the
right-hand side is again in \(d_j\mathbb Z\).  Hence both
\(T_{\lambda_h}\) and \(T_{\lambda_h}^{-1}\) preserve the generators
in \eqref{eq:bare-real-current-core}, proving the equality.
\end{proof}

For \(j\in I_0\), set
\[
\mathcal A_j:=\mathbb Z[q_j,q_j^{-1}],
\qquad
\beta_j:=d_j\delta-\alpha_j,
\qquad
K_{\beta_j}:=K_\delta^{d_j}K_j^{-1}.
\]
By \eqref{eq:translation-real-root-vectors} and
\eqref{eq:real-root-vector-definition},
\begin{equation}
\label{eq:beta-j-root-vectors}
E_{\beta_j}=T_{\lambda_j}T_j^{-1}(E_j),
\qquad
F_{\beta_j}=\Omega(E_{\beta_j}).
\end{equation}
There is an \(A_1^{(1)}\)-vertex homomorphism \cite{D1}\cite{JZ}
\[
\phi_j^{\mathrm{DJ}}:
U_{q_j}^{\mathrm{DJ}}(A_1^{(1)})
\longrightarrow
U_q^{\mathrm{DJ}}(\mathfrak g)
\]
sending the two source simple pairs
\((E_1,F_1)\) and \((E_0,F_0)\) to
\((E_j,F_j)\) and \((E_{\beta_j},F_{\beta_j})\) respectively, and the corresponding toral elements to
\(K_j\) and \(K_{\beta_j}\).

Choose the sign normalization on the source ordinary vertex so that
its finite current has sign \(o(j)\).  Write \(\delta^{(1)}\) for the
null root of the source \(A_1^{(1)}\). The homomorphism
\(\phi_j^{\mathrm{DJ}}\) satisfies
\begin{equation}
\label{eq:ordinary-vertex-intertwining}
\begin{gathered}
\phi_j^{\mathrm{DJ}}(K_{\delta^{(1)}})
=
K_\delta^{d_j},
\qquad
\phi_j^{\mathrm{DJ}}(K_1)
=
K_j,
\\
\phi_j^{\mathrm{DJ}}(x_{1,r}^{\pm})
=
x_{j,d_jr}^{\pm}
\qquad(r\in\mathbb Z),
\qquad
T_j\phi_j^{\mathrm{DJ}}
=
\phi_j^{\mathrm{DJ}}T_1.
\end{gathered}
\end{equation}
Here \(x_{1,r}^{\pm}\) and \(T_1\) denote, respectively, the real
currents and the Beck braid operator in the source
\(U_{q_j}^{\mathrm{DJ}}(A_1^{(1)})\).  

Let \(V_{j,\mathcal A}\) denote the image of
\[
\mathcal A\otimes_{\mathcal A_j}
U_{\mathcal A_j}^{\mathrm{Lus}}(A_1^{(1)})
\xrightarrow{\;1\otimes\phi_j^{\mathrm{DJ}}\;}
U_q^{\mathrm{DJ}}(\mathfrak g).
\]

\begin{lemma}
\label{lem:integral-ordinary-vertex}
For every \(j\in I_0\),
\[
V_{j,\mathcal A}
\subseteq
U_{\mathcal A}^{\mathrm{Dr}}(\mathfrak g),
\]
and \(V_{j,\mathcal A}\) contains
\[
(x_{j,d_jr}^{\pm})^{(a)}
\qquad
(r\in\mathbb Z,\ a\geq0).
\]
\end{lemma}

\begin{proof}
For \(r\in\mathbb Z\) and \(a\geq0\), the source
\(A_1^{(1)}\) current formulas give
\[
(x_{1,r}^{+})^{(a)}
=
o(j)^{ra}T_{\lambda_1}^{-r}(E_1^{(a)}),
\qquad
(x_{1,r}^{-})^{(a)}
=
o(j)^{ra}T_{\lambda_1}^{r}(F_1^{(a)}).
\]
The braid-integrality statement \eqref{eq:braid-integral-stability},
applied over \(\mathcal A_j\), shows that Beck's braid operators and
their inverses preserve
\(U_{\mathcal A_j}^{\mathrm{Lus}}(A_1^{(1)})\). Hence
\[
(x_{1,r}^{\pm})^{(a)}
\in
U_{\mathcal A_j}^{\mathrm{Lus}}(A_1^{(1)}).
\]
Equation~\eqref{eq:ordinary-vertex-intertwining} now gives
\[
\phi_j^{\mathrm{DJ}}
\bigl((x_{1,r}^{\pm})^{(a)}\bigr)
=
(x_{j,d_jr}^{\pm})^{(a)}
\in V_{j,\mathcal A}.
\]
It remains to prove
\(V_{j,\mathcal A}\subseteq
U_{\mathcal A}^{\mathrm{Dr}}(\mathfrak g)\).
For the source finite simple generators, one has
\[
\phi_j^{\mathrm{DJ}}(E_1^{(a)})
=
E_j^{(a)}
=
(x_{j,0}^{+})^{(a)},
\qquad
\phi_j^{\mathrm{DJ}}(F_1^{(a)})
=
F_j^{(a)}
=
(x_{j,0}^{-})^{(a)}.
\]
Hence their images belong to
\(U_{\mathcal A}^{\mathrm{Dr}}(\mathfrak g)\).
For the other positive simple generator,
\eqref{eq:translation-action},
\eqref{eq:beta-j-root-vectors}, and
\(T_j^{-1}(E_j)=-K_j^{-1}F_j\) give
\[
\begin{aligned}
E_{\beta_j}
&=
T_{\lambda_j}T_j^{-1}(E_j)\\
&=
-T_{\lambda_j}(K_j^{-1})
 T_{\lambda_j}(F_j)\\
&=
-o(j)K_{\beta_j}x_{j,d_j}^{-}.
\end{aligned}
\]
Moreover,
\[
K_{\beta_j}x_{j,d_j}^{-}
=
q_j^2x_{j,d_j}^{-}K_{\beta_j}.
\]
Consequently,
\[
E_{\beta_j}^{(a)}
=
(-o(j))^a q_j^{-a(a-1)}
K_{\beta_j}^{a}(x_{j,d_j}^{-})^{(a)}
\in
U_{\mathcal A}^{\mathrm{Dr}}(\mathfrak g).
\]
Equation \eqref{eq:Omega-current-core} gives
\[
F_{\beta_j}^{(a)}
=
\Omega(E_{\beta_j}^{(a)})
\in
U_{\mathcal A}^{\mathrm{Dr}}(\mathfrak g).
\]
Finally, the source toral generators are sent to
\(K_j^{\pm1}\) and \(K_{\beta_j}^{\pm1}\), which already belong to
\(U_{\mathcal A}^{\mathrm{Dr}}(\mathfrak g)\).
Thus the images of all defining generators of the source Lusztig
integral form lie in
\(U_{\mathcal A}^{\mathrm{Dr}}(\mathfrak g)\), proving the inclusion.
\end{proof}

Let
\[
U_{\mathcal A,0}^{\mathrm{Lus}}
:=
\mathcal A\Bigl\langle
E_j^{(a)},F_j^{(a)},K_j^{\pm1}
\ \Bigm|\
j\in I_0,\ a\geq0
\Bigr\rangle
\]
be the finite Lusztig integral form. By \eqref{eq:current-divided-powers-braid-formula} with \(k=0\),
\[
E_j^{(a)}=(x_{j,0}^{+})^{(a)},
\qquad
F_j^{(a)}=(x_{j,0}^{-})^{(a)}
\qquad(j\in I_0,\ a\geq0).
\]
Together with the defining toral generators \(K_j^{\pm1}\), this gives
\[
U_{\mathcal A,0}^{\mathrm{Lus}}
\subseteq
U_{\mathcal A}^{\mathrm{Dr}}(\mathfrak g).
\]

\begin{proposition}[Stability under finite braid operators]
\label{prop:finite-braid-stability}
For every \(i\in I_0\),
\[
T_i
\bigl(U_{\mathcal A}^{\mathrm{Dr}}(\mathfrak g)\bigr)
=
U_{\mathcal A}^{\mathrm{Dr}}(\mathfrak g).
\]
\end{proposition}

\begin{proof}
We prove that \(T_i^\varepsilon\) preserves
\(U_{\mathcal A}^{\mathrm{Dr}}(\mathfrak g)\) for
\(\varepsilon\in\{\pm1\}\).
First consider the currents of color \(i\).  Equation
\eqref{eq:ordinary-vertex-intertwining} and its inverse give
\[
T_i^\varepsilon\phi_i^{\mathrm{DJ}}
=
\phi_i^{\mathrm{DJ}}T_1^\varepsilon
\qquad(\varepsilon\in\{\pm1\}),
\]
where \(T_1\) denotes the corresponding braid operator at the finite node of the
source \(U_{q_i}^{\mathrm{DJ}}(A_1^{(1)})\).
Since
\[
T_1^\varepsilon
\bigl(U_{\mathcal A_i}^{\mathrm{Lus}}(A_1^{(1)})\bigr)
=
U_{\mathcal A_i}^{\mathrm{Lus}}(A_1^{(1)}),
\]
it follows that
\[
T_i^\varepsilon(V_{i,\mathcal A})
=
V_{i,\mathcal A}.
\]
Lemma~\ref{lem:integral-ordinary-vertex} therefore handles all
divided real currents of color \(i\).

Now let \(j\neq i\).  By
\eqref{eq:translation-lattice-pairing} and
\eqref{eq:fundamental-translation-definition},
\[
t_{\lambda_j}(\alpha_i)=\alpha_i,
\qquad
s_it_{\lambda_j}=t_{\lambda_j}s_i.
\]
Since
\[
t_{\lambda_j}^{\pm1}(\alpha_i)=\alpha_i>0,
\]
the corresponding factorizations are length additive.  Hence the
standard braid lifts commute, and therefore
\begin{equation}
\label{eq:finite-translation-commutation}
T_i^\varepsilon T_{\lambda_j}^{r}
=
T_{\lambda_j}^{r}T_i^\varepsilon
\qquad
(\varepsilon\in\{\pm1\},\ r\in\mathbb Z).
\end{equation}
Applying \cite[Section~41.1.2]{L} to the finite Cartan datum
\((a_{jk})_{j,k\in I_0}\), we obtain
\[
T_i^\varepsilon
\bigl(U_{\mathcal A,0}^{\mathrm{Lus}}\bigr)
=
U_{\mathcal A,0}^{\mathrm{Lus}}
\qquad
(\varepsilon\in\{\pm1\}).
\]
Consequently,
\[
T_i^\varepsilon(E_j^{(a)})
\in
U_{\mathcal A,0}^{\mathrm{Lus}}
\subseteq
U_{\mathcal A}^{\mathrm{Dr}}(\mathfrak g).
\]
Using \eqref{eq:current-divided-powers-braid-formula},
\eqref{eq:finite-translation-commutation}, and
Lemma~\ref{lem:translation-stability}, we obtain
\[
\begin{aligned}
T_i^\varepsilon
 \bigl((x_{j,d_jr}^{+})^{(a)}\bigr)
&=
o(j)^{ra}
T_i^\varepsilon T_{\lambda_j}^{-r}(E_j^{(a)})\\
&=
o(j)^{ra}
T_{\lambda_j}^{-r}
 \bigl(T_i^\varepsilon(E_j^{(a)})\bigr)\\
&\in
U_{\mathcal A}^{\mathrm{Dr}}(\mathfrak g).
\end{aligned}
\]
The corresponding assertion for the negative currents follows from
\[
\Omega T_i^\varepsilon
=
T_i^\varepsilon\Omega
\]
and \eqref{eq:Omega-current-core}.

Finally, \(T_i^\varepsilon\) fixes \(K_\delta\) and sends every
\(K_j^{\pm1}\) to a Laurent monomial in the finite toral generators.
Thus both \(T_i\) and \(T_i^{-1}\) preserve
\(U_{\mathcal A}^{\mathrm{Dr}}(\mathfrak g)\), proving the equality.
\end{proof}

Fix the node \(i^\ast\) and the element \(w\) in
\eqref{eq:nonexceptional-type-table}.  The table gives
\begin{equation}
\label{eq:chosen-node-length}
d_{i^\ast}=d_0=1.
\end{equation}
Set
\[
\beta_0:=\delta-\alpha_{i^\ast}=\beta_{i^\ast}.
\]
The vertex homomorphism \(\phi_{i^\ast}^{\mathrm{DJ}}\) sends the
affine simple positive generator \(E_0\) of
\(U_{q_{i^\ast}}^{\mathrm{DJ}}(A_1^{(1)})\) to
\(E_{\beta_{i^\ast}}=E_{\beta_0}\).  Consequently,
\[
(1\otimes\phi_{i^\ast}^{\mathrm{DJ}})
(1\otimes E_0^{(a)})
=
E_{\beta_0}^{(a)}.
\]
Lemma~\ref{lem:integral-ordinary-vertex} therefore gives
\begin{equation}
\label{eq:beta-zero-integrality}
E_{\beta_0}^{(a)}
\in
V_{i^\ast,\mathcal A}
\subseteq
U_{\mathcal A}^{\mathrm{Dr}}(\mathfrak g)
\qquad(a\geq0).
\end{equation}

\section{Recovery of the affine node}
\label{sec:finite-Weyl-endpoint}

\begin{lemma}
\label{lem:explicit-Weyl-elements}
For each row of \eqref{eq:nonexceptional-type-table}, the displayed
Weyl-group element satisfies
\[
 w(\alpha_{i^*})=\theta.
\]
\end{lemma}

\begin{proof}
We verify the four types separately.  In type \(A_{2n-1}^{(2)}\), use the
coordinate realization
\[
 \alpha_1=2e_1,
 \qquad
 \alpha_r=e_r-e_{r-1}
 \quad(2\leq r\leq n).
\]
Here \(\theta=e_{n-1}+e_n\).  Applying the word in
\eqref{eq:nonexceptional-type-table} from right to left to
\(\alpha_n=e_n-e_{n-1}\) gives
\[
 e_n-e_{n-1}
 \xrightarrow{s_{n-1}}e_n-e_{n-2}
 \xrightarrow{}\cdots
 \xrightarrow{s_2}e_n-e_1
 \xrightarrow{s_1}e_n+e_1
 \xrightarrow{s_2}\cdots
 \xrightarrow{s_{n-1}}e_n+e_{n-1}.
\]
Thus \(w(\alpha_n)=\theta\).

In type \(D_{n+1}^{(2)}\), take
\[
 \alpha_1=e_1,
 \qquad
 \alpha_r=e_r-e_{r-1}
 \quad(2\leq r\leq n).
\]
Now \(\theta=e_n\), and
\[
 e_1\xrightarrow{s_2}e_2
 \xrightarrow{s_3}\cdots
 \xrightarrow{s_n}e_n.
\]
Therefore \(w(\alpha_1)=\theta\).

For type \(E_6^{(2)}\), the finite Cartan matrix and an \(F_4\) coordinate
realization are
\[
 \begin{pmatrix}
 2&-1&0&0\\
 -1&2&-2&0\\
 0&-1&2&-1\\
 0&0&-1&2
 \end{pmatrix},
 \qquad
 \alpha_1=\tfrac12(e_1-e_2-e_3-e_4),
 \quad
 \alpha_2=e_4,
 \quad
 \alpha_3=e_3-e_4,
 \quad
 \alpha_4=e_2-e_3.
\]
Here \(\theta=e_1=(2,3,2,1)\) in simple-root coordinates.  Beginning with
\(\alpha_1\) and applying, from right to left,
\(s_2,s_3,s_2,s_4,s_3,s_2,s_1\), the successive simple-root coordinates
are
written in the compact form
\(a_1a_2a_3a_4=a_1\alpha_1+\cdots+a_4\alpha_4\):
\[
\begin{array}{c|cccccccc}
 \text{step}&0&1&2&3&4&5&6&7\\
 \hline
 \text{root}&1000&1100&1110&1210&1211&1221&1321&2321.
\end{array}
\]
Thus \(w(\alpha_1)=\theta\).

Finally, in type \(D_4^{(3)}\), the finite \(G_2\) Cartan matrix is
\[
 \begin{pmatrix}2&-3\\-1&2\end{pmatrix},
\]
with \(\alpha_1\) short and \(\alpha_2\) long.  Direct calculation gives
\[
 \alpha_1\xrightarrow{s_2}\alpha_1+\alpha_2
 \xrightarrow{s_1}2\alpha_1+\alpha_2=\theta.
\]
Again \(w(\alpha_1)=\theta\).
\end{proof}

\begin{lemma}
\label{lem:finite-endpoint-lift}
Fix the node \(i=i^*\) and the Weyl-group element \(w\) from
\eqref{eq:nonexceptional-type-table}, and set
\[
 y=wt_{\lambda_i},
 \qquad
 \widehat w=ys_i,
 \qquad
 \widehat w_0=t_{\lambda_i}s_i.
\]
Then
\begin{align}
 \ell(wt_{\lambda_i})
 &=\ell(w)+\ell(t_{\lambda_i}),
 \label{eq:endpoint-length-additivity}\\
 \ell(\widehat w)
 &=\ell(w)+\ell(t_{\lambda_i})-1,
 \label{eq:endpoint-right-descent}\\
 T_{\widehat w}
 &=T_wT_{\lambda_i}T_i^{-1},
 \qquad
 \widehat w(\alpha_i)=\alpha_0,
 \label{eq:endpoint-standard-lift}\\
 T_{\widehat w_0}
 &=T_{\lambda_i}T_i^{-1}.
 \label{eq:beta0-standard-lift}
\end{align}
In particular,
\begin{equation}
 \label{eq:z-endpoint-root-vector}
 T_w(E_{\beta_0})
 =T_{\widehat w}(E_i)
 \in (U_q^+)_{\alpha_0}.
\end{equation}
\end{lemma}

\begin{proof}
By Lemma~\ref{lem:explicit-Weyl-elements}, \(w(\alpha_i)=\theta\).  Choose a
reduced expression
\(w=s_{j_L}\cdots s_{j_1}\), and put
\(w_r=s_{j_r}\cdots s_{j_1}\), with \(w_0=1\).  For \(r\geq1\),
\[
 \gamma_r=w_{r-1}^{-1}(\alpha_{j_r})
\]
is a positive finite root.  If
\(\gamma_r=\sum_hc_h\alpha_h\), then
\((\lambda_i\mid\gamma_r)=d_ic_i\geq0\), and hence
\[
 (w_{r-1}t_{\lambda_i})^{-1}(\alpha_{j_r})
 =t_{-\lambda_i}(\gamma_r)
 =\gamma_r+(\lambda_i\mid\gamma_r)\delta>0.
\]
Consequently,
\[
 \ell(w_rt_{\lambda_i})
 =\ell(w_{r-1}t_{\lambda_i})+1
 \qquad(1\leq r\leq L).
\]
Induction from \(w_0=1\) gives
\eqref{eq:endpoint-length-additivity}.

Since \(d_i=1\) and \(w(\alpha_i)=\theta\),
\[
 y(\alpha_i)
 =w(\alpha_i-\delta)
 =\theta-\delta
 =-\alpha_0<0.
\]
Thus
\[
\ell(ys_i)=\ell(y)-1.
\] So
\(y=\widehat ws_i\) and \(\ell(y)=\ell(\widehat w)+1\).  This proves
\eqref{eq:endpoint-right-descent}.  The multiplication rule for standard
braid lifts gives
\[
 T_wT_{\lambda_i}=T_y=T_{\widehat w}T_i,
\]
and hence the braid identity in \eqref{eq:endpoint-standard-lift}.  Moreover,
\[
 \widehat w(\alpha_i)
 =wt_{\lambda_i}(-\alpha_i)
 =w(\delta-\alpha_i)
 =\delta-\theta
 =\alpha_0.
\]

Similarly,
\[
 t_{\lambda_i}(\alpha_i)
 =\alpha_i-\delta
 =-\beta_0<0.
\]
Hence \(t_{\lambda_i}=\widehat w_0s_i\) is a length-additive
factorization, which proves \eqref{eq:beta0-standard-lift}.

Finally, \(\widehat w(\alpha_i)=\alpha_0\), so
Lemma~\ref{lem:simple-to-simple} gives
\[
 T_{\widehat w}(E_i)=E_0\in(U_q^+)_{\alpha_0}.
\]
By
\eqref{eq:beta0-standard-lift},
\(E_{\beta_0}=T_{\widehat w_0}(E_i)\), and hence
\[
 T_w(E_{\beta_0})
 =T_wT_{\lambda_i}T_i^{-1}(E_i)
 =T_{\widehat w}(E_i).
\]
This proves \eqref{eq:z-endpoint-root-vector}.
\end{proof}

\begin{proposition}
\label{prop:endpoint-Laurent-unit}
For every \(a\geq0\), one has
\begin{equation}
\label{eq:endpoint-unit}
 T_w(E_{\beta_0}^{(a)})
 =
 E_0^{(a)}
 \in\UDr.
\end{equation}
\end{proposition}

\begin{proof}
By Lemma~\ref{lem:finite-endpoint-lift},
\[
T_w(E_{\beta_0})=T_{\widehat w}(E_{i^\ast}),
\qquad
\widehat w(\alpha_{i^\ast})=\alpha_0.
\]
Lemma~\ref{lem:simple-to-simple} therefore gives
\begin{equation}
\label{eq:endpoint-scalar}
T_w(E_{\beta_0})=T_{\widehat w}(E_{i^\ast})=E_0.
\end{equation}

By \eqref{eq:beta-zero-integrality},
\(E_{\beta_0}^{(a)}\in\UDr\) for every \(a\geq0\).
Applying Proposition~\ref{prop:finite-braid-stability}
along a reduced expression for \(w\), we obtain
\begin{equation}
\label{eq:transported-divided-power-integrality}
T_w(E_{\beta_0}^{(a)})\in\UDr
\qquad(a\geq0).
\end{equation}
Since \eqref{eq:chosen-node-length} gives
\(q_{i^\ast}=q_0=q\), it follows that
\[
T_w(E_{\beta_0}^{(a)})
=
\frac{T_w(E_{\beta_0})^a}{[a]_{q_0}!}
=
E_0^{(a)}
\in\UDr
\qquad(a\geq0).
\]
\end{proof}

\begin{proof}[Proof of Theorem~\ref{thm:nonexceptional-equality}]
Proposition~\ref{prop:endpoint-Laurent-unit} gives
\(E_0^{(a)}\in \UDr\) for every \(a\geq0\).  Applying
 \(\Omega\) gives
\(F_0^{(a)}\in \UDr\).  The divided powers at every finite
simple node, as well as the elements \(K_j^{\pm1}\), already belong to
\(\UDr\) by the zero-current modes and its definition.
The identity \eqref{eq:K0-recovery} and its inverse expression
\[
 K_0^{-1}=K_\delta^{-1}\prod_{i\in I_0}K_i^{r_i}
\]
show that \(K_0^{\pm1}\in \UDr\).  Hence
\[
 U_{\mathcal A}^{\mathrm{Lus}}(\mathfrak g)
 \subseteq \UDr(\mathfrak g).
\]
The reverse inclusion is by Lemma~\ref{lem:forward-inclusion}.  Therefore the two integral forms are
equal.
\end{proof}

\part{The twisted type \texorpdfstring{$A_{2n}^{(2)}$}{A2n(2)}
and cyclotomic minimality}
\label{part:a2n}

In type \(A_{2n}^{(2)}\) we need to use a different method to reach the affine simple root \(\alpha_0\) because \(\alpha_0\) is strictly longer than every finite simple root.  As 
Weyl-group transport preserves root length, the method of
Part~\ref{part:nonexceptional} does not work.

\section{Root data and real currents in type
\texorpdfstring{$A_{2n}^{(2)}$}{A2n(2)}}
\label{sec:a2n-root-data}

From this section onward, fix $n\geq 1$.  When $n\geq 2$, the finite part
is of type $B_n$, and
\begin{equation}
 (d_0,d_1,\ldots,d_n)
 =(4,2,\ldots,2,1),
 \qquad
 \delta=\alpha_0+2\alpha_1+\cdots+2\alpha_n,
 \steptag{5.1}\label{eq:a2n-symmetrizers}
\end{equation}
and the nonzero off-diagonal Cartan entries are
\begin{align}
 (a_{01},a_{10})&=(-1,-2),
 &
 (a_{i,i+1},a_{i+1,i})&=(-1,-1)
 &&(1\leq i\leq n-2),
 \notag\\
 (a_{n-1,n},a_{n,n-1})&=(-1,-2).
 \steptag{5.2}\label{eq:a2n-cartan}
\end{align}
For $n=1$, we instead prescribe
\begin{equation}
 (a_{ij})=
 \begin{pmatrix}
  2&-1\\
  -4&2
 \end{pmatrix},
 \qquad
 (d_0,d_1)=(4,1),
 \qquad
 \delta=\alpha_0+2\alpha_1.
 \steptag{5.3}\label{eq:a2-rank-one-data}
\end{equation}

These conventions are summarized by
\cref{fig:a2n-dynkin}.  As above, arrows point toward the shorter simple
root.  The rank-one quadruple bond is displayed separately because
\eqref{eq:a2-rank-one-data} is not obtained by substituting \(n=1\) into
\eqref{eq:a2n-cartan}.

\begin{figure}[htbp]
\centering
\renewcommand{\arraystretch}{2.15}
\begin{tabular}{c@{\qquad}c}
\(A_2^{(2)}\ (n=1)\) &
\begin{tikzpicture}[x=1cm,y=1cm,baseline=(current bounding box.center)]
  \DynkinNode{r0}{(0,0)}{0}
  \DynkinNode{r1}{(1.45,0)}{1}
  \DynkinQuadruple{r0}{r1}\DynkinArrow{r0}{r1}
\end{tikzpicture}
\\
\(A_4^{(2)}\ (n=2)\) &
\begin{tikzpicture}[x=1cm,y=1cm,baseline=(current bounding box.center)]
  \DynkinNode{s0}{(0,0)}{0}
  \DynkinNode{s1}{(1.05,0)}{1}
  \DynkinNode{s2}{(2.10,0)}{2}
  \DynkinDouble{s0}{s1}\DynkinArrow{s0}{s1}
  \DynkinDouble{s1}{s2}\DynkinArrow{s1}{s2}
\end{tikzpicture}
\\
\(A_{2n}^{(2)}\ (n\geq3)\) &
\begin{tikzpicture}[x=1cm,y=1cm,baseline=(current bounding box.center)]
  \DynkinNode{b0}{(0,0)}{0}
  \DynkinNode{b1}{(1.05,0)}{1}
  \DynkinNode{bm}{(3.15,0)}{n-1}
  \DynkinNode{bn}{(4.20,0)}{n}
  \DynkinDouble{b0}{b1}\DynkinArrow{b0}{b1}
  \draw[dynkin bond] (b1)--(1.65,0);
  \node at (2.10,0) {\(\cdots\)};
  \draw[dynkin bond] (2.55,0)--(bm);
  \DynkinDouble{bm}{bn}\DynkinArrow{bm}{bn}
\end{tikzpicture}
\end{tabular}
\caption{The affine Dynkin diagrams of \(A_{2n}^{(2)}\) used in Part~\ref{part:a2n}.}
\label{fig:a2n-dynkin}
\end{figure}
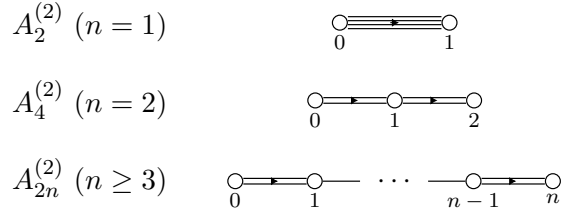

Retain and specialize the regular root
datum fixed in \cref{sec:common} to this Cartan matrix:
\begin{equation}
 (X,Y,I,\langle\ ,\ \rangle,
   \{\alpha_i\}_{i\in I},\{h_i\}_{i\in I}),
 \qquad I=\{0,1,\ldots,n\},
 \steptag{5.4}\label{eq:a2n-root-datum}
\end{equation}
where $X$ and $Y$ are free abelian groups in perfect duality,
$\langle h_i,\alpha_j\rangle=a_{ij}$, and
$\bigoplus_{i\in I}\mathbb Z\alpha_i\subset X$.  We abbreviate
\[
 U_{\mathcal A}^{\mathrm{Lus},(n)}
 :=U_{\mathcal A}^{\mathrm{Lus}}(A_{2n}^{(2)}).
\]

Specializing \eqref{eq:fundamental-translation-definition} to
\(\nu_i=1\), write
\(\widetilde\omega_i:=t_{\lambda_i}=t_{\omega_i^\vee}\) and
\(T_{\widetilde\omega_i}:=T_{\lambda_i}\).
Then \eqref{eq:current-divided-powers-braid-formula} with \(a=1\)
gives
\begin{equation}
 x_{i,k}^{+}
 =o(i)^k(T_{\widetilde\omega_i})^{-k}(E_i),
 \qquad
 x_{i,k}^{-}
 =o(i)^k(T_{\widetilde\omega_i})^{k}(F_i)
 \qquad(k\in\mathbb Z).
 \steptag{5.5}\label{eq:damiani-current-normalization}
\end{equation}
Equation~\eqref{eq:translation-real-root-vectors} gives
\[
\begin{aligned}
 E_{k\delta+\alpha_i}
 &=(T_{\widetilde\omega_i})^{-k}(E_i)
  =o(i)^k x_{i,k}^{+}
 &&(k\geq0),\\
 E_{k\delta-\alpha_i}
 &=(T_{\widetilde\omega_i})^kT_i^{-1}(E_i)
  =-o(i)^kK_{k\delta-\alpha_i}x_{i,k}^{-}
 &&(k>0).
\end{aligned}
\]
In particular,
\begin{equation}
 \deg x_{i,k}^{+}=k\delta+\alpha_i,
 \qquad
 \deg x_{i,k}^{-}=k\delta-\alpha_i.
 \steptag{5.6}\label{eq:current-degrees}
\end{equation}

Write
\(\UDrn{n}:=U_{\mathcal A}^{\mathrm{Dr}}(A_{2n}^{(2)})\).
Since \(I_i=\mathbb Z\), \eqref{eq:bare-real-current-core} becomes
\begin{equation}
 \UDrn{n}
 =\mathcal A\Bigl\langle
 K_\delta^{\pm1},K_i^{\pm1},(x_{i,k}^{\pm})^{(a)}
 \ \Bigm|\
 i\in I_0,\ k\in\mathbb Z,\ a\geq0
 \Bigr\rangle.
 \steptag{5.7}\label{eq:bare-current-a2n}
\end{equation}
Finally, \eqref{eq:Omega-commutes-braid} and
\eqref{eq:current-divided-powers-braid-formula} give
\begin{equation}
 \Omega\bigl((x_{i,k}^{\pm})^{(a)}\bigr)
 =(x_{i,-k}^{\mp})^{(a)},
 \qquad
 \Omega(\UDrn{n})=\UDrn{n}.
 \steptag{5.7a}\label{eq:a2n-Omega-stability}
\end{equation}
\begin{lemma}
\label{lem:affine-recovery-criterion}
Let $\mathscr X$ be an $\mathcal A$-subalgebra satisfying
\[
U_{\mathcal A}^{\mathrm{Dr},(n)}
\subseteq \mathscr X
\subseteq
U_{\mathcal A}^{\mathrm{Lus},(n)},
\]
and let $B=S^{-1}\mathcal A\subset\mathcal K$.
If, for every $a\geq0$,
\begin{equation}
 E_0^{(a)},F_0^{(a)}\in B\mathscr X,
 \steptag{5.8}\label{eq:affine-recovery-hyp}
\end{equation}
then
\[
 B\mathscr X=B U_{\mathcal A}^{\mathrm{Lus},(n)}.
\]
\end{lemma}

\begin{proof}
Setting $k=0$ in \eqref{eq:damiani-current-normalization} gives
\[
E_i^{(a)}=(x_{i,0}^{+})^{(a)},
\qquad
F_i^{(a)}=(x_{i,0}^{-})^{(a)}
\in\UDrn{n}
\qquad(i\in I_0,\ a\geq0).
\]
The elements $K_\delta^{\pm1},K_i^{\pm1}$ $(i\in I_0)$ also belong to
$\UDrn{n}\subseteq\mathscr X$, while
\eqref{eq:affine-recovery-hyp} supplies the affine simple divided powers.
Moreover,
\[
K_0^{\pm1}
=
\left(
K_\delta\prod_{i\in I_0}K_i^{-r_i}
\right)^{\pm1}
\in B\mathscr X.
\]
Thus $B\mathscr X$ contains $B$ and all defining generators of
$U_{\mathcal A}^{\mathrm{Lus},(n)}$, so
\[
B U_{\mathcal A}^{\mathrm{Lus},(n)}
\subseteq B\mathscr X.
\]
The reverse inclusion follows from
$\mathscr X\subseteq U_{\mathcal A}^{\mathrm{Lus},(n)}$.
\end{proof}

\section{Quantum scaled roots when \texorpdfstring{$n=1$}{n=1} and
\texorpdfstring{$C^{\ast}$}{C*}-recovery}
\label{sec:rank-one-sufficiency}
\subsection{The rank-one \texorpdfstring{$q$}{q}-Weyl cell}

By \eqref{eq:a2-rank-one-data},
\[
 q_0=q^4,
 \qquad
 q_1=q,
 \qquad
 \delta-\alpha_1=\alpha_0+\alpha_1.
\]
The root \(\delta-\alpha_1\) has the same root parameter \(q\) as
\(\alpha_1\).  Thus the divided powers of
\(E_{\delta-\alpha_1}\) and \(F_1\) are normalized by \([s]_q!\),
whereas those of \(E_0\) and \(F_0\) are normalized by
\([a]_{q^4}!\).

Let
\[
 W:=K_1E_0.
\]

Using \eqref{eq:a2n-initial-translation-word},
\eqref{eq:real-root-vector-definition}, and
\eqref{eq:braid-generator-formulas}, we get that
\(E_{\delta-\alpha_1}=T_0(E_1)
=-E_0E_1+q^{-4}E_1E_0\).
By the $q$-commutator techniques of 
\cite{J2}, we have that 
\[
 E_{\delta-\alpha_1}F_1-F_1E_{\delta-\alpha_1}
 =-[4]_qW.
\]
Since \(E_0F_1=F_1E_0\) and
\[
 K_1F_1K_1^{-1}=q^{-2}F_1,
\]
we have
\[
 WF_1
 =K_1E_0F_1
 =K_1F_1E_0
 =q^{-2}F_1K_1E_0
 =q^{-2}F_1W.
\]
Note that our real root vectors differ
from Akasaka's only by signs. It follows from \cite[Lemma~3.9(2)]{A} that
\[
 E_0E_{\delta-\alpha_1}
 =q^4E_{\delta-\alpha_1}E_0,
\]
while the weight of \(E_{\delta-\alpha_1}\) gives
\[
 K_1E_{\delta-\alpha_1}K_1^{-1}
 =q^{-2}E_{\delta-\alpha_1},
 \qquad\text{equivalently}\qquad
 E_{\delta-\alpha_1}K_1
 =q^2K_1E_{\delta-\alpha_1}.
\]
Consequently,
\[
 E_{\delta-\alpha_1}W
 =E_{\delta-\alpha_1}K_1E_0
 =q^2K_1E_{\delta-\alpha_1}E_0
 =q^{-2}K_1E_0E_{\delta-\alpha_1}
 =q^{-2}WE_{\delta-\alpha_1}.
\]
We have therefore obtained
\begin{equation}
 \begin{aligned}
 E_{\delta-\alpha_1}F_1-F_1E_{\delta-\alpha_1}
   &=-[4]_qW,\\
 WF_1
   &=q^{-2}F_1W,\\
 E_{\delta-\alpha_1}W
   &=q^{-2}WE_{\delta-\alpha_1}.
 \end{aligned}
 \steptag{6.1}\label{eq:rank-one-weyl}
\end{equation}

Formula \eqref{eq:damiani-current-normalization} gives
\begin{equation}
 x_{1,1}^-
 =-K_\delta^{-1}K_1E_{\delta-\alpha_1},
 \qquad
 K_1E_{\delta-\alpha_1}
 =q^{-2}E_{\delta-\alpha_1}K_1.
 \steptag{6.2}\label{eq:rank-one-dictionary}
\end{equation}
The first identity follows from
\eqref{eq:damiani-current-normalization} and
\eqref{eq:translation-real-root-vectors}; the second
is the Cartan relation above.

We now express the divided powers of \(E_{\delta-\alpha_1}\) in terms
of the current generator \(x_{1,1}^-\).  From
\[
 E_{\delta-\alpha_1}K_1
 =q^2K_1E_{\delta-\alpha_1},
\]
induction on \(s\) gives
\[
 (K_1E_{\delta-\alpha_1})^s
 =
 q^{2\binom{s}{2}}
 K_1^sE_{\delta-\alpha_1}^s.
\]
Indeed, assuming the formula for \(s\), one has
\begin{align*}
 (K_1E_{\delta-\alpha_1})^{s+1}
 &=
 (K_1E_{\delta-\alpha_1})^s
 K_1E_{\delta-\alpha_1}\\
 &=
 q^{2\binom{s}{2}}
 K_1^sE_{\delta-\alpha_1}^s
 K_1E_{\delta-\alpha_1}\\
 &=
 q^{2\binom{s}{2}+2s}
 K_1^{s+1}E_{\delta-\alpha_1}^{s+1}\\
 &=
 q^{2\binom{s+1}{2}}
 K_1^{s+1}E_{\delta-\alpha_1}^{s+1}.
\end{align*}
Since \(K_\delta\) is central,
\eqref{eq:rank-one-dictionary} therefore gives
\[
 (x_{1,1}^-)^s
 =
 (-1)^sK_\delta^{-s}
 (K_1E_{\delta-\alpha_1})^s
 =
 (-1)^s q^{2\binom{s}{2}}
 K_\delta^{-s}K_1^sE_{\delta-\alpha_1}^s.
\]
Left multiplication by \(K_\delta^sK_1^{-s}\) and division by
\([s]_q!\) yield
\[
 E_{\delta-\alpha_1}^{(s)}
 =
 (-1)^s q^{-2\binom{s}{2}}
 K_\delta^sK_1^{-s}(x_{1,1}^-)^{(s)}.
\]
All factors on the right belong to \(\UDrn{1}\).  Furthermore,
\eqref{eq:damiani-current-normalization} gives \(x_{1,0}^-=F_1\).
Hence, for every \(s\geq0\),
\begin{equation}
 E_{\delta-\alpha_1}^{(s)}
 =
 (-1)^s q^{-2\binom{s}{2}}
 K_\delta^sK_1^{-s}(x_{1,1}^-)^{(s)}
 \in\UDrn{1},
 \qquad
 F_1^{(s)}
 =(x_{1,0}^-)^{(s)}
 \in\UDrn{1}.
 \steptag{6.3}\label{eq:YF-integral}
\end{equation}

\subsection{A direct multiplier for the affine simple root}

\begin{proposition}[Quantum scaled roots when \texorpdfstring{$n=1$}{n=1}]
\label{prop:rank-one-direct-multiplier}
For \(a\geq0\), define
\begin{equation}
 d_a:=\prod_{h=1}^a(q^{2h}+1)(q^{4h}+1),
 \qquad
 d_0=1.
 \steptag{6.4}\label{eq:def-da}
\end{equation}
Then \(d_a\in\mathcal A\) and
\begin{equation}
 d_aE_0^{(a)},\qquad
 d_aF_0^{(a)}
 \in\UDrn{1}.
 \steptag{6.5}\label{eq:da-affine-powers}
\end{equation}
\end{proposition}

\begin{proof}
Write
$Z_W:=E_{\delta-\alpha_1}F_1-F_1E_{\delta-\alpha_1}$. Then it follows from \eqref{eq:rank-one-weyl} that
\[
 E_{\delta-\alpha_1}Z_W
 =q^{-2}Z_WE_{\delta-\alpha_1},
 \qquad
 Z_WF_1
 =q^{-2}F_1Z_W.
\]

Apply Lemma~\ref{lem:q-Weyl-highest-contraction} with
\[
 D=\UDrn{1},
 \qquad
 Q=q,
 \qquad
 A=E_{\delta-\alpha_1},
 \qquad
 B=F_1,
 \qquad
 Z=Z_W.
\]
The required divided-power hypotheses follow from
\eqref{eq:YF-integral}.  Hence
\begin{equation}
 \frac{Z_W^a}{[a]_q!}
 =
 (-1)^a\frac{([4]_q)^aW^a}{[a]_q!}
 \in\UDrn{1}.
 \steptag{6.6}\label{eq:ZW-contraction}
\end{equation}

We next rewrite \(W^a\) in terms of the affine simple divided power.
The Cartan relation
\[
 K_1E_0K_1^{-1}=q^{-4}E_0
\]
is equivalent to
\[
 E_0K_1=q^4K_1E_0.
\]
It follows by induction that
\[
 W^a=(K_1E_0)^a
 =q^{4\binom{a}{2}}K_1^aE_0^a.
\]
Indeed, if the identity holds for \(a\), then
\begin{align*}
 (K_1E_0)^{a+1}
 &=
 q^{4\binom{a}{2}}
 K_1^aE_0^aK_1E_0\\
 &=
 q^{4\binom{a}{2}+4a}
 K_1^{a+1}E_0^{a+1}\\
 &=
 q^{4\binom{a+1}{2}}
 K_1^{a+1}E_0^{a+1}.
\end{align*}
Since \(q_0=q^4\), one has
\[
 E_0^a=[a]_{q^4}!E_0^{(a)}.
\]
Therefore
\begin{equation}
 W^a
 =(K_1E_0)^a
 =q^{4\binom{a}{2}}K_1^aE_0^a
 =q^{4\binom{a}{2}}[a]_{q^4}!\,
 K_1^aE_0^{(a)}.
 \steptag{6.7}\label{eq:W-power}
\end{equation}

Substituting \eqref{eq:W-power} into
\eqref{eq:ZW-contraction} gives
\[
 (-1)^aq^{4\binom{a}{2}}K_1^a
 \frac{([4]_q)^a[a]_{q^4}!}{[a]_q!}
 E_0^{(a)}
 \in\UDrn{1}.
\]
Put
\[
 \widetilde d_a
 :=\frac{([4]_q)^a[a]_{q^4}!}{[a]_q!}.
\]
The scalar coefficient can be computed factor by factor:
\begin{align*}
 \widetilde d_a
 &=
 \prod_{h=1}^a
 \frac{[4]_q[h]_{q^4}}{[h]_q}=
 \prod_{h=1}^a
 \frac{[4h]_q}{[h]_q},
\end{align*}
where
\[
 [4]_q[h]_{q^4}
 =
 \frac{q^4-q^{-4}}{q-q^{-1}}
 \frac{q^{4h}-q^{-4h}}{q^4-q^{-4}}
 =
 \frac{q^{4h}-q^{-4h}}{q-q^{-1}}
 =[4h]_q.
\]
For every \(h\geq1\),
\begin{align}
 \frac{[4h]_q}{[h]_q}
 &=
 \frac{q^{4h}-q^{-4h}}{q^h-q^{-h}}
 \notag\\
 &=
 q^{-3h}\frac{q^{8h}-1}{q^{2h}-1}
 \notag\\
 &=
 q^{-3h}
 \bigl(q^{6h}+q^{4h}+q^{2h}+1\bigr)
 \notag\\
 &=
 q^{-3h}(q^{2h}+1)(q^{4h}+1).
 \steptag{6.8}\label{eq:da-factor}
\end{align}
Equations \eqref{eq:def-da} and \eqref{eq:da-factor} give
\[
 \widetilde d_a
 =q^{-\frac{3a(a+1)}{2}}d_a.
\]
Thus
\[
 (-1)^a q^{\frac{a(a-7)}{2}}
 d_aK_1^aE_0^{(a)}
 \in\UDrn{1}.
\]
Multiplication by the unit
\[
 (-1)^a q^{\frac{a(7-a)}{2}}K_1^{-a}
\]
gives
\[
 d_aE_0^{(a)}\in\UDrn{1}.
\]
The definition \eqref{eq:def-da} makes
\(d_a\in\mathbb Z[q]\subseteq\mathcal A\) immediate.

Finally, \eqref{eq:a2n-Omega-stability} gives
\[
 \Omega(\UDrn{1})=\UDrn{1}.
\]
The polynomial \(d_a\) satisfies
\[
 \Omega(d_a)
 =\prod_{h=1}^a(q^{-2h}+1)(q^{-4h}+1)
 =q^{-3a(a+1)}d_a.
\]
Moreover,
\[
 \Omega(E_0^{(a)})
 =
 \frac{F_0^a}{[a]_{q^{-4}}!}
 =
 \frac{F_0^a}{[a]_{q^4}!}
 =
 F_0^{(a)}.
\]
Applying \(\Omega\) to
\(d_aE_0^{(a)}\in\UDrn{1}\) therefore gives
\[
 q^{-3a(a+1)}d_aF_0^{(a)}
 =
 \Omega\bigl(d_aE_0^{(a)}\bigr)
 \in\UDrn{1}.
\]
Multiplication by the unit \(q^{3a(a+1)}\) gives
\(d_aF_0^{(a)}\in\UDrn{1}\).
This proves \eqref{eq:da-affine-powers}.
\end{proof}
\medskip

\begin{theorem}[Rank-one \(\Cstar\)-sufficiency]
\label{thm:rank-one-Cstar}
Recall $C^*=\mathcal A[\Phi_{4r}(q)^{-1}\mid r\geq1]$.
One has
\begin{equation}
 \Cstar\UDrn{1}
 =
 \Cstar U_{\mathcal A}^{\mathrm{Lus},(1)}.
 \steptag{6.9}\label{eq:rank-one-Cstar}
\end{equation}
\end{theorem}

\begin{proof}
The forward inclusion
\[
 \Cstar\UDrn{1}
 \subseteq
 \Cstar U_{\mathcal A}^{\mathrm{Lus},(1)}
\]
follows from
\(\UDrn{1}\subseteq U_{\mathcal A}^{\mathrm{Lus},(1)}\).
It remains to recover the affine simple divided powers after extending
scalars to \(\Cstar\).
By \eqref{eq:da-factor}, 
 $d_a
 =
 \prod_{h=1}^a(q^{2h}+1)(q^{4h}+1)$.
We determine the cyclotomic factors of the two polynomials occurring
here.  The standard factorization
\[
 q^N-1=\prod_{d\mid N}\Phi_d(q)
\]
gives
\[
 q^{2h}+1
 =
 \frac{q^{4h}-1}{q^{2h}-1}
 =
 \prod_{\substack{d\mid4h\\d\nmid2h}}\Phi_d(q).
\]
Write \(h=2^\nu m\), where \(m\) is odd.  An integer \(d\) divides
\(4h\) but not \(2h\) precisely when
\[
 d=2^{\nu+2}t
 \qquad\text{for some }t\mid m.
\]
Writing \(s=2^\nu t\), this is equivalent to
\[
 d=4s,
 \qquad
 s\mid h,
 \qquad
 h/s\ \text{odd}.
\]
Therefore
\[
 q^{2h}+1
 =
 \prod_{\substack{s\mid h\\h/s\ \mathrm{odd}}}
 \Phi_{4s}(q).
\]

Similarly,
\[
 q^{4h}+1
 =
 \frac{q^{8h}-1}{q^{4h}-1}
 =
 \prod_{\substack{d\mid8h\\d\nmid4h}}\Phi_d(q).
\]
The same \(2\)-adic argument shows that the surviving indices are
\(d=8s\), where \(s\mid h\) and \(h/s\) is odd.  Hence
\[
 q^{4h}+1
 =
 \prod_{\substack{s\mid h\\h/s\ \mathrm{odd}}}
 \Phi_{8s}(q)
 =
 \prod_{\substack{s\mid h\\h/s\ \mathrm{odd}}}
 \Phi_{4(2s)}(q).
\]
It follows that
\[
 d_a
 =
 \prod_{h=1}^a
 \prod_{\substack{s\mid h\\h/s\ \mathrm{odd}}}
 \Phi_{4s}(q)\Phi_{8s}(q).
\]
Every factor \(\Phi_{4s}(q)\) is invertible in
\(\Cstar\), and
\[
 \Phi_{8s}(q)=\Phi_{4(2s)}(q)
\]
is invertible there as well.  Consequently,
\[
 d_a\in(\Cstar)^\times
 \qquad(a\geq0).
\]

Proposition~\ref{prop:rank-one-direct-multiplier} gives
\[
 d_aE_0^{(a)},\qquad
 d_aF_0^{(a)}
 \in\UDrn{1}
 \subseteq\Cstar\UDrn{1}.
\]
Multiplication by \(d_a^{-1}\in\Cstar\) therefore yields
\[
 E_0^{(a)},\qquad
 F_0^{(a)}
 \in\Cstar\UDrn{1}
 \qquad(a\geq0).
\]
Thus the hypothesis \eqref{eq:affine-recovery-hyp} of
Lemma~\ref{lem:affine-recovery-criterion} holds with
\[
 \mathscr X=\UDrn{1},
 \qquad
 B=\Cstar.
\]
Applying that lemma gives
\[
 \Cstar\UDrn{1}
 =
 \Cstar U_{\mathcal A}^{\mathrm{Lus},(1)}.
\]
\end{proof}
The sufficiency in Rank-one is now complete. Under the classical specialization \(q\mapsto1\), \(E_0\mapsto e_0\),
we have \(d_a(1)=4^a\). Thus \(d_aE_0^{(a)}\) specializes to
\((4e_0)^{(a)}\), recovering the divided powers of classical scaled-roots 
in \cite[Remark~C.3]{DP}.
\section{Quantum scaled roots when
\texorpdfstring{$n\geq2$}{n>=2} and higher-rank
\texorpdfstring{$B^{\ast}$}{B*}-recovery}
\label{sec:quantum-dp-scaled-roots}

We next establish the corresponding
higher-rank statement.  The mechanism is genuinely different: the
required multiplier is produced by a three-root cell, and its nonunit factors
have only \(\Phi_{8r}\)-support.

Throughout this section, fix \(n\geq2\), retain the integral form \(\UDrn{n}\) of \eqref{eq:bare-current-a2n}, and put \(Q=q^2\).

\subsection{Braid group stability of
\texorpdfstring{$U_{\mathcal A}^{\mathrm{Dr},(n)}$}{the Drinfeld integral form}}
\label{subsec:a2n-braid-stability}

\begin{lemma}
\label[lemma]{lem:required-standard-lifts}
For \(1\leq h\leq n\) and \(\varepsilon\in\{\pm1\}\),
\begin{equation}
 (T_{\widetilde\omega_h})^\varepsilon
   \bigl(\UDrn{n}\bigr)=\UDrn{n},
 \qquad
 T_i\bigl(\UDrn{n}\bigr)\subseteq\UDrn{n}
 \quad(1\leq i<n).
 \label{eq:required-standard-lifts}
\end{equation}
Moreover, if \(1\leq j<n\), then
\[
 T_n\!\left((x_{j,r}^{\pm})^{(a)}\right)
 \in\UDrn{n}
 \qquad(r\in\mathbb Z,\ a\geq0).
\]
\end{lemma}

\begin{proof}
The translation assertion follows from the same generator-by-generator
argument as in Lemma~\ref{lem:translation-stability}, now using
\eqref{eq:damiani-current-normalization}.  Indeed,
\((T_{\widetilde\omega_h})^{\pm1}\) shifts only the mode
of a color-\(h\) current, up to a sign and central or toral Laurent
monomials.  Hence both operators preserve the generators of
\eqref{eq:bare-current-a2n}, and therefore
\[
 (T_{\widetilde\omega_h})^\varepsilon
 (\UDrn{n})=\UDrn{n}
 \qquad(\varepsilon=\pm1).
\]

Now let \(1\leq i<n\).  After the finite-node relabelling, the
color-\(i\) currents lie in an ordinary \(A_1^{(1)}\) vertex.  If \(p\)
and \(q\) are its two nodes, then
\[
 a_{pq}=a_{qp}=-2.
\]
The argument is therefore similar to the proof of
Proposition~\ref{prop:finite-braid-stability}: the embedded
\(A_1^{(1)}\) vertex treats the color-\(i\) currents, while the
different-color argument applies with
\[
 t_{\lambda_j}\ \text{replaced by}\ 
 \widetilde\omega_j=t_{\omega_j^\vee}
\]
and uses \eqref{eq:damiani-current-normalization}.  The negative-current
case follows from
\[
 \Omega T_i=T_i\Omega,
 \qquad
 \Omega(\UDrn{n})=\UDrn{n},
\]
and the toral generators are treated exactly as in that proof.  Hence
\[
 T_i\bigl(\UDrn{n}\bigr)\subseteq\UDrn{n}
 \qquad(1\leq i<n).
\]

It remains to prove the last assertion.  Let \(1\leq j<n\).  Since
\[
 \langle\omega_j^\vee,\alpha_n\rangle=0,
 \qquad
 s_n\widetilde\omega_j
 =\widetilde\omega_js_n,
\]
the different-color part of the proof of
Proposition~\ref{prop:finite-braid-stability} applies with \(i=n\).
Together with \eqref{eq:damiani-current-normalization} and
\eqref{eq:a2n-Omega-stability}, it gives
\[
 T_n\bigl((x_{j,r}^{\pm})^{(a)}\bigr)\in\UDrn{n}
 \qquad
 (1\leq j<n,\ r\in\mathbb Z,\ a\geq0).
\]
\end{proof}
No assertion is made here for the color-\(n\) currents: the embedded
\(A_1^{(1)}\) argument is unavailable at the node \(n\), and only the
different-color case is needed in the following.

\subsection{Transport of the non-extra-long roots}
\label{subsec:transport-to-beta-prime}
\noindent\textbf{Remark}. We explain why the higher rank case needs only the
\(\Phi_{8r}\), not the \(\Phi_{4r}\).
In our numbering, there is an embedding
\[
 \iota_n:U_q(A_2^{(2)})\lhook\joinrel\longrightarrow
 U_q(A_{2n}^{(2)})
\]
that satisfies
\[
 \iota_n(E_1)=E_n,
 \qquad
 \iota_n(E_0)=E_{\delta-2\alpha_n},
 \qquad
 T_n \circ \iota_n=\iota_n\circ T_1.
\]
(cf. \cite[Defn.~4.1.1, Props.~4.3.11,
4.4.9--4.4.10]{D1}).  This embedding does not reduce
\cref{thm:higher-rank-Bstar} to \cref{thm:rank-one-Cstar}.  Indeed, the
image of \(\iota_n\) is graded by
\(\mathbb Z\alpha_n+\mathbb Z\delta\), in which every element has even
\(\alpha_{n-1}\)-coefficient, whereas that of
\(\beta'=\delta-(\alpha_{n-1}+2\alpha_n)\) equals one; neither
\(\beta=\alpha_{n-1}\) nor \(\beta'\) lies in that image.  The
three-root cell \(\{\beta,\beta',\eta_n\}\) of
\cref{lem:special-three-root-cell} thus has at its disposal generators
that are unavailable in rank one, and this is precisely what sharpens
the multiplier: the cell produces the coefficient \(c_a\) of
\eqref{eq:eta-and-ca}, whose nonunit factors are exactly the
\(\Phi_{8d}(q)\), whereas the rank-one multiplier \(d_a\) of
\eqref{eq:def-da} carries in addition the factors \(\Phi_{4s}(q)\) with
\(s\) odd.

Let
\[
 \mu=\delta-\alpha_{n-1},
 \qquad
 \beta'=\delta-(\alpha_{n-1}+2\alpha_n),
 \qquad
 K_\mu:=K_\delta K_{n-1}^{-1}.
\]

\begin{proposition}
\label[proposition]{prop:E-mu-integral}
One has \((\mu\mid\mu)=4\),
\[
 K_\mu x_{n-1,1}^{-}K_\mu^{-1}=Q^{2}x_{n-1,1}^{-},
\]
and, for every \(a\geq0\),
\begin{equation}
 E_\mu=K_\mu x_{n-1,1}^{-},
 \qquad
 E_\mu^{(a)}
 =Q^{-2\binom a2}K_\mu^{a}\bigl(x_{n-1,1}^{-}\bigr)^{(a)}
 \in\UDrn{n}.
 \label{eq:E-mu-integral}
\end{equation}
\end{proposition}

\begin{proof}
Since \(\langle\omega_{n-1}^{\vee},\alpha_{n-1}\rangle=1\) and
\(\widetilde\omega_{n-1}=t_{\omega_{n-1}^{\vee}}\),
\begin{equation}
 \widetilde\omega_{n-1}(\alpha_{n-1})=\alpha_{n-1}-\delta,
 \qquad
 \widetilde\omega_{n-1}(\delta)=\delta.
 \label{eq:omega-on-alpha-n-1}
\end{equation}
Equations~\eqref{eq:translation-real-root-vectors} and
\eqref{eq:braid-generator-formulas} give
\[
 E_\mu
 =E_{\delta-\alpha_{n-1}}
 =T_{\widetilde\omega_{n-1}}
   \bigl(T_{n-1}\bigr)^{-1}(E_{n-1}),
 \qquad
 \bigl(T_{n-1}\bigr)^{-1}(E_{n-1})
 =-K_{n-1}^{-1}F_{n-1}.
\]
By \cite[Section~37.1.3]{L} the braid group operators act
on the torus by \(T_w(K_\eta)=K_{w(\eta)}\), so
\eqref{eq:omega-on-alpha-n-1} gives
\[
 T_{\widetilde\omega_{n-1}}\bigl(K_{n-1}^{-1}\bigr)
 =K_{\delta-\alpha_{n-1}}
 =K_\delta K_{n-1}^{-1}
 =K_\mu.
\]
The transported sign function satisfies \(o(n-1)=-1\), so
\eqref{eq:damiani-current-normalization} at \(k=1\) gives
\(T_{\widetilde\omega_{n-1}}(F_{n-1})=-x_{n-1,1}^{-}\).
Multiplying the two images yields
\[
 E_\mu=K_\mu x_{n-1,1}^{-}.
\]

By \eqref{eq:current-degrees}, \(\deg x_{n-1,1}^{-}=\delta-\alpha_{n-1}=\mu\).
Since \((\delta\mid\gamma)=0\) for every \(\gamma\) in the affine root
lattice,
\[
 (\mu\mid\mu)
 =(\alpha_{n-1}\mid\alpha_{n-1})
 =2d_{n-1}
 =4,
\]
and therefore
\begin{equation}
 K_\mu x_{n-1,1}^{-}K_\mu^{-1}
 =q^{(\mu\mid\mu)}x_{n-1,1}^{-}
 =Q^{2}x_{n-1,1}^{-},
 \qquad\text{equivalently}\qquad
 x_{n-1,1}^{-}K_\mu=Q^{-2}K_\mu x_{n-1,1}^{-}.
 \label{eq:K-mu-commutation}
\end{equation}
Writing \(x=x_{n-1,1}^{-}\), induction on \(a\) gives
\[
 (K_\mu x)^{a}=Q^{-2\binom a2}K_\mu^{a}x^{a}:
\]
the case \(a=0\) is trivial, and if the formula holds for \(a\), then
\eqref{eq:K-mu-commutation} gives \(x^{a}K_\mu=Q^{-2a}K_\mu x^{a}\),
whence
\[
 (K_\mu x)^{a+1}
 =Q^{-2\binom a2}K_\mu^{a}x^{a}K_\mu x
 =Q^{-2\binom a2-2a}K_\mu^{a+1}x^{a+1}
 =Q^{-2\binom{a+1}2}K_\mu^{a+1}x^{a+1}.
\]
Finally, \(\mu\) has root parameter
\(q^{(\mu\mid\mu)/2}=q^{2}=Q=q_{n-1}\), so \(E_\mu^{(a)}\) and
\(x^{(a)}\) are normalized by the same factorial \([a]_Q!\).  Dividing
\(E_\mu^{a}=(K_\mu x)^{a}\) by \([a]_Q!\) therefore gives the
displayed formula for \(E_\mu^{(a)}\).  Its right-hand side lies in
\(\UDrn{n}\), because \(K_\mu^{a}=(K_\delta K_{n-1}^{-1})^{a}\) is a
monomial in the group-like toral generators, \(x^{(a)}\) is a defining
generator of \eqref{eq:bare-current-a2n}, and
\(Q^{-2\binom a2}\in\mathcal A^{\times}\).
\end{proof}

\begin{proposition}[Transport to the root \(\beta'\)]
\label{prop:Bsharp-integral}
Put
\[
 \widehat w_\beta:=s_n\widetilde\omega_{n-1}s_{n-1},
 \qquad
 \mathcal T_\beta
 :=T_nT_{\widetilde\omega_{n-1}}
   \bigl(T_{n-1}\bigr)^{-1},
 \qquad
 B^{\sharp}:=\mathcal T_\beta(E_{n-1}).
\]
Then \(\widehat w_\beta(\alpha_{n-1})=\beta'\), the operator
\(\mathcal T_\beta\) is the standard lift
\(T_{\widehat w_\beta}\), and
\begin{equation}
 B^{\sharp}=T_n(E_\mu),
 \qquad
 \bigl(B^{\sharp}\bigr)^{(a)}
 =T_n\bigl(E_\mu^{(a)}\bigr)
 \in\UDrn{n}
 \qquad(a\geq0).
 \label{eq:Bsharp-integral}
\end{equation}
\end{proposition}

\begin{proof}
By \eqref{eq:a2n-cartan}, \(a_{n,n-1}=-2\) and \(a_{n-1,n}=-1\), so
\begin{equation}
 s_n(\alpha_{n-1})=\alpha_{n-1}+2\alpha_n,
 \qquad
 s_{n-1}(\alpha_n)=\alpha_n+\alpha_{n-1}.
 \label{eq:a2n-rank-two-reflections}
\end{equation}
Using \eqref{eq:omega-on-alpha-n-1}, \(s_n(\delta)=\delta\), and
\eqref{eq:a2n-rank-two-reflections}, we compute
\[
 \alpha_{n-1}
 \xmapsto{\ s_{n-1}\ }-\alpha_{n-1}
 \xmapsto{\ \widetilde\omega_{n-1}\ }\delta-\alpha_{n-1}=\mu
 \xmapsto{\ s_n\ }\delta-(\alpha_{n-1}+2\alpha_n)=\beta',
\]
which proves \(\widehat w_\beta(\alpha_{n-1})=\beta'\).

Next we identify the standard lift.  Let
\(y:=\widetilde\omega_{n-1}s_{n-1}\).  By
\eqref{eq:omega-on-alpha-n-1},
\(\widetilde\omega_{n-1}(\alpha_{n-1})=\alpha_{n-1}-\delta<0\), so the
right-descent criterion gives
\(\ell(y)=\ell(\widetilde\omega_{n-1})-1\) and hence
\[
 T_y
 =T_{\widetilde\omega_{n-1}}
  \bigl(T_{n-1}\bigr)^{-1}.
\]
Moreover \(\langle\omega_{n-1}^{\vee},\alpha_n\rangle=0\) gives
\(\widetilde\omega_{n-1}^{-1}(\alpha_n)=\alpha_n\), so by
\eqref{eq:a2n-rank-two-reflections}
\[
 y^{-1}(\alpha_n)
 =s_{n-1}\widetilde\omega_{n-1}^{-1}(\alpha_n)
 =s_{n-1}(\alpha_n)
 =\alpha_n+\alpha_{n-1}>0,
\]
whence \(\ell(s_ny)=\ell(y)+1\) and
\[
 T_{\widehat w_\beta}
 =T_nT_y
 =T_nT_{\widetilde\omega_{n-1}}
   \bigl(T_{n-1}\bigr)^{-1}
 =\mathcal T_\beta.
\]
Combining this with the factorization of \(E_\mu\) in the proof of
\cref{prop:E-mu-integral} gives
\[
 B^{\sharp}
 =T_n
   \Bigl(T_{\widetilde\omega_{n-1}}
     \bigl(T_{n-1}\bigr)^{-1}(E_{n-1})\Bigr)
 =T_n(E_\mu).
\]

Since \(s_n\) preserves the invariant form,
\[
 (\beta'\mid\beta')=(s_n\mu\mid s_n\mu)=(\mu\mid\mu)=4,
\]
so \(\beta'\) and \(\mu\) have the same root parameter \(Q\), and both
\(E_\mu^{(a)}\) and \((B^{\sharp})^{(a)}\) are normalized by
\([a]_Q!\).  As \(T_n\) is a \(\mathcal K\)-algebra
automorphism,
\[
 \bigl(B^{\sharp}\bigr)^{(a)}
 =\frac{T_n(E_\mu)^{a}}{[a]_Q!}
 =T_n\bigl(E_\mu^{(a)}\bigr).
\]
Applying \(T_n\) to \eqref{eq:E-mu-integral} and using
\(T_n(K_\mu)=K_{s_n(\mu)}=K_{\beta'}\) gives the explicit
expression
\[
 \bigl(B^{\sharp}\bigr)^{(a)}
 =Q^{-2\binom a2}
  K_{\beta'}^{a}\,
  T_n\Bigl(\bigl(x_{n-1,1}^{-}\bigr)^{(a)}\Bigr),
 \qquad
 K_{\beta'}=K_\delta K_{n-1}^{-1}K_n^{-2}.
\]
Here \(K_{\beta'}^{\pm a}\in\UDrn{n}\), and
\(T_n\bigl((x_{n-1,1}^{-})^{(a)}\bigr)\in\UDrn{n}\) by the
last assertion of \cref{lem:required-standard-lifts}, applied with the
color \(n-1\neq n\).  This proves \eqref{eq:Bsharp-integral}.
\end{proof}

\noindent
Proposition~\ref{prop:Bsharp-integral} completes the integral transport for this
non-extra-long root.  It remains to identify \(B^{\sharp}\) with the
root vector \(E_{\beta'}\) fixed by the convex order.

Let
\begin{equation}
 \beta=\alpha_{n-1},
 \qquad
 \eta_j=\delta-2(\alpha_j+\cdots+\alpha_n)
 \quad(1\leq j\leq n),
 \qquad
 c_a=\prod_{h=1}^{a}(q^{4h}+1),
 \qquad c_0=1.
 \label{eq:eta-and-ca}
\end{equation}
The roots \(\eta_j\) are extra-long real roots with root parameter \(Q^2\),
whereas \(\beta\) and \(\beta'\) have root parameter \(Q\).

\begin{lemma}
\label{lem:exact-non-extra-long-endpoint}
With \(\mu=\delta-\alpha_{n-1}\),
\(\beta'=\delta-(\alpha_{n-1}+2\alpha_n)\), and
\(B^{\sharp}=T_n(E_\mu)\) as above, one has
\[
 B^{\sharp}=E_{\beta'}.
\]
Consequently \((E_{\beta'})^{(a)}\in\UDrn{n}\) for every \(a\geq0\).
\end{lemma}

\begin{proof}
Let \(e_i=\alpha_i+\cdots+\alpha_n\) and
\(w=s_0s_1\cdots s_n\). After transporting the finite-node
numbering, \cite[Lemma~4.2.3]{D1} gives
\begin{equation}
 w^m(e_i)=
 \begin{cases}
  e_{i+m},&i+m\leq n,\\
  e_{i+m-n}-\delta,&i+m>n,
 \end{cases}
 \qquad 0\leq m\leq n.
 \label{eq:endpoint-w-power-action}
\end{equation}
Set \(u=(s_2\cdots s_n)(s_1\cdots s_n)\). Direct calculation gives
\[
 u(e_i)=
 \begin{cases}
  e_{i+2},&1\leq i\leq n-2,\\
  -e_2,&i=n-1,\\
  -e_1,&i=n.
 \end{cases}
\]
Since \(\widehat w_\beta=s_n\widetilde\omega_{n-1}s_{n-1}\),
the translation action gives
\[
\begin{aligned}
 \widehat w_\beta(e_i)
 &=s_n(e_i-\delta)=e_i-\delta
 &&(1\leq i\leq n-2),\\
 \widehat w_\beta(e_{n-1})
 &=s_n(e_n)=-e_n,\\
 \widehat w_\beta(e_n)
 &=s_n(e_{n-1}-\delta)=e_{n-1}-\delta.
\end{aligned}
\]
On the other hand, \(s_0(e_1)=\delta-e_1\) and
\(s_0(e_j)=e_j\) for \(j\geq2\), so
\eqref{eq:endpoint-w-power-action} gives
\[
\begin{aligned}
 w^{n-2}s_0u(e_i)
 &=w^{n-2}(e_{i+2})=e_i-\delta
 &&(1\leq i\leq n-2),\\
 w^{n-2}s_0u(e_{n-1})
 &=-w^{n-2}(e_2)=-e_n,\\
 w^{n-2}s_0u(e_n)
 &=w^{n-2}(e_1-\delta)=e_{n-1}-\delta.
\end{aligned}
\]
Both elements also fix \(\delta\). Since
\(\delta,e_1,\ldots,e_n\) form a basis of the affine root lattice,
their actions coincide. The Weyl-group action is faithful, so
\begin{equation}
 \widehat w_\beta=w^{n-2}s_0u.
 \label{eq:endpoint-weyl-identity}
\end{equation}

For \(\widetilde\omega_{n-1}\), the short and extra-long root
strings each contribute \(n-1\) inversions, while the long strings
contribute \(2\binom{n-1}{2}+2(n-1)=n(n-1)\). Hence
\[
 \ell(\widetilde\omega_{n-1})
 =(n-1)+n(n-1)+(n-1)=(n-1)(n+2).
\]
The descent and ascent calculations in the proof of
Proposition~\ref{prop:Bsharp-integral} give
\[
 \ell(\widehat w_\beta)
 =\ell(\widetilde\omega_{n-1}s_{n-1})+1
 =\ell(\widetilde\omega_{n-1}).
\]
The word \(w^{n-2}s_0(s_2\cdots s_n)(s_1\cdots s_n)\)
has exactly
\[
 (n-2)(n+1)+1+(n-1)+n=(n-1)(n+2)
\]
letters, so \eqref{eq:endpoint-weyl-identity} is a reduced
factorization. Proposition~\ref{prop:Bsharp-integral} therefore gives
\begin{equation}
 T_nT_{\widetilde\omega_{n-1}}T_{n-1}^{-1}
 =T_{\widehat w_\beta}
 =T_{w^{n-2}}T_0T_u.
 \label{eq:braid-endpoint-identity}
\end{equation}

Since \(u(\alpha_{n-1})=\alpha_1\),
Lemma~\ref{lem:simple-to-simple} gives
\(T_u(E_{n-1})=E_1\). And since \(s_0(\alpha_1)=\alpha_0+\alpha_1=\delta-e_1-e_2\),
\eqref{eq:endpoint-w-power-action} gives
\[
 w^{n-2}s_0(\alpha_1)
 =\delta-e_{n-1}-e_n
 =\delta-(\alpha_{n-1}+2\alpha_n)
 =\beta'.
\]
By \eqref{eq:a2n-initial-translation-word} and
\eqref{eq:real-root-vector-definition},
\(w^{n-2}s_0\) is a prefix followed by \(s_1\), so
\[
 E_{\beta'}=T_{w^{n-2}}T_0(E_1).
\]
Applying \eqref{eq:braid-endpoint-identity} to \(E_{n-1}\) yields
\begin{equation}
 B^{\sharp}=T_n(E_\mu)
 =T_{w^{n-2}}T_0T_u(E_{n-1})
 =T_{w^{n-2}}T_0(E_1)=E_{\beta'}.
 \label{eq:Bsharp-normalization}
\end{equation}
The divided-power assertion follows from \eqref{eq:Bsharp-integral}.
\end{proof}

\subsection{The special three-root cell}
\label{subsec:special-three-root-cell}

\begin{lemma}[The special three-root cell]
\label[lemma]{lem:special-three-root-cell}
Put
\[
 A=E_\beta,
 \qquad
 B=E_{\beta'},
 \qquad
 G=E_{\eta_n}.
\]
Then \(A^{(a)},B^{(a)}\in\UDrn{n}\) for every \(a\geq0\), and
\begin{equation}
 AB-BA=[2]_QG,
 \qquad
 AG=Q^{-2}GA,
 \qquad
 GB=Q^{-2}BG.
 \label{eq:special-cell-relations}
\end{equation}
\end{lemma}

\begin{proof}
The divided powers of \(A=E_{n-1}=x_{n-1,0}^{+}\) belong to
\(\UDrn{n}\) by \eqref{eq:bare-current-a2n}; those of \(B\)
belong to \(\UDrn{n}\) by
Lemma~\ref{lem:exact-non-extra-long-endpoint}.

Let \(\mathcal W=T_{w^{n-2}}\) and
\begin{equation}
 a=E_1,\qquad b=T_0(E_1),\qquad g=T_0T_1(E_0).
 \label{eq:cell-generators}
\end{equation}
Equation~\eqref{eq:endpoint-w-power-action} and
Lemma~\ref{lem:simple-to-simple} give \(\mathcal W(a)=A\).
The proof of Lemma~\ref{lem:exact-non-extra-long-endpoint} gives
\(\mathcal W(b)=B\). By \eqref{eq:endpoint-w-power-action},
\[
 w^{n-1}(\alpha_0)
 =w^{n-1}(\delta-2e_1)
 =\delta-2e_n
 =\eta_n.
\]
By \eqref{eq:a2n-initial-translation-word}, \(w^{n-1}\) is a prefix
followed by \(s_0\); \eqref{eq:real-root-vector-definition} therefore gives
\[
\begin{aligned}
 G
 &=T_{w^{n-1}}(E_0)\\
 &=\mathcal W T_0T_1\cdots T_n(E_0)\\
 &=\mathcal W T_0T_1(E_0)
 =\mathcal W(g),
\end{aligned}
\]
where \(T_j(E_0)=E_0\) for \(j\geq2\), since \(a_{j0}=0\).

Moreover, \(s_1s_0s_1\) is reduced in the finite Weyl subgroup
of type \(B_2\) generated by \(s_0,s_1\), and
\[
 s_1s_0s_1(\alpha_0)
 =s_1s_0(\alpha_0+2\alpha_1)
 =s_1(\alpha_0+2\alpha_1)
 =\alpha_0.
\]
Lemma~\ref{lem:simple-to-simple} therefore gives
\[
 T_1T_0T_1(E_0)=E_0.
\]
Applying \(T_1^{-1}\), we obtain
\[
 g=T_0T_1(E_0)=T_1^{-1}(E_0).
\]
It therefore suffices to prove
\begin{equation}
 ab-ba=[2]_Q g,\qquad ag=Q^{-2}ga,\qquad gb=Q^{-2}bg.
 \label{eq:rank-two-cell}
\end{equation}

Since \(a_{01}=-1\), \(a_{10}=-2\), \(q_0=Q^2\), and \(q_1=Q\),
equation~\ref{eq:braid-generator-formulas} gives
\[
\begin{aligned}
 b=T_0(E_1)
 &=-E_0E_1+q_0^{-1}E_1E_0
  =-E_0E_1+Q^{-2}E_1E_0,\\
 g=T_1^{-1}(E_0)
 &=\sum_{r=0}^{2}(-1)^r q_1^{-r}
   E_1^{(r)}E_0E_1^{(2-r)}\\
 &=E_0E_1^{(2)}-Q^{-1}E_1E_0E_1
   +Q^{-2}E_1^{(2)}E_0.
\end{aligned}
\]
Consequently,
\[
 ab-ba
 =E_0E_1^2-(1+Q^{-2})E_1E_0E_1+Q^{-2}E_1^2E_0
 =[2]_Q g.
\]
For the remaining relations, the Chevalley relations give
\[
 [E_0,F_1]=0,\qquad K_1E_0K_1^{-1}=Q^{-2}E_0.
\]
Since \(T_1(a)=-F_1K_1\) and \(T_1^{-1}(a)=-K_1^{-1}F_1\), we obtain
\[
 T_1(a)E_0=Q^{-2}E_0T_1(a),\qquad
 E_0T_1^{-1}(a)=Q^{-2}T_1^{-1}(a)E_0.
\]
Applying \(T_1^{-1}\) to the first equality and \(T_0T_1\) to the
second gives \(ag=Q^{-2}ga\) and \(gb=Q^{-2}bg\), respectively.
Applying the algebra automorphism \(\mathcal W\) to
\eqref{eq:rank-two-cell} proves \eqref{eq:special-cell-relations}.
\end{proof}

\begin{proposition}[Quantum scaled roots when
\texorpdfstring{$n\geq2$}{n>=2}]
\label{prop:special-cell-base-scaled}
For every \(a\geq0\),
\begin{equation}
 c_aE_{\eta_n}^{(a)}\in\UDrn{n}.
 \label{eq:eta-n-scaled}
\end{equation}
\end{proposition}

\begin{proof}
Set \(Z=AB-BA=[2]_QG\). By
Lemma~\ref{lem:special-three-root-cell},
\[
 AZ=Q^{-2}ZA,
 \qquad
 ZB=Q^{-2}BZ,
 \qquad
 A^{(m)},B^{(m)}\in\UDrn{n}\quad(m\geq0).
\]
Apply Lemma~\ref{lem:q-Weyl-highest-contraction} with
\(D=\UDrn{n}\). It gives
\begin{equation}
 \frac{Z^a}{[a]_Q!}\in\UDrn{n}.
 \label{eq:Z-integral}
\end{equation}
For every \(h\geq1\), the definitions of the symmetric quantum integers give
\[
 [h]_{Q^2}=[h]_Q\frac{Q^h+Q^{-h}}{[2]_Q}.
\]
Multiplying this identity for \(h=1,\ldots,a\) gives
\begin{equation}
 [a]_{Q^2}!
 =\frac{[a]_Q!}{[2]_Q^a}\widetilde c_a,
 \qquad
 \widetilde c_a:=\prod_{h=1}^{a}(Q^h+Q^{-h}).
 \label{eq:factorial-conversion-Q}
\end{equation}
Since \(Q=q^2\), the definition \eqref{eq:eta-and-ca} gives
\[
 \widetilde c_a
 =\prod_{h=1}^a q^{-2h}(q^{4h}+1)
 =q^{-a(a+1)}c_a.
\]
Since \(G=E_{\eta_n}\) has root parameter \(Q^2\), equations
\eqref{eq:Z-integral}--\eqref{eq:factorial-conversion-Q} give
\[
 \frac{Z^a}{[a]_Q!}
 =\widetilde c_a\frac{G^a}{[a]_{Q^2}!}
 =q^{-a(a+1)}c_aE_{\eta_n}^{(a)}\in\UDrn{n}.
\]
Multiplication by the unit \(q^{a(a+1)}\) proves
\eqref{eq:eta-n-scaled}.
\end{proof}

Therefore, the coefficient \(c_a\) is created entirely inside the special three-root
cell.

\subsection{Weyl transport of the scaled roots}
\label{subsec:scaled-root-weyl-transport}

\begin{lemma}
\label{lem:exact-scaled-root-endpoint}
Let
\[
 w_\eta=s_1s_2\cdots s_{n-1},
\]
Then
\begin{equation}
 w_\eta(\eta_n)=\eta_1=\alpha_0,
 \qquad
 T_{w_\eta}(E_{\eta_n})=E_0.
 \label{eq:exact-scaled-root-endpoint}
\end{equation}
\end{lemma}

\begin{proof}
In the finite coordinates
\(\alpha_i=e_i-e_{i+1}\) for \(i<n\) and \(\alpha_n=e_n\), one has
\(\eta_j=\delta-2e_j\). Hence
\(s_{j-1}(\eta_j)=\eta_{j-1}\) for \(2\leq j\leq n\),
and the first identity follows.

After transporting the finite-node numbering,
\cite[Cor.~4.2.8]{D1} gives
\[
 E_{\eta_n}=T_{w_\eta}^{-1}(E_0),
\]
since \(s_0ws_n=s_1\cdots s_{n-1}=w_\eta\).
Applying \(T_{w_\eta}\) yields
\[
 T_{w_\eta}(E_{\eta_n})=E_0,
\]
which proves the second identity in
\eqref{eq:exact-scaled-root-endpoint}.
\end{proof}

\begin{proposition}
\label{prop:scaled-affine-divided-powers}
For every \(a\geq0\),
\begin{equation}
 c_aE_0^{(a)},
 \qquad
 c_aF_0^{(a)}
 \in\UDrn{n}.
 \label{eq:scaled-affine-powers}
\end{equation}
\end{proposition}

\begin{proof}
Proposition~\ref{prop:special-cell-base-scaled} gives
\(c_aE_{\eta_n}^{(a)}\in\UDrn{n}\).
Lemmas~\ref{lem:required-standard-lifts} and
\ref{lem:exact-scaled-root-endpoint} yield
\[
 c_aE_0^{(a)}
 =T_{w_\eta}\!\left(c_aE_{\eta_n}^{(a)}\right)
 \in\UDrn{n}.
\]
Finally,
\[
 c_a(q^{-1})=q^{-2a(a+1)}c_a(q).
\]
Applying \eqref{eq:a2n-Omega-stability} to the positive-root assertion gives
\[
 q^{-2a(a+1)}c_aF_0^{(a)}\in\UDrn{n}.
\]
Multiplication by the unit \(q^{2a(a+1)}\) yields the negative-root
assertion.
\end{proof}

\subsection{Cyclotomic support and higher-rank recovery}
\label{subsec:cyclotomic-support-recovery}

\begin{theorem}[Higher-rank \(\Bstar\)-sufficiency]
\label{thm:higher-rank-Bstar}
Recall $B^*=\mathcal A[\Phi_{8r}(q)^{-1}\mid r\geq1]$. For every \(n\geq2\),
\begin{equation}
 \Bstar\UDrn{n}
 =\Bstar U_{\A}^{\mathrm{Lus},(n)}.
 \label{eq:higher-rank-Bstar}
\end{equation}
\end{theorem}

\begin{proof}
By \eqref{eq:eta-and-ca} and the cyclotomic factorization used in the proof of Theorem~\ref{thm:rank-one-Cstar}, one has
\[
 c_a=\prod_{h=1}^{a}(q^{4h}+1),
 \qquad
 q^{4h}+1
 =\prod_{\substack{d\mid h\\ h/d\ \mathrm{odd}}}\Phi_{8d}(q).
\]
Thus every nonunit irreducible factor of \(c_a\) is some
\(\Phi_{8d}(q)\), so \(c_a\in(\Bstar)^{\times}\).
Proposition~\ref{prop:scaled-affine-divided-powers} therefore gives
\[
 E_0^{(a)},F_0^{(a)}\in\Bstar\UDrn{n}
 \qquad(a\geq0).
\]
Applying Lemma~\ref{lem:affine-recovery-criterion} with
\(\mathscr X=\UDrn{n}\) and \(B=\Bstar\) proves
\eqref{eq:higher-rank-Bstar}.
\end{proof}
This completes the sufficiency direction in twisted type \(A_{2n}^{(2)}\).Under the classical specialization \(q\mapsto1\), \(E_0\mapsto e_0\),
we have \(c_a(1)=2^a\). Thus \(c_aE_0^{(a)}\) specializes to
\((2e_0)^{(a)}\), recovering the divided powers of the classical
scaled roots in \cite[Remark~28]{P}.
\section{Coefficient ideals, root-degree detection, and parity separation
in modified form}
\label{sec:coefficient-ideals}

We now turn
to necessity.  The present section isolates the common mechanism used in
Sections~\ref{sec:scaled-root-datum}
and~\ref{sec:endpoint-obstructions}.
The successive implications are supplied by
\cref{lem:simple-block-degree,lem:real-current-parity-propagation,%
lem:parity-coefficient-obstruction,lem:coefficient-prime-invertible}.

For $n,r\geq1$, define
\begin{equation}
 \mathfrak c_{n,r}
 :=(\UDrn{n}:_{\mathcal A}E_0^{(r)})
 =\{a\in\mathcal A\mid aE_0^{(r)}\in \UDrn{n}\}.
 \steptag{8.1}\label{eq:coefficient-ideal}
\end{equation}
Since \(U_{\mathcal A}^{\mathrm{Dr},(n)}\) is an
\(\mathcal A\)-submodule of \(U_{\mathcal A}^{\mathrm{Lus},(n)}\), the quotient
\[
 U_{\mathcal A}^{\mathrm{Lus},(n)}
 \big/
 U_{\mathcal A}^{\mathrm{Dr},(n)}
\]
is an \(\mathcal A\)-module. Writing \([E_0^{(r)}]\) for the class of
\(E_0^{(r)}\) in this quotient, we have
\[
 \mathfrak c_{n,r}
 =
 \operatorname{Ann}_{\mathcal A}\!\bigl([E_0^{(r)}]\bigr).
\]
Hence \(\mathfrak c_{n,r}\) is an \(\mathcal A\)-ideal.

Unlike the explicit multipliers \(d_r\) and \(c_r\), the ideal
\(\mathfrak c_{n,r}\) records every coefficient that recovers
\(E_0^{(r)}\) from the bare real-current core.

\begin{lemma}
\label[lemma]{lem:coefficient-prime-invertible}
For every ordinary localization
\(B=S^{-1}\mathcal A\subseteq\mathcal K\) and every
\((m,n,r)\in\mathbb Z_{\geq1}^{3}\),
if \(\mathfrak c_{n,r}\subseteq(\Phi_m(q))\) and the integral
forms $\UDrn{n}$ and $U_{\mathcal A}^{\mathrm{Lus},(n)}$ coincide over \(B\), then \(\Phi_m(q)\) is invertible in \(B\).
\end{lemma}

\begin{proof}
Since
\[
 E_0^{(r)}
 \in B U_{\mathcal A}^{\mathrm{Lus},(n)}
 =B\UDrn{n}
 =S^{-1}\UDrn{n},
\]
denominator clearing gives
\[
 \exists\,s\in S:
 \qquad
 sE_0^{(r)}\in\UDrn{n}.
\]
Hence
\[
 s\in S\cap\mathfrak c_{n,r}
 \subseteq(\Phi_m(q)),
 \qquad
 \exists\,b\in\mathcal A:
 \quad s=\Phi_m(q)b.
\]
Since \(s^{-1}\in B\),
\[
 \Phi_m(q)^{-1}=bs^{-1}\in B,
 \qquad
 \Phi_m(q)\in B^\times.
\]
\end{proof}

We next fix the grading conventions for modified quantum groups.  Let
$\dot U_{\mathcal A}$ be the Lusztig modified integral form for the original regular \(A_{2n}^{(2)}\) root datum fixed in
\eqref{eq:a2n-root-datum}.  If
$z\in U_{\mathcal A}^{\mathrm{Lus},(n)}$ is homogeneous of root degree
$\nu$, first interpret in $\dot U_{\mathcal A}$ the identity
\begin{equation}
 z1_\lambda=1_{\lambda+\nu}z1_\lambda,
 \steptag{8.2}\label{eq:modified-homogeneous-block}
\end{equation}
and only then specialize at a root of unity.  We do not require the
ordinary form to embed globally into the modified form after
specialization.

Let \(\ell_i\in\mathbb Z_{>0}\).  Denote the weight lattice, root
lattice, and simple roots of a scaled root datum by
\[
 X^\sharp,
 \qquad
 Q^\sharp,
 \qquad
 \alpha_i^\sharp=\ell_i\alpha_i,
\]
respectively.

For
\[
 \epsilon\in
 \operatorname{Hom}_{\mathbb Z}
 \bigl(Q^\sharp,\mathbb Z/2\mathbb Z\bigr),
 \qquad
 \bar e\in\mathbb Z/2\mathbb Z,
\]
put
\begin{equation}
 \mathcal H_{\bar e}
 :=
 \bigoplus_{\substack{\lambda,\mu\in X^\sharp\\
                       \mu-\lambda\in Q^\sharp,\ 
                       \epsilon(\mu-\lambda)=\bar e}}
 1_\mu^\sharp\dot U^\sharp1_\lambda^\sharp.
 \steptag{8.3}\label{eq:target-parity-spaces}
\end{equation}
The modified block decomposition
\cite[Sect.~23.1.2 and Lemma~23.2.2(b)]{L} gives
\begin{equation}
 \dot U^\sharp
 =\bigoplus_{\substack{\lambda,\mu\in X^\sharp\\
                       \mu-\lambda\in Q^\sharp}}
  1_\mu^\sharp\dot U^\sharp1_\lambda^\sharp
 =\mathcal H_{\bar0}\oplus\mathcal H_{\bar1},
 \qquad
 \mathcal H_{\bar e}\mathcal H_{\bar f}
 \subseteq\mathcal H_{\overline{e+f}}.
 \steptag{8.4}\label{eq:target-parity-grading}
\end{equation}
Indeed,
\[
 \bigl(1_\mu^\sharp\dot U^\sharp1_\lambda^\sharp\bigr)
 \bigl(1_\rho^\sharp\dot U^\sharp1_\sigma^\sharp\bigr)
 =0
 \qquad(\lambda\neq\rho),
\]
whereas, for \(\lambda=\rho\),
\[
 \bigl(1_\mu^\sharp\dot U^\sharp1_\lambda^\sharp\bigr)
 \bigl(1_\lambda^\sharp\dot U^\sharp1_\sigma^\sharp\bigr)
 \subseteq
 1_\mu^\sharp\dot U^\sharp1_\sigma^\sharp,
 \qquad
 \epsilon(\mu-\sigma)
 =\epsilon(\mu-\lambda)+\epsilon(\lambda-\sigma).
\]

For every scalar extension \(\Bbbk\otimes_{\mathcal A}M\) used below, put
\[
 \overline y:=1\otimes y
 \in\Bbbk\otimes_{\mathcal A}M
 \qquad(y\in M).
\]
This convention is distinct from an overline on an integer, which denotes
its class in \(\mathbb Z/2\mathbb Z\).

\begin{lemma}
\label[lemma]{lem:simple-block-degree}
For  $\pi\in\{\Phi_h(q)\mid h\in\mathbb Z_{\geq1}\}$, let $\Bbbk:=\operatorname{Frac}\bigl(\mathcal A/(\pi)\bigr), 
 \dot U_{\Bbbk}:=\Bbbk\otimes_{\mathcal A}\dot U_{\mathcal A}$,  and
let \(\dot U^\sharp\) be the modified \(\Bbbk\)-form attached to the
scaled root datum, and let
\[
 \mathcal F\in
 \operatorname{Hom}_{\Bbbk\text{-}\mathrm{alg}}
 \bigl(\dot U_{\Bbbk},\dot U^\sharp\bigr).
\]
Assume
\begin{equation}
 \forall\,\mu\in X,
 \qquad
 \mathcal F(\overline{1_\mu})
 =
 \begin{cases}
  1_\mu^\sharp,&\mu\in X^\sharp,\\
  0,&\mu\notin X^\sharp.
 \end{cases}
 \steptag{8.5}\label{eq:simple-block-map}
\end{equation}
Then
let \(\nu\in Q\), \(z\in
\bigl(U_{\mathcal A}^{\mathrm{Lus},(n)}\bigr)_\nu\),
and \(\lambda\in X\).
If \(\mathcal F(\overline{z1_\lambda})\neq0\), then
\(\nu\in Q^\sharp\), and
\(\mathcal F(\overline{z1_\lambda})\) is a nonzero homogeneous
element of \(\dot U^\sharp\) of degree \(\nu\).
\end{lemma}

\begin{proof}
Let
\[
 y:=\mathcal F(\overline{z1_\lambda}).
\]
By \eqref{eq:modified-homogeneous-block},
\[
 \overline{z1_\lambda}
 =\overline{1_{\lambda+\nu}z1_\lambda1_\lambda},
 \qquad
 y
 =\mathcal F(\overline{1_{\lambda+\nu}})
  y\mathcal F(\overline{1_\lambda}).
\]
Hence
\[
 y\neq0
 \quad\Longrightarrow\quad
 \lambda,\lambda+\nu\in X^\sharp,
 \qquad
 0\neq y
 \in
 1_{\lambda+\nu}^\sharp
 \dot U^\sharp
 1_\lambda^\sharp.
\]
The modified PBW block decomposition gives
\[
 1_\mu^\sharp\dot U^\sharp1_\lambda^\sharp\neq0
 \quad\Longrightarrow\quad
 \mu-\lambda\in Q^\sharp,
 \qquad
 1_\mu^\sharp\dot U^\sharp1_\lambda^\sharp
 \subseteq(\dot U^\sharp)_{\mu-\lambda}.
\]
Taking \(\mu=\lambda+\nu\) yields
\[
 \nu\in Q^\sharp,
 \qquad
 0\neq y\in(\dot U^\sharp)_\nu.
\]
\end{proof}

For \(n\in\mathbb Z_{\geq1}\), let
\[
 \begin{aligned}
 \mathscr C_n
 :={}&
 \left\{
  (x_{i,k}^{\varepsilon})^{(s)}
  \ \middle|\
  1\leq i\leq n,\ k\in\mathbb Z,\
  s\in\mathbb Z_{\geq0},\
  \varepsilon\in\{1,-1\}
 \right\}\\
 &\cup\{K_\delta^{\pm1}\}
 \cup\{K_i^{\pm1}\mid1\leq i\leq n\}.
 \end{aligned}
\]
be the set of generators of \(U_{\mathcal A}^{\mathrm{Dr},(n)}\)
\begin{lemma}
\label[lemma]{lem:real-current-parity-propagation}
Let \(\Bbbk\) be an \(\mathcal A\)-algebra, let
\[
 D=D_{\bar0}\oplus D_{\bar1},
 \qquad
 D_{\bar e}D_{\bar f}\subseteq D_{\overline{e+f}},
\]
be a \(\mathbb Z/2\mathbb Z\)-graded \(\Bbbk\)-algebra, and let
\[
 \varphi\in
 \operatorname{Hom}_{\Bbbk\text{-}\mathrm{alg}}
 \left(
  \Bbbk\otimes_{\mathcal A}\dot U_{\mathcal A},D
 \right).
\]
If
any nonzero image \(\varphi(\overline{u1_\lambda})\),
with \(\lambda\in X\) and \(u\in\mathscr C_n\),
belongs to \(D_{\bar0}\),
then
\[
 \varphi(\overline{\UDrn{n}1_0})\subseteq D_{\bar0}.
\]
\end{lemma}

\begin{proof}
For
\[
 M=u_1\cdots u_t,
 \qquad
 u_j\in\mathscr C_n,
 \qquad
 \gamma_j:=\deg u_j,
 \qquad
 \lambda_j:=\sum_{p=j+1}^{t}\gamma_p,
\]
one has
\[
 M1_0=(u_11_{\lambda_1})\cdots(u_t1_0).
\]
Thus
\[
 \varphi(\overline{M1_0})
 =0
 \quad\text{or}\quad
 \varphi(\overline{M1_0})
 \in(D_{\bar0})^t\subseteq D_{\bar0}.
\]
For \(t=0\), use \(1=(x_{1,0}^{+})^{(0)}\in\mathscr C_n\).
Since
\[
 \UDrn{n}
 =\operatorname{span}_{\mathcal A}
 \{u_1\cdots u_t\mid t\geq0,\ u_j\in\mathscr C_n\},
\]
\(\Bbbk\)-linearity yields the assertion.
\end{proof}

\begin{lemma}
\label[lemma]{lem:parity-coefficient-obstruction}
Let $
 n,r\in\mathbb Z_{\geq1},
 \pi\in\{\Phi_h(q)\mid h\in\mathbb Z_{\geq1}\},
 \Bbbk:=\operatorname{Frac}\bigl(\mathcal A/(\pi)\bigr)$, 
and let
\[
 \chi:
 \Bbbk\otimes_{\mathcal A}
 \bigl(U_{\mathcal A}^{\mathrm{Lus},(n)}1_0\bigr)
 \longrightarrow D_{\bar0}\oplus D_{\bar1}
\]
be \(\Bbbk\)-linear, where the sum is internal.  Then
\[
 \left.
 \begin{gathered}
  \chi(\overline{\UDrn{n}1_0})\subseteq D_{\bar0},\\
  0\neq\chi(\overline{E_0^{(r)}1_0})\in D_{\bar1}
 \end{gathered}
 \right\}
 \quad\Longrightarrow\quad
 \mathfrak c_{n,r}\subseteq(\pi).
 \steptag{8.6}\label{eq:parity-coefficient-inclusion}
\]
\end{lemma}

\begin{proof}
For \(a\in\mathfrak c_{n,r}\), put
\[
 c:=aE_0^{(r)}\in\UDrn{n}.
\]
Then
\[
 \chi(\overline{c1_0})
 =\bar a\,\chi(\overline{E_0^{(r)}1_0})
 \in D_{\bar0}.
\]
If \(a\notin(\pi)\), then \(\bar a\in\Bbbk^\times\), whence
\[
 0\neq\chi(\overline{c1_0})
 \in D_{\bar0}\cap D_{\bar1}=\{0\},
\]
a contradiction.  Therefore
\[
 a\in(\pi),
 \qquad
 \mathfrak c_{n,r}\subseteq(\pi).
\]
\end{proof}

\section{The scaled root datum and cyclotomic obstructions for
\texorpdfstring{$r\geq2$}{r>=2}}
\label{sec:scaled-root-datum}

The preceding lemmas are quite general, now we apply them to our case of the
nonendpoint cyclotomic primes by means of Lusztig's modified quantum
Frobenius. Recall the coefficient ideals \(\mathfrak c_{n,a}\) from
\eqref{eq:coefficient-ideal}. In this section, we use scaled root data and modified
quantum Frobenius to prove
\[
 \mathfrak c_{n,r}\subseteq(\Phi_{8r}(q))
 \quad(n\geq1,\ r\geq2),
 \qquad
 \mathfrak c_{1,r}\subseteq(\Phi_{4r}(q))
 \quad(r\geq3\text{ odd}).
\]
The scaled-data formulas below will also be used at \(r=1\) in
Section~\ref{sec:endpoint-obstructions}. We therefore retain that value
in the definitions and lattice calculations; the Frobenius applications
in the present section are restricted to the higher indices stated above.

\subsection{The scaled datum at \texorpdfstring{$\Phi_{8r}(q)$}{Phi8r(q)}}
\label{subsec:phi8r-scaled-datum}

Fix $r\geq 1$. Consider
$\Bbbk_r^{(8)}:=\operatorname{Frac}\bigl(\mathcal A/(\Phi_{8r}(q))\bigr)$ and fix
 \begin{equation}
 \zeta_r:=q\bmod \Phi_{8r}(q).
 \label{eq:kr-zeta-r}
\end{equation}
Then
\[
 \operatorname{ord}(\zeta_r)=8r,
 \qquad
 \operatorname{ord}(\zeta_r^2)=4r.
\]
Write
\[
 (d_0,d_1,\ldots,d_n)
 =(4,2,\ldots,2,1).
\]
The local quantum parameter is $q_i=q^{d_i}$.  Hence Lusztig's
even-order Frobenius integer is
\begin{equation}
 \ell_i
 :=\operatorname{ord}\bigl((q_i\bmod\Phi_{8r}(q))^2\bigr)
 =\operatorname{ord}(\zeta_r^{2d_i})
 =\frac{8r}{\gcd(8r,2d_i)}
 =\frac{4r}{\gcd(4r,d_i)}.
 \label{eq:ell-i-phi8r}
\end{equation}
Therefore
\begin{equation}
 \ell_0=r,
 \qquad
 \ell_i=2r\quad(1\leq i<n),
 \qquad
 \ell_n=4r.
 \label{eq:ell-values-phi8r}
\end{equation}
The middle family is empty when $n=1$.

Retain the regular \(A_{2n}^{(2)}\) root datum
\((X,Y,I,\langle\ ,\ \rangle)\) fixed in \eqref{eq:a2n-root-datum} as the source root datum.  Define
\begin{align}
 X^{\bullet}
 &:=\{\lambda\in X\mid
       \langle h_i,\lambda\rangle\in\ell_i\mathbb Z
       \text{ for every }i\in I\},
 \label{eq:scaled-X-lattice}\\
 Y^{\bullet}
 &:=\operatorname{Hom}_{\mathbb Z}(X^{\bullet},\mathbb Z),
 \label{eq:scaled-lattices}\\
 \alpha_i^{\bullet}&:=\ell_i\alpha_i,
 \quad \langle h_i^{\bullet},\lambda\rangle
 :=\ell_i^{-1}\langle h_i,\lambda\rangle
 \qquad(\lambda\in X^{\bullet}).
 \label{eq:scaled-simple-data}
\end{align}
If
\[
 i\cdot j:=d_i a_{ij}=d_j a_{ji}
\]
is the symmetric form on $\mathbb Z[I]$, put
\begin{equation}
 i\mathbin{\cdot_{\bullet}}j:=\ell_i\ell_j(i\cdot j),
 \qquad
 d_i^{\bullet}:=d_i\ell_i^2.
 \label{eq:scaled-symmetric-form}
\end{equation}
The coefficients $d_i^{\bullet}$ are the integral symmetrizers of the
scaled Cartan datum.

\begin{lemma}
\label[lemma]{lem:scaled-datum-regular}
With the data of
\eqref{eq:scaled-X-lattice}--\eqref{eq:scaled-symmetric-form},
\[
 \Bigl(
  X^{\bullet},Y^{\bullet},I,\langle\ ,\ \rangle,
  (\alpha_i^{\bullet})_{i\in I},(h_i^{\bullet})_{i\in I}
 \Bigr)
 \quad\text{is a regular root datum}.
\]
For every \((i,j)\in I^2\),
\begin{equation}
 a_{ij}^{\bullet}
 :=\langle h_i^{\bullet},\alpha_j^{\bullet}\rangle
 =\ell_i^{-1}\ell_j a_{ij}=a_{ji}.
 \label{eq:scaled-cartan-transpose}
\end{equation}
Moreover,
\begin{equation}
 L_r^{\bullet}
 :=\bigoplus_{i\in I}\mathbb Z\alpha_i^{\bullet}
 =r\mathbb Z\alpha_0
   \oplus\bigoplus_{i=1}^{n-1}2r\mathbb Z\alpha_i
   \oplus4r\mathbb Z\alpha_n.
 \label{eq:scaled-root-lattice}
\end{equation}
\end{lemma}

\begin{proof}
By \eqref{eq:ell-i-phi8r},
\[
 \ell_i
 =\min\{m\in\mathbb Z_{\geq1}\mid md_i\in4r\mathbb Z\}.
\]
Thus \eqref{eq:scaled-X-lattice}--\eqref{eq:scaled-symmetric-form}
are Lusztig's \(4r\)-modified root datum
\cite[\S\S2.2.4--2.2.5]{L}.  Since the source datum is regular, so is
the modified datum.  Moreover,
\[
 d_i\ell_i=4r
 \qquad(i\in I),
\]
and hence
\[
 a_{ij}^\bullet
 =\ell_i^{-1}\ell_j a_{ij}
 =\frac{d_i}{d_j}a_{ij}
 =a_{ji}.
\]
Finally, \eqref{eq:ell-values-phi8r} gives
\[
 L_r^\bullet
 =r\mathbb Z\alpha_0
  \oplus\bigoplus_{i=1}^{n-1}2r\mathbb Z\alpha_i
  \oplus4r\mathbb Z\alpha_n,
\]
\end{proof}

Define
\begin{equation}
 \epsilon_r:L_r^{\bullet}\longrightarrow\mathbb Z/2\mathbb Z,
 \qquad
 \epsilon_r\!\left(\sum_{i=0}^{n}b_i\alpha_i^{\bullet}\right)
 :=\overline{b_0}.
 \label{eq:epsilon-r}
\end{equation}

\begin{lemma}[Uniform evenness of real-current degrees]
\label[lemma]{lem:uniform-current-evenness}
Let \(s\in\mathbb Z_{\geq0}\), \(k\in\mathbb Z\), \(\varepsilon\in\{1,-1\}\), and \(1\leq i\leq n\). Then \(\deg\!\left((x_{i,k}^{\varepsilon})^{(s)}\right)=s(k\delta+\varepsilon\alpha_i)\), and if \(s(k\delta+\varepsilon\alpha_i)\in L_r^{\bullet}\), then \(\epsilon_r\bigl(s(k\delta+\varepsilon\alpha_i)\bigr)=\overline0\).
\end{lemma}

\begin{proof}
Let
\[
 \nu:=s(k\delta+\varepsilon\alpha_i).
\]
Since
\[
 \delta=\alpha_0+2\alpha_1+\cdots+2\alpha_n,
\]
we have
\begin{equation*}
 \nu
 =sk\alpha_0
  +\sum_{\substack{1\leq j\leq n\\j\neq i}}2sk\alpha_j
  +s(2k+\varepsilon)\alpha_i.
\end{equation*}
If $\nu\in L_r^{\bullet}$, its $\alpha_0^{\bullet}=r\alpha_0$
coordinate is
\[
 b_0=\frac{sk}{r}\in\mathbb Z.
\]
It remains to prove $b_0\in2\mathbb Z$.

If $i<n$, the $\alpha_n$-coordinate of $\nu$ is $2sk$.  Hence
\[
 \nu\in L_r^{\bullet}
 \quad\Longrightarrow\quad
 4r\mid2sk
 \quad\Longrightarrow\quad
 2r\mid sk
 \quad\Longrightarrow\quad
 b_0=\frac{sk}{r}\in2\mathbb Z.
\]

Suppose $i=n$.  The $\alpha_0$- and $\alpha_n$-coordinates give
\begin{equation*}
 \nu\in L_r^{\bullet}
 \quad\Longrightarrow\quad
 r\mid sk,
 \qquad
 4r\mid s(2k+\varepsilon).
\end{equation*}
Since
\[
 \gcd(k,2k+\varepsilon)=\gcd(k,\varepsilon)=1,
\]
\[
 r\mid sk,\quad r\mid s(2k+\varepsilon)
 \quad\Longrightarrow\quad
 r\mid s.
\]
Write $s=rt$.  Then
\[
 4r\mid rt(2k+\varepsilon)
 \quad\Longrightarrow\quad
 4\mid t(2k+\varepsilon).
\]
Because $2k+\varepsilon$ is odd,
\[
 \gcd(4,2k+\varepsilon)=1
 \quad\Longrightarrow\quad
 4\mid t.
\]
Consequently
\[
 b_0=\frac{sk}{r}=tk\in4\mathbb Z\subseteq2\mathbb Z.
\]
Thus $\epsilon_r(\nu)=\overline0$ in both cases.  The same chains include
$s=0$ and $k=0$.
\end{proof}

Let's explain why the affine-coordinate parity is chosen.
The parity \(\epsilon_r\) detects the affine coordinate: the Frobenius
image of \(E_0^{(r)}1_0\) has \(\alpha_0^\bullet\)-coordinate one, whereas
the lattice divisibilities force every surviving real-current degree to
have even affine coordinate.  Lemma~\ref{lem:uniform-current-evenness}
remains valid for \(r=1\).  The restriction \(r\geq2\) below comes solely from the numerical inequalities \(\ell_i\geq1-a_{ij}\) required in Lusztig's Frobenius theorem for \(i\neq j\) with \(\ell_j\geq2\).  This explains why \(\Phi_8\) requires a different detector rather than a different parity.

We now assume $r\geq2$.  Let $\dot U_{\mathcal A}^{\Lus,(n)}$ be the
source modified integral form, and let
$\dot U_{\mathcal A,r}^{\bullet}$ be the modified integral form attached to
the scaled datum.  Let
\begin{equation}
 \dot U_{\zeta_r}
 :=\Bbbk_r^{(8)}\otimes_{\mathcal A,\,q\mapsto\zeta_r}
      \dot U_{\mathcal A}^{\Lus,(n)},
 \qquad
 \dot U_r^{\bullet}
 :=\Bbbk_r^{(8)}\otimes_{\mathcal A,\,q\mapsto\zeta_r}
      \dot U_{\mathcal A,r}^{\bullet}.
 \label{eq:source-target-specializations}
\end{equation}
For \(\bar e\in\mathbb Z/2\mathbb Z\), let
\begin{equation}
 \mathcal H_{\bar e}^{(r)}
 :=
 \bigoplus_{\substack{\lambda,\mu\in X^\bullet\\
                       \mu-\lambda\in L_r^\bullet,\ 
                       \epsilon_r(\mu-\lambda)=\bar e}}
 1_\mu^\bullet\dot U_r^\bullet1_\lambda^\bullet.
 \label{eq:higher-target-parity-spaces}
\end{equation}
Then \eqref{eq:target-parity-grading} gives
\begin{equation}
 \dot U_r^\bullet
 =\mathcal H_{\bar0}^{(r)}\oplus\mathcal H_{\bar1}^{(r)},
 \qquad
 \mathcal H_{\bar e}^{(r)}\mathcal H_{\bar f}^{(r)}
 \subseteq\mathcal H_{\overline{e+f}}^{(r)}.
 \label{eq:higher-target-parity-grading}
\end{equation}

The following is \cite[\S\S35.1.1--35.1.3;
Thm.~35.1.9]{L}, written in the parameter (q).

\begin{frobeniuslemma}[Lusztig's modified quantum Frobenius criterion in the
$q$-parameter]
\label[lemma]{lem:modified-quantum-frobenius}
Let \(\Bbbk\) be a field and \(\bar q\in\Bbbk^\times\).  Assume and set
\[
 \operatorname{char}\Bbbk=0,
 \qquad
 N:=\operatorname{ord}(\bar q^2)\in2\mathbb Z_{>0},
 \qquad
 \mathcal A_N':=\mathcal A/(\Phi_{2N}(q)),
\]
and assume
\[
 \left(\mathcal A\xrightarrow{\ q\mapsto\bar q\ }\Bbbk\right)
 =
 \left(
  \mathcal A\twoheadrightarrow\mathcal A_N'
  \longrightarrow\Bbbk
 \right).
\]
Let
\[
 (d_i)_{i\in I}\in\mathbb Z_{>0}^{I}
 \quad\text{(the source symmetrizers)},
 \qquad
 d_i a_{ij}=d_j a_{ji}
 \quad\bigl((i,j)\in I^2\bigr),
\]
and, for every \(i\in I\), define
\[
 \ell_i
 :=\min\{m\in\mathbb Z_{\geq1}\mid md_i\in N\mathbb Z\}
 =\frac{N}{\gcd(N,d_i)}
 =\operatorname{ord}(\bar q^{2d_i}).
\]
Let \(\dot U_{\Bbbk}\) be the specialized source modified form and
\(\dot U_{\Bbbk}^{\bullet}\) the specialized modified form for the
associated \((\ell_i)_{i\in I}\)-scaled root datum.  Assume
\begin{enumerate}[label=\textup{(\roman*)}]
 \item
 There exists a function \(o:I\longrightarrow\{1,-1\}\) such that
 \(o(i)=-o(j)\) whenever \(i\neq j\) and \(a_{ij}\neq0\).

 \item
 For all \(i,j\in I\) with \(i\neq j\) and \(\ell_j\geq2\), one has
 \begin{equation}
  \ell_i\geq1-a_{ij}.
  \label{eq:frobenius-size-condition}
 \end{equation}
\end{enumerate}
Then there exists a \(\Bbbk\)-algebra homomorphism
\(\operatorname{Fr}:\dot U_{\Bbbk}\longrightarrow
\dot U_{\Bbbk}^{\bullet}\).
For every \(\lambda\in X\), it sends \(1_\lambda\) to
\(1_\lambda^{\bullet}\) if \(\lambda\in X^{\bullet}\), and to zero
otherwise. Moreover, for all \(i\in I\), \(s\geq0\), and
\(\lambda\in X\), with \(s\) an integer, it satisfies
\[
\begin{aligned}
 \operatorname{Fr}(E_i^{(s)}1_\lambda)
 &=
 \begin{cases}
  E_i^{\bullet,(s/\ell_i)}1_\lambda^{\bullet},
    &\lambda\in X^{\bullet},\ \ell_i\mid s,\\
  0,&\text{otherwise},
 \end{cases}\\[4pt]
 \operatorname{Fr}(F_i^{(s)}1_\lambda)
 &=
 \begin{cases}
  F_i^{\bullet,(s/\ell_i)}1_\lambda^{\bullet},
    &\lambda\in X^{\bullet},\ \ell_i\mid s,\\
  0,&\text{otherwise}.
 \end{cases}
\end{aligned}
\]
\end{frobeniuslemma}

\medskip
\noindent\textbf{Application to the $\Phi_{8r}(q)$ specialization.}
\begin{align}
 \forall\,(n,r)\in
 \mathbb Z_{\geq1}\times\mathbb Z_{\geq2},
 \qquad
 \exists\,\operatorname{Fr}_{n,r}\in
 \operatorname{Hom}_{\Bbbk_r^{(8)}\text{-}\mathrm{alg}}
 \bigl(\dot U_{\zeta_r},\dot U_r^{\bullet}\bigr):
 \label{eq:higher-frobenius}\\
 \forall\,\lambda\in X,
 \qquad
 \operatorname{Fr}_{n,r}(1_\lambda)=
 \begin{cases}
  1_\lambda^{\bullet},&\lambda\in X^{\bullet},\\
  0,&\lambda\notin X^{\bullet},
 \end{cases}
 \label{eq:higher-frobenius-idempotents}\\
 \operatorname{Fr}_{n,r}(E_0^{(r)}1_0)
 =E_0^{\bullet}1_0^{\bullet}\neq0.
 \label{eq:E0-higher-frobenius-nonzero}
\end{align}
\par

\begin{proof}
Here
\[
 N=\operatorname{ord}(\zeta_r^2)=4r,
 \qquad
 2N=8r
\]
Since \(\xi_{8r}\in\mathbb C\) is a primitive \(8r\)-th root of unity,
\[
 \mathcal A/(\Phi_{8r}(q))
 \cong
 \mathbb Z[\xi_{8r},\xi_{8r}^{-1}]
 =
 \mathbb Z[\xi_{8r}].
\]
This is an integral domain of characteristic zero.  Consequently,
\[
 \Bbbk_r^{(8)}
 =
 \operatorname{Frac}\!\left(\mathcal A/(\Phi_{8r}(q))\right)
 \cong
 \mathbb Q(\xi_{8r})
\]
is a cyclotomic field of characteristic zero.  Thus the field hypothesis in the criterion holds.  Moreover,
$N$ is even and
\[
 \mathcal A\longrightarrow\Bbbk_r^{(8)}
\]
factors through
$\mathcal A_N'=\mathcal A/(\Phi_{8r}(q))$ by
\eqref{eq:kr-zeta-r}.
Set
\[
 o_{\mathrm{Fr}}(i):=(-1)^{n-i}
 \qquad(i\in I).
\]
Then
\[
 i\neq j,\quad a_{ij}\neq0
 \quad\Longrightarrow\quad
 |i-j|=1
 \quad\Longrightarrow\quad
 o_{\mathrm{Fr}}(i)=-o_{\mathrm{Fr}}(j).
\]
Since \(r\geq2\), \eqref{eq:ell-values-phi8r} and the Cartan matrices
fixed in \cref{sec:a2n-root-data} give, for \(i\neq j\),
\[
 1-a_{ij}
 \leq
 \begin{cases}
  2,&i=0,\\
  3,&1\leq i<n \\
  5,&i=n
 \end{cases}
 \leq
 \begin{cases}
  r,&i=0,\\
  2r,&1\leq i<n,\\
  4r,&i=n
 \end{cases}
 =\ell_i.
\]
Thus conditions \textup{(i)}--\textup{(ii)} of the Frobenius criterion
hold, and the criterion gives \eqref{eq:higher-frobenius}--
\eqref{eq:higher-frobenius-idempotents}.  Since
\[
 0\in X^{\bullet},
 \qquad
 \ell_0=r,
\]
its positive block formula gives
\[
 \operatorname{Fr}_{n,r}(E_0^{(r)}1_0)
 =E_0^{\bullet,(r/\ell_0)}1_0^{\bullet}
 =E_0^{\bullet}1_0^{\bullet}.
\]
The final element is nonzero by the modified PBW basis
\cite[\S\S23.2]{L}; this proves
\eqref{eq:E0-higher-frobenius-nonzero}.
\end{proof}

\begin{theorem}
\label{thm:phi8r-coefficient-ideal}
For any integers \(n\geq1\) and \(r\geq2\), one has
\begin{equation}
 \mathfrak c_{n,r}\subseteq(\Phi_{8r}(q)).
 \label{eq:phi8r-coefficient-containment}
\end{equation}
\end{theorem}

\begin{proof}
Fix
\[
 (n,r)\in
 \mathbb Z_{\geq1}\times\mathbb Z_{\geq2},
 \qquad
 (\lambda,u)\in X\times\mathscr C_n,
 \qquad
 \nu:=\deg u.
\]
equation \eqref{eq:higher-frobenius} supplies the algebra
homomorphism required in \cref{lem:simple-block-degree}, while
\eqref{eq:higher-frobenius-idempotents} is precisely its
block-idempotent hypothesis \eqref{eq:simple-block-map}.  If
\[
 \operatorname{Fr}_{n,r}(\overline{u1_\lambda})\neq0
\]
and \(u\) is a real-current generator, then
\[
 \nu\in L_r^\bullet
 \quad\text{by }\cref{lem:simple-block-degree},
 \qquad
 \epsilon_r(\nu)=\overline0
 \quad\text{by }\cref{lem:uniform-current-evenness},
\]
and hence
\[
 0\neq
 \operatorname{Fr}_{n,r}(\overline{u1_\lambda})
 \in\mathcal H_{\bar0}^{(r)}.
\]
For a toral generator,
\[
 \nu=0,
 \qquad
 \operatorname{Fr}_{n,r}(\overline{u1_\lambda})
 \in\mathcal H_{\bar0}^{(r)}.
\]
Thus \cref{lem:real-current-parity-propagation} gives
\begin{equation}
 \operatorname{Fr}_{n,r}
 \bigl(\overline{\UDrn{n}1_0}\bigr)
 \subseteq\mathcal H_{\bar0}^{(r)}.
 \label{eq:higher-core-even}
\end{equation}
Moreover,
\[
 r\alpha_0=\alpha_0^\bullet\in L_r^\bullet,
 \qquad
 \epsilon_r(r\alpha_0)=\overline1,
\]
and \eqref{eq:E0-higher-frobenius-nonzero} gives
\[
 0\neq
 \operatorname{Fr}_{n,r}(\overline{E_0^{(r)}1_0})
 =E_0^\bullet1_0^\bullet
 \in\mathcal H_{\bar1}^{(r)}.
\]
Hence \cref{lem:parity-coefficient-obstruction}, applied to
\[
 \chi_{n,r}
 :=\left.\operatorname{Fr}_{n,r}\right|_{
 \Bbbk_r^{(8)}\otimes_{\mathcal A}
 (U_{\mathcal A}^{\mathrm{Lus},(n)}1_0)},
\]
gives
\[
 \mathfrak c_{n,r}\subseteq(\Phi_{8r}(q)),
\]
which is \eqref{eq:phi8r-coefficient-containment}.
\end{proof}

\subsection{The rank-one primes \texorpdfstring{$\Phi_{4r}(q)$}{Phi4r(q)} for odd
\texorpdfstring{$r\geq3$}{r at least 3}}
\label{subsec:rank-one-phi4r}

We now take $n=1$ and let $r\geq3$ be odd.  Let
\begin{equation}
 \Bbbk_r^{(4)}
 :=\operatorname{Frac}\bigl(\mathcal A/(\Phi_{4r}(q))\bigr),
 \qquad
 \xi_r:=q\bmod\Phi_{4r}(q).
 \label{eq:k4r-xi-r}
\end{equation}
Then
\[
 \operatorname{ord}(\xi_r)=4r,
 \qquad
 \operatorname{ord}(\xi_r^2)=2r.
\]
For $(d_0,d_1)=(4,1)$,
\[
 \ell_i=\operatorname{ord}(\xi_r^{2d_i})
       =\frac{4r}{\gcd(4r,2d_i)}.
\]
Since $r$ is odd,
\[
 \gcd(4r,8)=4,
 \qquad
 \gcd(4r,2)=2,
\]
and hence
\begin{equation}
 \ell_0=r,
 \qquad
 \ell_1=2r.
 \label{eq:rank-one-ell-phi4r}
\end{equation}

Set
\begin{equation}
 X_r^{\star}
 :=\{\lambda\in X\mid
       \langle h_i,\lambda\rangle\in\ell_i\mathbb Z
       \text{ for }i=0,1\},
 \qquad
 \alpha_i^{\star}:=\ell_i\alpha_i.
 \label{eq:rank-one-scaled-datum}
\end{equation}
Set also
\begin{equation}
 Y_r^{\star}:=\operatorname{Hom}_{\mathbb Z}(X_r^{\star},\mathbb Z),
 \qquad
 \langle h_i^{\star},\lambda\rangle
 :=\ell_i^{-1}\langle h_i,\lambda\rangle,
 \qquad
 L_r^{\star}:=r\mathbb Z\alpha_0\oplus2r\mathbb Z\alpha_1.
 \label{eq:rank-one-starred-complete-datum}
\end{equation}
If $i\cdot j=d_i a_{ij}=d_j a_{ji}$ is the source symmetric form, put
\begin{equation}
 i\mathbin{\cdot_{\star}}j:=\ell_i\ell_j(i\cdot j),
 \qquad
 d_i^{\star}:=d_i\ell_i^2.
 \label{eq:rank-one-starred-symmetric-form}
\end{equation}
The coefficients $d_i^{\star}$ are the integral symmetrizers of the
starred Cartan datum.

\begin{lemma}
\label[lemma]{lem:rank-one-scaled-datum-regular}
With the data of
\eqref{eq:rank-one-scaled-datum}--
\eqref{eq:rank-one-starred-symmetric-form},
\[
 \Bigl(
  X_r^{\star},Y_r^{\star},\{0,1\},\langle\ ,\ \rangle,
  (\alpha_i^{\star})_{i=0,1},(h_i^{\star})_{i=0,1}
 \Bigr)
 \quad\text{is a regular root datum},
\]
and
\[
 (a_{ij}^{\star})_{i,j=0,1}
 =\begin{pmatrix}2&-2\\-2&2\end{pmatrix},
 \qquad
 (d_0^{\star},d_1^{\star})=(4r^2,4r^2),
 \qquad
 \bigoplus_{i=0}^{1}\mathbb Z\alpha_i^{\star}
 =L_r^{\star}\subseteq X_r^{\star}.
\]
\end{lemma}
\begin{proof}
Since \(r\) is odd,
\[
 \ell_i
 =\min\{m\in\mathbb Z_{\geq1}\mid md_i\in2r\mathbb Z\},
 \qquad
 (\ell_0,\ell_1)=(r,2r),
\]
Hence \eqref{eq:rank-one-scaled-datum}--
\eqref{eq:rank-one-starred-symmetric-form} are Lusztig's
\(2r\)-modified root datum \cite[\S\S2.2.4--2.2.5]{L}, and regularity
follows from that of the source datum.  Directly,
\[
 (a_{ij}^\star)_{i,j=0,1}
 =\left(\ell_i^{-1}\ell_j a_{ij}\right)_{i,j=0,1}
 =\begin{pmatrix}2&-2\\-2&2\end{pmatrix}.
\]
\[
 (d_0^\star,d_1^\star)
 =(4r^2,4r^2),
 \qquad
 \bigoplus_{i=0}^{1}\mathbb Z\alpha_i^\star
 =r\mathbb Z\alpha_0\oplus2r\mathbb Z\alpha_1
 =L_r^\star.
\]
\end{proof}

Define
\begin{equation}
 \epsilon_r^{\star}:L_r^{\star}\longrightarrow\mathbb Z/2\mathbb Z,
 \qquad
 \epsilon_r^{\star}
 (b_0\alpha_0^{\star}+b_1\alpha_1^{\star})=\overline{b_0}.
 \label{eq:rank-one-starred-parity}
\end{equation}
Let $\dot U_{\mathcal A,r}^{\star}$ be the modified integral form for the
starred datum, and put
\begin{equation}
 \dot U_r^{\star}
 :=\Bbbk_r^{(4)}\otimes_{\mathcal A,\,q\mapsto\xi_r}
      \dot U_{\mathcal A,r}^{\star}.
 \label{eq:rank-one-scaled-target}
\end{equation}
For \(\bar e\in\mathbb Z/2\mathbb Z\), put
\begin{equation}
 \mathcal G_{\bar e}^{(r)}
 :=
 \bigoplus_{\substack{\lambda,\mu\in X_r^\star\\
                       \mu-\lambda\in L_r^\star,\ 
                       \epsilon_r^\star(\mu-\lambda)=\bar e}}
 1_\mu^\star\dot U_r^\star1_\lambda^\star.
 \label{eq:rank-one-target-parity-spaces}
\end{equation}
Then
\begin{equation}
 \dot U_r^\star
 =\mathcal G_{\bar0}^{(r)}\oplus\mathcal G_{\bar1}^{(r)},
 \qquad
 \mathcal G_{\bar e}^{(r)}\mathcal G_{\bar f}^{(r)}
 \subseteq\mathcal G_{\overline{e+f}}^{(r)}.
 \label{eq:rank-one-target-parity-grading}
\end{equation}

\medskip
\noindent\textbf{Application to the rank-one $\Phi_{4r}(q)$ specialization.}
\begin{align}
 &\forall\,r\in(2\mathbb Z+1)\cap\mathbb Z_{\geq3},
 \qquad
 \begin{gathered}
 \dot U_{\xi_r}^{(1)}
 :=\Bbbk_r^{(4)}\otimes_{\mathcal A,\,q\mapsto\xi_r}
       \dot U_{\mathcal A}^{\Lus,(1)},\\[-1pt]
 \exists\,\operatorname{Fr}_{r}^{(4)}\in
 \operatorname{Hom}_{\Bbbk_r^{(4)}\text{-}\mathrm{alg}}
 \bigl(\dot U_{\xi_r}^{(1)},\dot U_r^{\star}\bigr):
 \end{gathered}
 \label{eq:rank-one-phi4r-frobenius}\\
 &
 \forall\,\lambda\in X,
 \qquad
 \operatorname{Fr}_{r}^{(4)}(1_\lambda)=
 \begin{cases}
  1_\lambda^{\star},&\lambda\in X_r^{\star},\\
  0,&\lambda\notin X_r^{\star},
 \end{cases}
 \notag\\
 &
 \operatorname{Fr}_{r}^{(4)}(E_0^{(r)}1_0)
 =E_0^{\star}1_0^{\star}\neq0.
 \notag
\end{align}
\par

\begin{proof}
Here
\[
 N=\operatorname{ord}(\xi_r^2)=2r,
 \qquad
 2N=4r
\]
Similar to the proof above, $\Bbbk_r^{(4)}$ has characteristic zero.  Thus the field hypothesis in the criterion holds.  Moreover,
$N$ is even and
\[
 \mathcal A\longrightarrow\Bbbk_r^{(4)}
\]
factors through
$\mathcal A_N'=\mathcal A/(\Phi_{4r}(q))$ by
\eqref{eq:k4r-xi-r}.
The rank-one graph has no cycle.  Since $a_{01}=-1$ and $a_{10}=-4$,
the two ordered size conditions are
\[
 \ell_0=r\geq3\geq2=1-a_{01},
 \qquad
 \ell_1=2r\geq6\geq5=1-a_{10}.
\]
Thus Lusztig's modified Frobenius criterion gives
\eqref{eq:rank-one-phi4r-frobenius} and the displayed idempotent formula.
Since
\[
 0\in X_r^{\star},\qquad \ell_0=r,\qquad r\mid r,
\]
the positive simple-block formula gives
\[
 \operatorname{Fr}_{r}^{(4)}(E_0^{(r)}1_0)
 =E_0^{\star,(r/\ell_0)}1_0^{\star}
 =E_0^{\star}1_0^{\star}.
\]
Its nonvanishing follows from the target modified PBW basis
\cite[\S\S23.2.1--23.2.2]{L}.
\end{proof}

\begin{proposition}
\label[proposition]{prop:phi4r-coefficient-ideal}
For every odd integer \(r\geq3\), one has
\begin{equation}
 \mathfrak c_{1,r}\subseteq(\Phi_{4r}(q)).
 \label{eq:phi4r-coefficient-containment}
\end{equation}
\end{proposition}

\begin{proof}
Fix
\[
 r\in(2\mathbb Z+1)\cap\mathbb Z_{\geq3},
 \qquad
 (\lambda,u)\in X\times\mathscr C_1,
 \qquad
 \nu:=\deg u.
\]
The idempotent formula in
\eqref{eq:rank-one-phi4r-frobenius} verifies
\eqref{eq:simple-block-map}.  Suppose
\[
 \operatorname{Fr}_{r}^{(4)}(\overline{u1_\lambda})\neq0.
\]
If
\[
 u=(x_{1,k}^{\varepsilon})^{(s)},
 \qquad
 \nu=sk\alpha_0+s(2k+\varepsilon)\alpha_1,
\]
then \cref{lem:simple-block-degree} gives
\[
 \nu\in L_r^\star,
 \qquad
 r\mid sk,
 \qquad
 2r\mid s(2k+\varepsilon).
\]
Since
\[
 \gcd(k,2k+\varepsilon)=1,
\]
one has \(r\mid s\).  Writing \(s=rt\) gives
\[
 2\mid t(2k+\varepsilon)
 \Longrightarrow
 2\mid t
 \Longrightarrow
 \epsilon_r^\star(\nu)
 =\overline{tk}
 =\overline0.
\]
Thus every nonzero real-current block image belongs to
\(\mathcal G_{\bar0}^{(r)}\).  For a toral generator,
\[
 \nu=0,
 \qquad
 \operatorname{Fr}_{r}^{(4)}(\overline{u1_\lambda})
 \in\mathcal G_{\bar0}^{(r)}.
\]
Hence \cref{lem:real-current-parity-propagation} gives
\begin{equation}
 \operatorname{Fr}_{r}^{(4)}
 \bigl(\overline{\UDrn{1}1_0}\bigr)
 \subseteq\mathcal G_{\bar0}^{(r)}.
 \label{eq:rank-one-core-even}
\end{equation}
Moreover,
\[
 r\alpha_0=\alpha_0^\star\in L_r^\star,
 \qquad
 \epsilon_r^\star(r\alpha_0)=\overline1,
\]
and
\[
 0\neq
 \operatorname{Fr}_{r}^{(4)}(\overline{E_0^{(r)}1_0})
 =E_0^\star1_0^\star
 \in\mathcal G_{\bar1}^{(r)}.
\]
Thus \cref{lem:parity-coefficient-obstruction}, applied to
\[
 \chi_r^{(4)}
 :=\left.\operatorname{Fr}_{r}^{(4)}\right|_{
 \Bbbk_r^{(4)}\otimes_{\mathcal A}
 (U_{\mathcal A}^{\mathrm{Lus},(1)}1_0)},
\]
gives
\[
 \mathfrak c_{1,r}\subseteq(\Phi_{4r}(q)),
\]
which is \eqref{eq:phi4r-coefficient-containment}.
\end{proof}

The restriction to odd $r$ is essential for the displayed starred datum.
If $r$ were even, then
\[
 \ell_0=\operatorname{ord}(\xi_r^8)
 =\frac{4r}{\gcd(4r,8)}=\frac r2,
\]
so the scaled affine root would be
$\alpha_0^{\star}=(r/2)\alpha_0$ and, when the Frobenius criterion applies,
\[
 E_0^{(r)}1_0\longmapsto
 E_0^{\star,(r/\ell_0)}1_0^{\star}
 =E_0^{\star,(2)}1_0^{\star}.
\]
Its $\alpha_0^{\star}$-coordinate is even, rather than the required odd
coordinate.  For an even index $2s\geq4$,
\[
 \Phi_{4(2s)}(q)=\Phi_{8s}(q),
\]
and $s\geq2$, so the corresponding prime is already covered by
\cref{thm:phi8r-coefficient-ideal}.  The remaining $\Phi_4$ and $\Phi_8$
endpoints, for which the affine scaling integer $\ell_0$ equals $1$, are
treated in
\cref{sec:endpoint-obstructions}.
\section{The scaled root datum and cyclotomic obstructions for
\texorpdfstring{$r=1$}{r=1}}
\label{sec:endpoint-obstructions}

Throughout Sections~\ref{sec:endpoint-obstructions}--
\ref{sec:least-ordinary-localization} and
Appendix~\ref{app:endpoint-technical}, we use
\(\mathbb N=\mathbb Z_{\geq0}\).
At the two remaining primes, the scaled affine integer is
\(\ell_0=1\), so the size condition
\eqref{eq:frobenius-size-condition} fails and the Frobenius criterion of
Section~\ref{sec:scaled-root-datum} does not apply. Using the root data
of that section at \(r=1\), we replace the Frobenius target
\(\dot U_{\Bbbk}^{\bullet}\) by a common triangular quotient of the
specialized source form. Its technical construction is given in
Appendix~\ref{app:endpoint-technical}.
Subsection~\ref{subsec:endpoint-parity-obstruction} then applies
Lemmas~\ref{lem:real-current-parity-propagation}
and~\ref{lem:parity-coefficient-obstruction} to prove
\[
 \mathfrak c_{n,1}\subseteq(\Phi_8(q))
 \quad(n\geq1),
 \qquad
 \mathfrak c_{1,1}\subseteq(\Phi_4(q)).
\]

\subsection{The two endpoint instances}
\label{subsec:endpoint-instances}

\begin{convention}[Endpoint data]
\label[convention]{conv:endpoint-data}
Fix \(n\in\mathbb Z_{\geq1}\) and retain the regular \(A_{2n}^{(2)}\)
root datum \((X,Y,I,\langle\ ,\ \rangle)\) of
\eqref{eq:a2n-root-datum} as the source datum, with symmetrizers
\((d_i)_{i\in I}\) and \(q_i=q^{d_i}\).  Assume given a field

\begin{equation}
 \Bbbk,
 \qquad
 \bar q\in\Bbbk^{\times},
 \qquad
 \operatorname{char}\Bbbk=0,
 \qquad
 N:=\operatorname{ord}(\bar q^{\,2})\in2\mathbb Z_{>0},
 \label{eq:endpoint-ground-field}
\end{equation}
Here \(\bar q\) denotes the specialized value of \(q\) in \(\Bbbk\). Let \((\ell_i)_{i\in I}\), \(X^{\bullet}\), \(\alpha_i^{\bullet}\),
\(h_i^{\bullet}\), \(L^{\bullet}\) be the associated scaled data of
\eqref{eq:scaled-X-lattice}--\eqref{eq:scaled-simple-data} and
\eqref{eq:scaled-root-lattice}, that is
\begin{equation}
 \ell_i
 =\operatorname{ord}\bigl(\bar q^{\,2d_i}\bigr)
 =\frac{N}{\gcd(N,d_i)},
 \qquad
 \alpha_i^{\bullet}=\ell_i\alpha_i,
 \qquad
 L^{\bullet}=\bigoplus_{i\in I}\mathbb Z\alpha_i^{\bullet},
 \label{eq:endpoint-scaled-data}
\end{equation}
\begin{equation}
 X^{\bullet}
 =\bigl\{\lambda\in X\bigm|
   \langle h_i,\lambda\rangle\in\ell_i\mathbb Z
   \ \text{for every }i\in I\bigr\}.
 \label{eq:endpoint-scaled-weights}
\end{equation}
We assume throughout
Section~\ref{sec:endpoint-obstructions}
and Appendix~\ref{app:endpoint-technical} that
\begin{equation}
 \bigl(X^{\bullet},Y^{\bullet},I,\langle\ ,\ \rangle,
 (\alpha_i^{\bullet})_{i\in I},(h_i^{\bullet})_{i\in I}\bigr)
 \ \text{is a regular root datum},
 \qquad
 L^{\bullet}\subseteq X^{\bullet},
 \label{eq:endpoint-datum-regular}
\end{equation}
\begin{equation}
 \ell_0=1 .
 \label{eq:endpoint-affine-scaling-one}
\end{equation}
Write \(q_{\Bbbk,i}:=\bar q^{\,d_i}\) for the specialized local
parameter, so that \(\ell_i=\operatorname{ord}(q_{\Bbbk,i}^{2})\).
\end{convention}

\begin{lemma}[The two endpoint instances]
\label[lemma]{lem:endpoint-instances}
Convention~\ref{conv:endpoint-data} is satisfied in the following two ways,
both obtained from \cref{sec:scaled-root-datum} at \(r=1\).
\begin{enumerate}[label=\textup{(\alph*)}]
 \item \emph{The \(\Phi_8(q)\) endpoint.}
 Let \(n\in\mathbb Z_{\geq1}\) and put
 \(
  (\Bbbk,\bar q):=\bigl(\Bbbk_1^{(8)},\zeta_1\bigr)
 \)
 as in \eqref{eq:kr-zeta-r} with \(r=1\).  Then \(N=4\),
 \[
  (\ell_0,\ell_1,\ldots,\ell_n)=(1,2,\ldots,2,4),
  \qquad
  L^{\bullet}=L_1^{\bullet}
  =\mathbb Z\alpha_0
   \oplus\bigoplus_{i=1}^{n-1}2\mathbb Z\alpha_i
   \oplus4\mathbb Z\alpha_n ,
 \]
 the middle sum being empty for \(n=1\).
 \item \emph{The \(\Phi_4(q)\) endpoint.}
 Let \(n=1\) and put
 \(
  (\Bbbk,\bar q):=\bigl(\Bbbk_1^{(4)},\xi_1\bigr)
 \)
 as in \eqref{eq:k4r-xi-r} with \(r=1\).  Then \(N=2\),
 \[
  (\ell_0,\ell_1)=(1,2),
  \qquad
  X^{\bullet}=X_1^{\star},
  \qquad
  L^{\bullet}=L_1^{\star}
  =\mathbb Z\alpha_0\oplus2\mathbb Z\alpha_1 .
 \]
\end{enumerate}
\end{lemma}

\begin{proof}
We verify the two instances separately.

(a)  In \cref{subsec:phi8r-scaled-datum}, \eqref{eq:kr-zeta-r} gives
\(\operatorname{ord}(\zeta_1)=8\), hence \(N=4\);
\eqref{eq:ell-values-phi8r} gives \((\ell_0,\ldots,\ell_n)=(1,2,\ldots,2,4)\);
and \cref{lem:scaled-datum-regular}, which is stated for every \(r\geq1\),
gives \eqref{eq:endpoint-datum-regular} together with \(L^{\bullet}=L_1^{\bullet}\).

(b)Here
\(\Bbbk=\operatorname{Frac}(\mathcal A/(\Phi_4(q)))\)
has characteristic zero, and \(\bar q=\xi_1\) has order \(4\).
Thus
\[
 N=\operatorname{ord}(\bar q^2)=2,
 \qquad
 (\ell_0,\ell_1)
 =
 \left(
 \frac{2}{\gcd(2,4)},
 \frac{2}{\gcd(2,1)}
 \right)
 =(1,2).
\]
Evaluating the defining formulas of
\cref{subsec:rank-one-phi4r} at \(r=1\), we obtain
\[
 X^\bullet=X_1^\star
 =\{\lambda\in X\mid
       \langle h_1,\lambda\rangle\in2\mathbb Z\},
 \qquad
 L^\bullet=L_1^\star
 =\mathbb Z\alpha_0\oplus2\mathbb Z\alpha_1.
\]
We verify regularity directly.
Since
\(\langle h_1,\alpha_0\rangle=-4\) and
\(\langle h_1,2\alpha_1\rangle=4\),
both scaled simple roots
\(\alpha_0^\bullet=\alpha_0\) and
\(\alpha_1^\bullet=2\alpha_1\) belong to \(X^\bullet\).
The scaled coroots
\[
 h_0^\bullet=h_0|_{X^\bullet},
 \qquad
 h_1^\bullet=\tfrac12 h_1|_{X^\bullet}
\]
belong to
\(Y^\bullet=\operatorname{Hom}_{\mathbb Z}(X^\bullet,\mathbb Z)\),
and
\[
 (a_{ij}^\bullet)_{i,j=0,1}
 =
 \begin{pmatrix}
  2&-2\\
  -2&2
 \end{pmatrix},
 \qquad
 (d_0^\bullet,d_1^\bullet)=(4,4).
\]
Since \(2X\subseteq X^\bullet\subseteq X\),
the lattice \(X^\bullet\) has finite index in \(X\).
The scaled simple roots and coroots therefore remain
linearly independent, by regularity of the source datum.
This proves \eqref{eq:endpoint-datum-regular}.

In both cases \(\ell_0=1\).
\end{proof}

\subsection{The endpoint quotient $\mathcal Q$}
\label{subsec:endpoint-triangular-quotient}

We retain here the quotient and the properties needed for the parity
argument. The detailed word calculus, Gaussian vanishing, and
triangular straightening are collected in
Appendix~\ref{app:endpoint-technical}.
Let
\[
 \dot V
 :=\Bbbk\otimes_{\mathcal A,\,q\mapsto\bar q}
   \dot U_{\mathcal A}^{\Lus,(n)}.
\]
For a labelled divided-power word
\(\mathbf w=(u_1,\ldots,u_m;\eta)\), where \(\eta\in X\) and
\(u_j=E_{i_j}^{(a_j)}\) or \(F_{i_j}^{(a_j)}\) with \(a_j\geq1\), set
\[
 \lambda_j(\mathbf w)
 :=\eta+\sum_{v=j+1}^{m}\deg u_v
 \quad(0\leq j\leq m),
 \qquad
 \operatorname{ev}(\mathbf w):=u_1\cdots u_m1_\eta.
\]
For \(m=0\), the evaluation is \(1_\eta\).

The definition below encodes the two selection rules suggested by
quantum Frobenius in \cref{lem:modified-quantum-frobenius}: only weights
in \(X^\bullet\) and divided-power exponents divisible by the
corresponding \(\ell_i\) are allowed to survive. The intermediate
weights \(\lambda_j(\mathbf w)\) record the successive weight
idempotents through which the word passes. We exclude any word that
passes through an idempotent outside \(X^\bullet\), or contains a
divided-power letter whose exponent is not a multiple of its scaling
integer. At the endpoints the Frobenius size condition fails, so we
implement these selection rules directly in a quotient of \(\dot V\).

Call \(\mathbf w\) bad if
\[
 \bigl(\exists\,j\in\{0,\ldots,m\}:\
        \lambda_j(\mathbf w)\notin X^\bullet\bigr)
 \quad\text{or}\quad
 \bigl(\exists\,j\in\{1,\ldots,m\}:\
        \ell_{i_j}\nmid a_j\bigr),
\]
and define
\[
 \mathfrak J
 :=\operatorname{span}_{\Bbbk}
   \{\operatorname{ev}(\mathbf w)\mid\mathbf w\text{ is bad}\}.
\]
The formal definitions of these words and their evaluations are given in
Appendix~\ref{app:endpoint-words}.
To obtain the required obstruction, we must show that \(\mathfrak J\)
is a two-sided ideal and that \(E_01_0\) survives in the quotient.
The following theorem establishes these facts; its proof uses the
Gaussian and triangular arguments in the appendix.

\begin{theorem}
\label{thm:endpoint-triangular-quotient}
Under \cref{conv:endpoint-data},
\[
 \dot V\mathfrak J\subseteq\mathfrak J,
 \qquad
 \mathfrak J\dot V\subseteq\mathfrak J.
\]
Hence \(\mathcal Q:=\dot V/\mathfrak J\) is a root-graded
\(\Bbbk\)-algebra, and
\begin{equation}
 \bigl\{\nu\bigm|(\mathcal Q)_{\nu}\neq0\bigr\}
 \subseteq L^{\bullet},
 \qquad
 E_01_0+\mathfrak J
 \in(\mathcal Q)_{\alpha_0}\setminus\{0\}.
 \label{eq:endpoint-quotient-support}
\end{equation}
\end{theorem}

\begin{proof}
A product of a divided-power word and a bad word, in either order,
is zero or the evaluation of a bad word. Hence \(\mathfrak J\) is a
two-sided ideal. Since the evaluated bad words are root homogeneous,
\(\mathfrak J\) is root graded. The straightening arguments in
Appendix~\ref{app:endpoint-straightening} give its triangular
decomposition. This decomposition yields the root support and the
nonvanishing in \eqref{eq:endpoint-quotient-support}; the full proof is
given in Appendix~\ref{app:endpoint-quotient-proof}.
\end{proof}

The root support in \eqref{eq:endpoint-quotient-support} allows the
parity on \(L^\bullet\) to grade \(\mathcal Q\).
Subsection~\ref{subsec:endpoint-parity-obstruction} verifies that every
surviving current-generator block is even, whereas
\(E_01_0+\mathfrak J\) is nonzero and odd.
Lemma~\ref{lem:real-current-parity-propagation} then makes the whole
current core even, and Lemma~\ref{lem:parity-coefficient-obstruction}
gives the endpoint coefficient-ideal obstructions.

\subsection{Parity separation and coefficient ideals}
\label{subsec:endpoint-parity-obstruction}

Retain \cref{conv:endpoint-data} and the notation of
\cref{subsec:endpoint-triangular-quotient}.  Exactly as in
\eqref{eq:epsilon-r} and \eqref{eq:rank-one-starred-parity}, define
\begin{equation}
 \epsilon:L^{\bullet}\longrightarrow\mathbb Z/2\mathbb Z,
 \qquad
 \epsilon\Bigl(\sum_{i\in I}b_i\alpha_i^{\bullet}\Bigr):=\overline{b_0},
 \qquad
 \epsilon\in\operatorname{Hom}_{\mathbb Z}
 \bigl(L^{\bullet},\mathbb Z/2\mathbb Z\bigr).
 \label{eq:endpoint-parity}
\end{equation}
By \eqref{eq:endpoint-quotient-support} the root grading of
\(\mathcal Q\) is supported on \(L^{\bullet}\), so
\begin{equation}
 \mathcal Q_{\bar e}
 :=\bigoplus_{\substack{\nu\in L^{\bullet}\\ \epsilon(\nu)=\bar e}}
   (\mathcal Q)_{\nu}
 \qquad(\bar e\in\mathbb Z/2\mathbb Z)
 \label{eq:endpoint-quotient-parity-spaces}
\end{equation}
satisfies
\begin{equation}
 \mathcal Q=\mathcal Q_{\bar0}\oplus\mathcal Q_{\bar1},
 \qquad
 \mathcal Q_{\bar e}\mathcal Q_{\bar f}
 \subseteq\mathcal Q_{\overline{e+f}} ,
 \label{eq:endpoint-quotient-parity-grading}
\end{equation}
the second inclusion because \(\mathcal Q\) is root graded and
\(\epsilon\) is additive.  Let
\begin{equation}
 \varphi:
 \Bbbk\otimes_{\mathcal A,\,q\mapsto\bar q}
 \dot U_{\mathcal A}^{\Lus,(n)}
 =\dot V
 \longrightarrow\mathcal Q
 \label{eq:endpoint-detector}
\end{equation}
be the canonical projection, a \(\Bbbk\)-algebra homomorphism by
\cref{thm:endpoint-triangular-quotient}.  Since \(\mathfrak J\) is root
graded,
\begin{equation}
 \varphi\Bigl(
  \bigl(\Bbbk\otimes_{\mathcal A}
   \dot U_{\mathcal A}^{\Lus,(n)}\bigr)_{\nu}
 \Bigr)
 \subseteq(\mathcal Q)_{\nu}
 \qquad
 \Bigl(\nu\in\bigoplus_{i\in I}\mathbb Z\alpha_i\Bigr).
 \label{eq:endpoint-detector-graded}
\end{equation}
Thus \(\varphi\) plays here the role played by
\(\operatorname{Fr}_{n,r}\) in \eqref{eq:higher-frobenius} and by
\(\operatorname{Fr}_r^{(4)}\) in
\eqref{eq:rank-one-phi4r-frobenius}, and
\(\mathcal Q_{\bar0}\oplus\mathcal Q_{\bar1}\) replaces
\eqref{eq:higher-target-parity-grading} and
\eqref{eq:rank-one-target-parity-grading}.

\begin{lemma}
\label[lemma]{lem:endpoint-current-evenness}
Let \((s,k,\varepsilon,i)\in
\mathbb N\times\mathbb Z\times\{1,-1\}\times\{1,\ldots,n\}\).
Then \(\deg\!\left((x_{i,k}^{\varepsilon})^{(s)}\right)
=s(k\delta+\varepsilon\alpha_i)\). Let
\[
 \gamma:=s(k\delta+\varepsilon\alpha_i)
 =sk\alpha_0
  +\sum_{\substack{1\leq j\leq n\\j\neq i}}2sk\,\alpha_j
  +s(2k+\varepsilon)\alpha_i.
\]
Then, in either instance of \cref{lem:endpoint-instances},
if \(\gamma\in L^{\bullet}\), then
\(\epsilon(\gamma)=\overline{sk}=\overline0\).
\end{lemma}

\begin{proof}
Since \(\ell_0=1\), the \(\alpha_0^{\bullet}\)-coordinate of \(\gamma\)
is \(sk\), so \(\epsilon(\gamma)=\overline{sk}\); it remains to show
\(2\mid sk\).  Note that \(2k+\varepsilon\) is odd.

In instance (a) of \cref{lem:endpoint-instances}, membership in
\(L^{\bullet}=L_1^{\bullet}\) forces the \(\alpha_n\)-coordinate of
\(\gamma\) to lie in \(4\mathbb Z\).  If \(i<n\) that coordinate is
\(2sk\), so \(4\mid2sk\) and \(2\mid sk\).  If \(i=n\) it is
\(s(2k+\varepsilon)\), so \(4\mid s(2k+\varepsilon)\); since
\(\gcd(4,2k+\varepsilon)=1\), this gives \(4\mid s\) and hence
\(2\mid sk\).  This is \cref{lem:uniform-current-evenness} at \(r=1\),
where that lemma remains valid.

In instance (b) we have \(n=1\), hence \(i=1\), and membership in
\(L^{\bullet}=L_1^{\star}\) forces the \(\alpha_1\)-coordinate of
\(\gamma\) to lie in \(2\mathbb Z\):
\[
 2\mid s(2k+\varepsilon),
 \qquad
 \gcd(2,2k+\varepsilon)=1
 \quad\Longrightarrow\quad
 2\mid s
 \quad\Longrightarrow\quad
 2\mid sk .
\]
\end{proof}

\begin{theorem}
\label{thm:endpoint-coefficient-ideals}
For every integer \(n\geq1\), one has
\begin{equation}
 \mathfrak c_{n,1}\subseteq\bigl(\Phi_8(q)\bigr),
 \qquad
 \mathfrak c_{1,1}\subseteq\bigl(\Phi_4(q)\bigr).
 \label{eq:endpoint-coefficient-ideals}
\end{equation}
\end{theorem}

\begin{proof}
Fix one of the two instances of \cref{lem:endpoint-instances} and let
\[
 \pi:=
 \begin{cases}
  \Phi_8(q),&\text{instance (a)},\\
  \Phi_4(q),&\text{instance (b)},
 \end{cases}
 \qquad
 \Bbbk=\operatorname{Frac}\bigl(\mathcal A/(\pi)\bigr),
\]
so that \cref{conv:endpoint-data} holds and
\cref{thm:endpoint-triangular-quotient} applies.  Let \(\varphi\) be as
in \eqref{eq:endpoint-detector}.

\emph{Step 1: every nonzero generator block is even.}
Let \((\lambda,u)\in X\times\mathscr C_n\) with
\(\varphi(\overline{u1_\lambda})\neq0\), and put \(\gamma:=\deg u\).
If \(u\in\{K_\delta^{\pm1}\}\cup\{K_i^{\pm1}\mid1\leq i\leq n\}\), then
\(\gamma=0\), \(\epsilon(0)=\overline0\), and
\eqref{eq:endpoint-detector-graded} gives
\(\varphi(\overline{u1_\lambda})\in(\mathcal Q)_{0}
\subseteq\mathcal Q_{\bar0}\).
If \(u=(x_{i,k}^{\varepsilon})^{(s)}\) with \(1\leq i\leq n\),
\(k\in\mathbb Z\), \(s\in\mathbb N\), \(\varepsilon\in\{1,-1\}\), then
\(\gamma=s(k\delta+\varepsilon\alpha_i)\) by
\eqref{eq:current-degrees}, and
\eqref{eq:endpoint-detector-graded} together with
\eqref{eq:endpoint-quotient-support} give
\[
 0\neq\varphi(\overline{u1_\lambda})\in(\mathcal Q)_{\gamma}
 \quad\Longrightarrow\quad
 \gamma\in L^{\bullet}.
\]
Hence \(\epsilon(\gamma)=\overline0\) by
\cref{lem:endpoint-current-evenness}, and
\(\varphi(\overline{u1_\lambda})\in\mathcal Q_{\bar0}\).  Therefore
\[
 \forall\,(\lambda,u)\in X\times\mathscr C_n,
 \qquad
 \varphi(\overline{u1_\lambda})\neq0
 \Longrightarrow
 \varphi(\overline{u1_\lambda})\in\mathcal Q_{\bar0}.
\]

\emph{Step 2: the whole real-current core is even.}
By \eqref{eq:endpoint-quotient-parity-grading},
\(\mathcal Q=\mathcal Q_{\bar0}\oplus\mathcal Q_{\bar1}\) is a
\(\mathbb Z/2\mathbb Z\)-graded \(\Bbbk\)-algebra, and \(\varphi\) is a
\(\Bbbk\)-algebra homomorphism, so Step~1 and
\cref{lem:real-current-parity-propagation} give
\begin{equation}
 \varphi\bigl(\overline{\UDrn{n}1_0}\bigr)
 \subseteq\mathcal Q_{\bar0}.
 \label{eq:endpoint-core-even}
\end{equation}

\emph{Step 3: the affine divided power is odd.}
By \cref{lem:endpoint-lattice-basics}(b),
\(\alpha_0=\alpha_0^{\bullet}\in L^{\bullet}\), whose
\(\alpha_0^{\bullet}\)-coordinate is \(1\); hence
\(\epsilon(\alpha_0)=\overline1\).  Since \(\ell_0=1\) we have
\(E_0^{(1)}=E_0\), and \eqref{eq:endpoint-quotient-support} gives
\begin{equation}
 0\neq\varphi\bigl(\overline{E_0^{(1)}1_0}\bigr)
 =E_01_0+\mathfrak J
 \in(\mathcal Q)_{\alpha_0}
 \subseteq\mathcal Q_{\bar1}.
 \label{eq:endpoint-E0-odd}
\end{equation}

\emph{Step 4: conclusion.}
Apply \cref{lem:parity-coefficient-obstruction} with \(r=1\),
\(D_{\bar e}=\mathcal Q_{\bar e}\), and
\[
 \chi:=\left.\varphi\right|_{
  \Bbbk\otimes_{\mathcal A}
  \bigl(U_{\mathcal A}^{\mathrm{Lus},(n)}1_0\bigr)} .
\]
Its two hypotheses are \eqref{eq:endpoint-core-even} and
\eqref{eq:endpoint-E0-odd}, and the sum
\(\mathcal Q_{\bar0}\oplus\mathcal Q_{\bar1}\) is internal by
\eqref{eq:endpoint-quotient-parity-grading}.  Hence
\[
 \mathfrak c_{n,1}\subseteq(\pi).
\]
Instance (a) is available for every \(n\geq1\) and gives
\(\mathfrak c_{n,1}\subseteq(\Phi_8(q))\); instance (b) has \(n=1\) and
gives \(\mathfrak c_{1,1}\subseteq(\Phi_4(q))\).  This is
\eqref{eq:endpoint-coefficient-ideals}.
\end{proof}
\begin{corollary}
\label[corollary]{cor:exact-endpoint-coefficient-ideals}
For every integer \(n\geq1\), one has
\begin{equation}
 \mathfrak c_{n,1}
 =
 \begin{cases}
  \bigl(\Phi_4(q)\Phi_8(q)\bigr),&n=1,\\
  \bigl(\Phi_8(q)\bigr),&n\geq2.
 \end{cases}
 \label{eq:exact-endpoint-coefficient-ideals}
\end{equation}
\end{corollary}

\begin{proof}
For \(n=1\), equation \eqref{eq:da-affine-powers} with \(a=1\)
gives
\[
 \Phi_4(q)\Phi_8(q)=d_1\in\mathfrak c_{1,1}.
\]
Since \(\mathfrak c_{1,1}\) is an \(\mathcal A\)-ideal,
Theorem~\ref{thm:endpoint-coefficient-ideals} yields
\[
 \bigl(\Phi_4(q)\Phi_8(q)\bigr)
 \subseteq\mathfrak c_{1,1}
 \subseteq
 \bigl(\Phi_4(q)\bigr)\cap\bigl(\Phi_8(q)\bigr)
 =
 \bigl(\Phi_4(q)\Phi_8(q)\bigr),
\]
where the last equality follows because \(\Phi_4(q)\) and
\(\Phi_8(q)\) are nonassociate prime elements of \(\mathcal A\).

For \(n\geq2\), equation \eqref{eq:scaled-affine-powers}
with \(a=1\) gives
\(\Phi_8(q)=c_1\in\mathfrak c_{n,1}\).
Together with Theorem~\ref{thm:endpoint-coefficient-ideals},
this gives
\[
 \bigl(\Phi_8(q)\bigr)
 \subseteq\mathfrak c_{n,1}
 \subseteq\bigl(\Phi_8(q)\bigr).
\]
\end{proof}

\section{The least ordinary localization}
\label{sec:least-ordinary-localization}

Retain the localization convention of
\cref{def:sufficient-localization} and the discussion preceding it; in
particular, both sides of every equality below are extended to the same
coefficient ring.

\begin{proof}[Proof of Theorem~\ref{thm:uniform-least-ring}]
If
\[
 \mathfrak g\notin
 \{A_{2m}^{(2)}\mid m\geq1\},
\]
then \cref{thm:nonexceptional-equality} gives
\[
 \UDr(\mathfrak g)
 =U_{\mathcal A}^{\mathrm{Lus}}(\mathfrak g),
 \qquad
 B_{\min}(\mathfrak g)=\mathcal A.
\]

Let
\[
 B=S^{-1}\mathcal A\subseteq\mathcal K
\]
be an ordinary localization.  We treat the two exceptional cases
separately.

Suppose first that \(n=1\).  By
\cref{thm:rank-one-Cstar},
\[
 \Cstar\UDrn{1}
 =
 \Cstar U_{\mathcal A}^{\mathrm{Lus},(1)}.
\]
Hence \(\Cstar\subseteq B\) implies, by extension of scalars,
\[
 B\UDrn{1}
 =
 BU_{\mathcal A}^{\mathrm{Lus},(1)}.
\]

Conversely, suppose that
\[
 B\UDrn{1}
 =
 BU_{\mathcal A}^{\mathrm{Lus},(1)}.
\]
The coefficient-ideal calculations give
\[
 \mathfrak c_{1,r}\subseteq\bigl(\Phi_{4r}(q)\bigr)
 \qquad
 (r\geq1\ \text{odd})
\]
by \cref{thm:endpoint-coefficient-ideals,prop:phi4r-coefficient-ideal},
and
\[
 \mathfrak c_{1,r}\subseteq\bigl(\Phi_{8r}(q)\bigr)
 \qquad(r\geq1)
\]
by \cref{thm:endpoint-coefficient-ideals,thm:phi8r-coefficient-ideal}.
Therefore \cref{lem:coefficient-prime-invertible} implies that
\[
 \Phi_{4r}(q)\in B^\times
 \quad(r\geq1\ \text{odd}),
 \qquad
 \Phi_{8r}(q)\in B^\times
 \quad(r\geq1).
\]
Every positive integer is either odd or of the form \(2r\), and
\(\Phi_{4(2r)}(q)=\Phi_{8r}(q)\).  Thus
\[
 \Phi_{4r}(q)\in B^\times
 \qquad(r\geq1),
\]
which is equivalent to \(\Cstar\subseteq B\).  Consequently,
\[
 B\UDrn{1}
 =
 BU_{\mathcal A}^{\mathrm{Lus},(1)}
 \quad\Longleftrightarrow\quad
 \Cstar\subseteq B.
\]

Now suppose that \(n\geq2\).  By
\cref{thm:higher-rank-Bstar},
\[
 \Bstar\UDrn{n}
 =
 \Bstar U_{\mathcal A}^{\mathrm{Lus},(n)},
\]
so \(\Bstar\subseteq B\) again implies equality after extension of
scalars.  Conversely, if
\[
 B\UDrn{n}
 =
 BU_{\mathcal A}^{\mathrm{Lus},(n)},
\]
then
\[
 \mathfrak c_{n,r}\subseteq\bigl(\Phi_{8r}(q)\bigr)
 \qquad(r\geq1)
\]
by \cref{thm:endpoint-coefficient-ideals,thm:phi8r-coefficient-ideal}.
Hence \cref{lem:coefficient-prime-invertible} gives
\[
 \Phi_{8r}(q)\in B^\times
 \qquad(r\geq1),
\]
or equivalently \(\Bstar\subseteq B\).  Therefore
\[
 B\UDrn{n}
 =
 BU_{\mathcal A}^{\mathrm{Lus},(n)}
 \quad\Longleftrightarrow\quad
 \Bstar\subseteq B
 \qquad(n\geq2).
\]

It follows that
\[
 B_{\min}(A_{2n}^{(2)})
 =
 \begin{cases}
  \Cstar,&n=1,\\
  \Bstar,&n\geq2.
 \end{cases}
\]
This proves all exceptional cases and completes the theorem.
\end{proof}

\begin{corollary}[No finite denominator set]
\label{cor:no-finite-denominator-set}
For every
\[
 n\in\mathbb Z_{\geq1},
 \qquad
 t\in\mathbb Z_{\geq0},
 \qquad
 (f_1,\ldots,f_t)\in
 \bigl(\mathcal A\setminus\{0\}\bigr)^t,
\]
the ordinary localization
\[
 B:=\mathcal A[f_1^{-1},\ldots,f_t^{-1}]
 \subset\mathcal K
\]
satisfies
\[
 B\UDrn{n}\neq BU_{\mathcal A}^{\Lus,(n)}.
\]
\end{corollary}

\begin{proof}
By Gauss's lemma, \(\mathbb Z[q]\) is a UFD, and every cyclotomic
polynomial \(\Phi_m(q)\) is prime in \(\mathbb Z[q]\).  Localizing at
the powers of \(q\) gives
\[
 \mathcal A=\mathbb Z[q,q^{-1}],
\]
which is again a UFD.  Since \(\Phi_m(0)=\pm1\), the prime \(\Phi_m(q)\) is not
associate to \(q\) and hence divides no power of \(q\).  Thus
\[
 (\Phi_m(q))\cap\{q^k\mid k\geq0\}=\varnothing,
\]
so the prime ideal \((\Phi_m(q))\) remains prime after localization
at the powers of \(q\).  Hence \(\Phi_m(q)\) is prime in
\(\mathcal A=\mathbb Z[q,q^{-1}]\).

The units of \(\mathcal A\) are precisely \(\pm q^k\), \(k\in\mathbb Z\).
If two cyclotomic polynomials are associates in \(\mathcal A\), then
\[
 \Phi_m(q)=\pm q^k\Phi_{m'}(q)
\]
for some \(k\in\mathbb Z\).  Comparing the lowest powers of \(q\)
gives \(k=0\), because both cyclotomic polynomials have nonzero
constant term.  Their monicity then gives
\[
 \Phi_m(q)=\Phi_{m'}(q),
\]
and hence \(m=m'\).  Therefore both
\[
 \{\Phi_{4r}(q)\mid r\geq1\}
 \quad\text{and}\quad
 \{\Phi_{8r}(q)\mid r\geq1\}
\]
are infinite families of pairwise nonassociate primes in
\(\mathcal A\).

Now let
\[
 B=\mathcal A[f_1^{-1},\ldots,f_t^{-1}]
 \subseteq\mathcal K,
\]
where \(f_1,\ldots,f_t\neq0\).  If a nonunit prime \(p\in\mathcal A\)
is invertible in \(B\), then
\[
 p^{-1}
 =
 \frac{a}{f_1^{e_1}\cdots f_t^{e_t}}
\]
for some \(a\in\mathcal A\) and \(e_1,\ldots,e_t\geq0\).  Thus
\[
 f_1^{e_1}\cdots f_t^{e_t}=pa,
\]
so \(p\) divides one of \(f_1,\ldots,f_t\).  Since each \(f_j\) has
only finitely many prime divisors in the UFD \(\mathcal A\), only
finitely many pairwise nonassociate primes can become invertible in
\(B\).

On the other hand, \cref{thm:uniform-least-ring} shows that
\[
 B\UDrn{n}
 =
 BU_{\mathcal A}^{\mathrm{Lus},(n)}
\]
would force every \(\Phi_{4r}(q)\) to be invertible in \(B\) when
\(n=1\), and every \(\Phi_{8r}(q)\) to be invertible in \(B\) when
\(n\geq2\).  In either case this would make infinitely many pairwise
nonassociate primes invertible in \(B\), a contradiction.
\end{proof}
\appendix
\section{Technical details in Section~\ref{subsec:endpoint-triangular-quotient}}
\label{app:endpoint-technical}

We retain Convention~\ref{conv:endpoint-data} and the notation of
Section~\ref{sec:endpoint-obstructions}, with
\(\mathbb N=\mathbb Z_{\geq0}\).
This appendix supplies the endpoint verification, word calculus,
and triangular decomposition used in the proof of
\cref{thm:endpoint-triangular-quotient}.

\subsection{Endpoint data and Gaussian vanishing}
\label{app:endpoint-preliminaries}

\begin{lemma}
\label[lemma]{lem:endpoint-lattice-basics}
\begin{enumerate}[label=\textup{(\alph*)}]
 \item \(X^{\bullet}\) is a subgroup of \(X\), and
 \(\lambda+\nu\in X^{\bullet}\) for all
 \((\lambda,\nu)\in X^{\bullet}\times L^{\bullet}\).
 \item \(\alpha_0=\alpha_0^{\bullet}\in L^{\bullet}\), and
 \[
  I^{\circ}:=\{i\in I\mid\ell_i>1\}\subseteq I\setminus\{0\} .
 \]
 \item For \((i,a)\in I\times\mathbb Z_{\geq1}\),
 \[
  \ell_i\nmid a\ \Longrightarrow\ \ell_i>1,
  \qquad
  \ell_i\mid a\ \Longrightarrow\
  a\alpha_i=\frac{a}{\ell_i}\alpha_i^{\bullet}\in L^{\bullet} .
 \]
\end{enumerate}
\end{lemma}

\begin{proof}
(a)  Each defining condition in \eqref{eq:endpoint-scaled-weights} is a
subgroup condition, so \(X^{\bullet}\leq X\); the second assertion then
follows from \(L^{\bullet}\subseteq X^{\bullet}\) in
\eqref{eq:endpoint-datum-regular}.

(b)  \(\alpha_0^{\bullet}=\ell_0\alpha_0=\alpha_0\) by
\eqref{eq:endpoint-affine-scaling-one}, and the same identity gives
\(0\notin I^{\circ}\).

(c)  If \(\ell_i=1\) then \(\ell_i\mid a\); this is the first
implication.  The second is immediate.
\end{proof}

For \(c\in\mathbb Z\) and \(t\in\mathbb N\) we use Lusztig's symmetric
Gaussian convention: for \(c\geq0\) it is the usual Laurent polynomial,
with value zero for \(t>c\), and for \(c=-d<0\),
\begin{equation}
 \begin{bmatrix}-d\\t\end{bmatrix}_{u}
 :=(-1)^t
 \begin{bmatrix}d+t-1\\t\end{bmatrix}_{u} ,
 \label{eq:negative-symmetric-gaussian}
\end{equation}
as in \cite[\S1.3.1(a)]{L}.

\begin{lemma}[Endpoint Gaussian vanishing]
\label[lemma]{lem:endpoint-lucas}
Let \(\Bbbk\) be a field with \(\operatorname{char}\Bbbk=0\), let
\(\xi\in\Bbbk^{\times}\), and put
\(e:=\operatorname{ord}(\xi^{2})<\infty\).  Then
\begin{equation}
 \begin{bmatrix}c\\t\end{bmatrix}_{\xi}=0
 \qquad
 \bigl(c\in e\mathbb Z,\ t\in\mathbb N,\ e\nmid t\bigr),
 \label{eq:endpoint-gaussian-vanishing}
\end{equation}
with \eqref{eq:negative-symmetric-gaussian} understood for \(c<0\).
\end{lemma}

\begin{proof}
This is \cite[Lemma~34.1.2(a)]{L} with \(\ell=e\), after specializing the
quantum parameter to \(\xi\).  For \(e=1\) the condition \(e\nmid t\) is
empty.
\end{proof}

\subsection{Specialized halves, words, and half ideals}
\label{app:endpoint-words}

\medskip
\noindent\textbf{Specialized halves.}
Let
\begin{equation}
 \mathbf U_{\mathcal A}^{+}
 :=\mathcal A\bigl\langle E_i^{(a)}\bigm|i\in I,\ a\in\mathbb N\bigr\rangle,
 \qquad
 \mathbf U_{\mathcal A}^{-}
 :=\mathcal A\bigl\langle F_i^{(a)}\bigm|i\in I,\ a\in\mathbb N\bigr\rangle,
 \label{eq:integral-halves}
\end{equation}
and let \(\dot U_{\mathcal A}^{\Lus,(n)}\) be the source modified
integral form of \cref{sec:scaled-root-datum}, that is the form denoted
\(\dot U_{\mathcal A}\) in \cref{sec:coefficient-ideals}.  Define
\begin{equation}
 \dot V
 :=\Bbbk\otimes_{\mathcal A,\,q\mapsto\bar q}
   \dot U_{\mathcal A}^{\Lus,(n)},
 \qquad
 V^{\pm}
 :=\Bbbk\otimes_{\mathcal A,\,q\mapsto\bar q}
   \mathbf U_{\mathcal A}^{\pm},
 \label{eq:endpoint-specializations}
\end{equation}
so that \(V^{\pm}\) are unital \(\Bbbk\)-algebras graded by
\(\bigoplus_{i\in I}\mathbb Z\alpha_i\), and \(\dot V\) is a
\(\Bbbk\)-algebra with orthogonal idempotents \(1_\lambda\)
(\(\lambda\in X\)).  We write \(E_i^{(a)},F_i^{(a)},1_\lambda\) again for
the images of these elements.  Lusztig's integral triangular
decomposition and its integral basis
\cite[\S\S23.2.1--23.2.2]{L} are statements about free
\(\mathcal A\)-modules and therefore remain direct after the base change
\eqref{eq:endpoint-specializations}:
\begin{equation}
 \dot V
 =\bigoplus_{\lambda\in X}V^{-}1_\lambda V^{+},
 \qquad
 V^{-}\otimes_{\Bbbk}\Bbbk1_\lambda\otimes_{\Bbbk}V^{+}
 \xrightarrow{\ \sim\ }V^{-}1_\lambda V^{+}
 \quad(\lambda\in X),
 \label{eq:endpoint-triangular-decomposition}
\end{equation}
the second map being multiplication.  In particular the specialized basis
gives
\begin{equation}
 0\neq1\in(V^{-})_{0},
 \qquad
 0\neq E_0\in(V^{+})_{\alpha_0} .
 \label{eq:endpoint-E0-nonzero}
\end{equation}

\medskip
\noindent\textbf{Letters and words.}
We use formal letters and words so that their labels remain
part of the data. Put
\begin{equation}
 \mathscr S:=\{1,-1\}\times I\times\mathbb Z_{\geq1},
 \qquad
 \mathscr S^\varepsilon
 :=\{\varepsilon\}\times I\times\mathbb Z_{\geq1}
 \quad(\varepsilon\in\{1,-1\}),
 \label{eq:endpoint-letters}
\end{equation}
and call the elements of \(\mathscr S\) \emph{letters}.
Write \(V^1:=V^+\) and \(V^{-1}:=V^-\).
Define
\(\operatorname{ev}_\varepsilon:
\mathscr S^\varepsilon\longrightarrow V^\varepsilon\) by
\begin{equation}
 \operatorname{ev}_\varepsilon(\varepsilon,i,a)
 :=
 \begin{cases}
  E_i^{(a)},&\varepsilon=1,\\
  F_i^{(a)},&\varepsilon=-1,
 \end{cases}
 \qquad
 \deg(\varepsilon,i,a):=\varepsilon a\alpha_i,
 \qquad
 \mathscr S^{\mathrm{nb}}
 :=
 \bigl\{(\varepsilon,i,a)\in\mathscr S
        \bigm|\ell_i\nmid a\bigr\}.
 \label{eq:endpoint-letter-definitions}
\end{equation}
A letter is \emph{nonblock} if it belongs to
\(\mathscr S^{\mathrm{nb}}\), and \emph{block} otherwise.
By Lemma~\ref{lem:endpoint-lattice-basics}(b)--(c),
\begin{equation}
 \mathscr S^{\mathrm{nb}}
 \subseteq\{1,-1\}\times I^\circ\times\mathbb Z_{\geq1},
 \qquad
 \deg(\mathscr S\setminus\mathscr S^{\mathrm{nb}})
 \subseteq L^\bullet.
 \label{eq:nonblock-color-nonzero}
\end{equation}

For each \(\varepsilon\in\{1,-1\}\), identify a letter with
the corresponding one-letter word and extend
\(\operatorname{ev}_\varepsilon\) to
\[
 \operatorname{ev}_\varepsilon:
 \bigsqcup_{m\geq0}(\mathscr S^\varepsilon)^m
 \longrightarrow V^\varepsilon
\]
by
\[
 \operatorname{ev}_\varepsilon(x_1,\ldots,x_m)
 :=
 \operatorname{ev}_\varepsilon(x_1)\cdots
 \operatorname{ev}_\varepsilon(x_m).
\]
The empty word has value \(1\in V^\varepsilon\).

The set of words with a terminal weight is
\begin{equation}
 \mathscr M
 :=
 \bigsqcup_{m\geq0}\bigl(\mathscr S^m\times X\bigr),
 \qquad
 \mathbf w=(x_1,\ldots,x_m;\eta).
 \label{eq:endpoint-words}
\end{equation}
For \(x_j=(\varepsilon_j,i_j,a_j)\), define the intermediate
weights and the map
\(\operatorname{ev}_{\dot V}:\mathscr M\longrightarrow\dot V\) by
\begin{equation}
 \begin{aligned}
  \lambda_j(\mathbf w)
  &:=\eta+\sum_{u=j+1}^m\deg(x_u)
  \qquad(0\leq j\leq m),\\
  \operatorname{ev}_{\dot V}(\mathbf w)
  &:=
  \begin{cases}
   1_\eta,&m=0,\\[2pt]
   \displaystyle
   \prod_{j=1}^m
   \bigl(
    \operatorname{ev}_{\varepsilon_j}(x_j)
    1_{\lambda_j(\mathbf w)}
   \bigr),&m\geq1.
  \end{cases}
 \end{aligned}
 \label{eq:endpoint-word-eval}
\end{equation}
The product is taken in increasing order of \(j\).
Each factor denotes the modified divided power
\(E_{i_j}^{(a_j)}1_{\lambda_j(\mathbf w)}\) or
\(F_{i_j}^{(a_j)}1_{\lambda_j(\mathbf w)}\), according to
the sign \(\varepsilon_j\), and therefore belongs to \(\dot V\).
Since
\[
 \lambda_m(\mathbf w)=\eta,
 \qquad
 \lambda_{j-1}(\mathbf w)
 =\lambda_j(\mathbf w)+\deg(x_j),
\]
the factors have matching intermediate weights, and
\[
 \operatorname{ev}_{\dot V}(\mathbf w)
 \in
 1_{\lambda_0(\mathbf w)}\dot V1_\eta.
\]

We write both \(\operatorname{ev}_\varepsilon\) and \(\operatorname{ev}_{\dot V}\) as \(\operatorname{ev}\) for these evaluation
maps when the domain is clear. With this convention, for
\(m\geq1\),
\begin{equation}
 \begin{aligned}
  \operatorname{ev}(\mathbf w)
  &=
  \operatorname{ev}(x_1)\cdots
  \operatorname{ev}(x_m)1_\eta\\
  &=
  \bigl(
   1_{\lambda_0(\mathbf w)}
   \operatorname{ev}(x_1)
   1_{\lambda_1(\mathbf w)}
  \bigr)
  \cdots
  \bigl(
   1_{\lambda_{m-1}(\mathbf w)}
   \operatorname{ev}(x_m)
   1_{\lambda_m(\mathbf w)}
  \bigr),
 \end{aligned}
 \label{eq:endpoint-word-idempotents}
\end{equation}
where the first line is shorthand for the product defined in
\eqref{eq:endpoint-word-eval}.
Since \(\mathbf U_{\mathcal A}^{\pm}\) is generated by the
divided powers, the triangular decomposition
\eqref{eq:endpoint-triangular-decomposition} gives
\begin{equation}
 \dot V
 =
 \operatorname{span}_{\Bbbk}\operatorname{ev}(\mathscr M).
 \label{eq:endpoint-words-span}
\end{equation}
Define
\begin{equation}
 \begin{aligned}
  \mathscr B
  :={}&\Bigl\{
   \mathbf w=(x_1,\ldots,x_m;\eta)\in\mathscr M
   \ \Bigm|\
   \bigl(\exists i \in \{0,...,m\}:\ \lambda_i(\mathbf w)\notin X^{\bullet}\bigr)
   \ \vee\
   \bigl(\exists j \in \{1,...,m\}:\ x_j\in\mathscr S^{\mathrm{nb}}\bigr)
  \Bigr\},\\[2pt]
  \mathscr B^{\triangle}
  :={}&\left\{
   \mathbf w=(x_1,\ldots,x_m;\eta)\in\mathscr B
   \ \middle|\
   \begin{array}{c}
    \exists\,p\in\{0,\ldots,m\}:\
    x_i\in\mathscr S^{-1}\ (1\leq i\leq p),\\
    x_j\in\mathscr S^{1}\ (p<j\leq m)
   \end{array}
  \right\},\\[2pt]
  \mathfrak J
  :={}&\operatorname{span}_{\Bbbk}
   \operatorname{ev}(\mathscr B),
 \end{aligned}
 \label{eq:endpoint-bad-words}
\end{equation}
Thus \(\mathscr B^\triangle\) is the subset of \(\mathscr B\) consisting of triangularly ordered bad words, namely those in which all negative letters precede all positive letters.
Words in \(\mathscr B\) are called
\emph{bad}, the remaining words \emph{good}.

The subspace \(\mathfrak J\) imposes exactly the two selection rules that
quantum Frobenius would impose: forbidden weight idempotents vanish, and
only divided powers whose exponents are multiples of the scaling integers
survive.  At \(\ell_0=1\) the criterion of \cref{sec:scaled-root-datum}
is unavailable.  The next three results prove that these rules are
nevertheless compatible with triangular straightening and with
multiplication.

\medskip
\noindent\textbf{The half ideals.}
Let
\begin{equation}
 \mathfrak I^{+}
 :=\sum_{i\in I^{\circ}}V^{+}E_iV^{+},
 \qquad
 \mathfrak I^{-}
 :=\sum_{i\in I^{\circ}}V^{-}F_iV^{-} ,
 \label{eq:endpoint-half-ideals}
\end{equation}
with \(I^{\circ}=\{i\in I\mid\ell_i>1\}\subseteq I\setminus\{0\}\) as in
\cref{lem:endpoint-lattice-basics}(b).  By construction
\(\mathfrak I^{\pm}\) is the two-sided ideal of the unital algebra
\(V^{\pm}\) generated by \(\{E_i\}_{i\in I^{\circ}}\), resp.
\(\{F_i\}_{i\in I^{\circ}}\):
\begin{equation}
 V^{\pm}\mathfrak I^{\pm}\subseteq\mathfrak I^{\pm},
 \qquad
 \mathfrak I^{\pm}V^{\pm}\subseteq\mathfrak I^{\pm},
 \qquad
 \mathfrak I^{\pm}
 =\bigoplus_{\nu}\bigl(\mathfrak I^{\pm}\cap(V^{\pm})_{\nu}\bigr),
 \label{eq:half-ideal-properties}
\end{equation}
the last equality because the generators are root homogeneous.  These are
ideals of the half algebras only, not of \(\dot V\).

\begin{lemma}[Endpoint half ideals]
\label[lemma]{lem:endpoint-half-ideals}
For \(\varepsilon\in\{1,-1\}\) put \(V^{1}:=V^{+}\),
\(V^{-1}:=V^{-}\), \(\mathfrak I^{1}:=\mathfrak I^{+}\),
\(\mathfrak I^{-1}:=\mathfrak I^{-}\),
\[
 \mathscr W^{\varepsilon}
 :=\left\{
   (x_1,\ldots,x_m)
   \ \middle|\
   \begin{array}{c}
    m\in\mathbb Z_{\geq1},\quad
    x_j\in\mathscr S^{\varepsilon}\ (1\leq j\leq m),\\
    \exists j\in\{1,\ldots,m\}:\ x_j\in\mathscr S^{\mathrm{nb}}
   \end{array}
  \right\}.
\]
Then
\[
 \mathfrak I^{\varepsilon}
 =\operatorname{span}_{\Bbbk}
  \operatorname{ev}\bigl(\mathscr W^{\varepsilon}\bigr)
 \qquad(\varepsilon\in\{1,-1\}).
\]
\end{lemma}

\begin{proof}
\emph{Step 1:}
Let \((\varepsilon,i,a)\in\mathscr S^{\mathrm{nb}}\), so \(a\geq1\) and
\(\ell_i\nmid a\); by \cref{lem:endpoint-lattice-basics}(c),
\(\ell_i>1\) and \(i\in I^{\circ}\).  Over \(\mathcal A\) one has
\([a]_{q_i}=\sum_{j=0}^{a-1}q_i^{\,a-1-2j}\), and multiplication by
\(q_i^{2}-1\) telescopes to the identity of Laurent polynomials
\begin{equation}
 \bigl(q_i^{2}-1\bigr)[a]_{q_i}
 =q_i^{\,1-a}\bigl(q_i^{2a}-1\bigr)
 \qquad\text{in }\mathcal A .
 \label{eq:cleared-quantum-integer}
\end{equation}
Applying \(\mathcal A\to\Bbbk\), \(q\mapsto\bar q\), gives
\begin{equation}
 \bigl(q_{\Bbbk,i}^{2}-1\bigr)[a]_{q_{\Bbbk,i}}
 =q_{\Bbbk,i}^{\,1-a}\bigl(q_{\Bbbk,i}^{2a}-1\bigr).
 \label{eq:cleared-quantum-integer-specialized}
\end{equation}
Here \(q_{\Bbbk,i}\in\Bbbk^{\times}\), so
\(q_{\Bbbk,i}^{\,1-a}\in\Bbbk^{\times}\); and since
\(\operatorname{ord}(q_{\Bbbk,i}^{2})=\ell_i<\infty\),
\[
 q_{\Bbbk,i}^{2a}=\bigl(q_{\Bbbk,i}^{2}\bigr)^{a}=1
 \quad\Longleftrightarrow\quad
 \ell_i\mid a ,
\]
which fails by hypothesis.  Hence the right-hand side of
\eqref{eq:cleared-quantum-integer-specialized} is nonzero and, \(\Bbbk\)
being a field, so are both factors on the left:
\begin{equation}
 q_{\Bbbk,i}^{2}\neq1,
 \qquad
 [a]_{q_{\Bbbk,i}}\in\Bbbk^{\times}.
 \label{eq:quantum-integer-invertible}
\end{equation}
The divided-power identities
\[
 E_i^{(a-1)}E_i=[a]_{q_i}E_i^{(a)},
 \qquad
 F_i^{(a-1)}F_i=[a]_{q_i}F_i^{(a)}
\]
hold in \(\mathbf U_{\mathcal A}^{\pm}\), hence in \(V^{\pm}\) after the
base change \eqref{eq:endpoint-specializations}; they cannot be obtained
inside \(V^{\pm}\) by dividing by \([a]_{q_{\Bbbk,i}}!\), which may
vanish.  Multiplying by \([a]_{q_{\Bbbk,i}}^{-1}\) gives
\begin{equation}
 ev((1,i,a))=E_i^{(a)}
 =[a]_{q_{\Bbbk,i}}^{-1}E_i^{(a-1)}E_i\in\mathfrak I^{+},
 \qquad
  ev((-1,i,a))=F_i^{(a)}
 =[a]_{q_{\Bbbk,i}}^{-1}F_i^{(a-1)}F_i\in\mathfrak I^{-} .
 \label{eq:nonblock-letter-in-half-ideal}
\end{equation}

\emph{Step 2: \(\operatorname{span}_{\Bbbk}
\operatorname{ev}(\mathscr W^{\varepsilon})
\subseteq\mathfrak I^{\varepsilon}\).}
Let \((x_1,\ldots,x_m)\in\mathscr W^{\varepsilon}\) and choose \(j\) with
\(x_j\in\mathscr S^{\mathrm{nb}}\).  Then
\[
 \operatorname{ev}(x_1,\ldots,x_m)
 =\underbrace{\bigl(\operatorname{ev}(x_1)\cdots
   \operatorname{ev}(x_{j-1})\bigr)}_{\in V^{\varepsilon}}
  \operatorname{ev}(x_j)
  \underbrace{\bigl(\operatorname{ev}(x_{j+1})\cdots
   \operatorname{ev}(x_m)\bigr)}_{\in V^{\varepsilon}} ,
\]
empty products being \(1\).  By Step~1,
\(\operatorname{ev}(x_j)\in\mathfrak I^{\varepsilon}\), and
\eqref{eq:half-ideal-properties} gives
\(\operatorname{ev}(x_1,\ldots,x_m)\in\mathfrak I^{\varepsilon}\).
Taking \(\Bbbk\)-spans yields the inclusion.

\emph{Step 3: the reverse inclusion.}
Let \(i\in I^{\circ}\) and \(x,y\in V^{+}\).  Since \(V^{+}\) is spanned
by the elements \(\operatorname{ev}(u_1,\ldots,u_p)\) with
\(u_1,\ldots,u_p\in\mathscr S^{1}\) and \(p\in\mathbb N\), we may expand
\(x\) and \(y\) in such elements and obtain
\[
 xE_iy
 \in\operatorname{span}_{\Bbbk}
 \bigl\{\operatorname{ev}\bigl(u_1,\ldots,u_p,(1,i,1),v_1,\ldots,v_s\bigr)
 \bigm|u_\bullet,v_\bullet\in\mathscr S^{1}\bigr\}.
\]
Every word occurring here contains the letter \((1,i,1)\), and
\((1,i,1)\in\mathscr S^{\mathrm{nb}}\) because \(\ell_i>1\) does not
divide \(1\); hence every such word lies in \(\mathscr W^{1}\).  Summing
over \(i\in I^{\circ}\) gives
\(\mathfrak I^{+}\subseteq\operatorname{span}_{\Bbbk}
\operatorname{ev}(\mathscr W^{1})\).  The case \(\varepsilon=-1\) is
identical in \(V^{-}\).
\end{proof}

\subsection{Bad-word straightening and triangular decomposition}
\label{app:endpoint-straightening}

\begin{lemma}
\label[lemma]{lem:endpoint-bad-straightening}
\[
 \operatorname{span}_{\Bbbk}\operatorname{ev}(\mathscr B)
 =\operatorname{span}_{\Bbbk}\operatorname{ev}(\mathscr B^{\triangle}).
\]
\end{lemma}

\begin{proof}
Since \(\mathscr B^{\triangle}\subseteq\mathscr B\), one inclusion is
immediate.  For the reverse inclusion we first analyse the two local
replacements used in triangular straightening. We show that each local replacement sends a bad input word to a
\(\Bbbk\)-linear combination of bad words. To prove that badness is preserved, we argue contrapositively: if an output word occurring with nonzero coefficient is good, then the input word is good.

\emph{Different colors.}
Let \(i\neq j\), \(a,b\geq1\), \(\eta\in X\).  Then
\begin{equation}
 E_i^{(a)}F_j^{(b)}1_\eta
 =F_j^{(b)}1_{\eta+a\alpha_i}E_i^{(a)}1_\eta ,
 \label{eq:endpoint-commute}
\end{equation}
so the input word \(\bigl((1,i,a),(-1,j,b);\eta\bigr)\), with
intermediate weight \(\eta-b\alpha_j\), is replaced by the single word
\(\bigl((-1,j,b),(1,i,a);\eta\bigr)\), with intermediate weight
\(\eta+a\alpha_i\).  Assume the output word is good.  Then
\(\ell_i\mid a\), \(\ell_j\mid b\) and \(\eta\in X^{\bullet}\), so
\cref{lem:endpoint-lattice-basics}(c) gives
\(b\alpha_j\in L^{\bullet}\) and hence, by
\cref{lem:endpoint-lattice-basics}(a),
\[
 \eta-b\alpha_j\in\eta+L^{\bullet}\subseteq X^{\bullet}.
\]
Both input letters are block, its intermediate weight lies in
\(X^{\bullet}\), and the outer weights and outside data are unchanged;
hence the input word is good.

\emph{Equal colors.}
Let \(i\in I\), \(a,b\geq1\), \(\eta\in X\).  Specializing Lusztig's
modified divided-power identity \cite[\S23.1.3]{L} gives
\begin{equation}
 E_i^{(a)}F_i^{(b)}1_\eta
 =\sum_{t=0}^{\min(a,b)}
  \begin{bmatrix}
   a-b+\langle h_i,\eta\rangle\\t
  \end{bmatrix}_{q_{\Bbbk,i}}
  F_i^{(b-t)}1_{\eta+(a-t)\alpha_i}E_i^{(a-t)}1_\eta .
 \label{eq:endpoint-straightening}
\end{equation}
In every summand
\(-(b-t)\alpha_i+(a-t)\alpha_i=(a-b)\alpha_i\), so the two outer weights
are preserved.  Set \(E_i^{(0)}=F_i^{(0)}=1\) and delete every
zero-exponent factor from the corresponding word; if \(t=a\), resp.
\(t=b\), the deleted intermediate weight \(\eta+(a-t)\alpha_i\) equals the
right, resp. the left, outer weight.  Hence every summand is again the
evaluation of a word in \(\mathscr M\).

Fix a summand with nonzero coefficient and suppose its word is good.
Write \(\ell:=\ell_i\).  Goodness gives
\[
 \ell\mid a-t,
 \qquad
 \ell\mid b-t,
 \qquad
 \eta\in X^{\bullet},
\]
the first two conditions holding trivially when the corresponding factor
was deleted, since then the exponent is \(0\).  Therefore
\[
 a-b=(a-t)-(b-t)\in\ell\mathbb Z,
 \qquad
 \langle h_i,\eta\rangle\in\ell\mathbb Z,
 \qquad
 a-b+\langle h_i,\eta\rangle\in\ell\mathbb Z ,
\]
and \cref{lem:endpoint-lucas}, applied with \(\xi=q_{\Bbbk,i}\) and
\(e=\ell=\operatorname{ord}(q_{\Bbbk,i}^{2})\), forces \(\ell\mid t\)
because the coefficient is nonzero.  Consequently
\[
 \ell\mid a,
 \qquad
 \ell\mid b,
 \qquad
 \eta-b\alpha_i\in\eta+L^{\bullet}\subseteq X^{\bullet},
\]
so both input letters are block and the input intermediate weight lies in
\(X^{\bullet}\).  As before the outer weights and all outside data agree
with those of the good output; hence the input word is good.

By contraposition, in both cases every nonzero output word obtained from
a bad input word is bad.

\emph{Global induction.}
For \(\mathbf w=(x_1,\ldots,x_m;\eta)\in\mathscr M\), put
\[
 \operatorname{inv}(\mathbf w)
 :=
 \#\bigl\{(r,s)\mid r<s,\ 
 x_r\in\mathscr S^{1},\ x_s\in\mathscr S^{-1}\bigr\}.
\]
We prove by induction on \(\operatorname{inv}(\mathbf w)\) that
\[
 \mathbf w\in\mathscr B
 \quad\Longrightarrow\quad
 \operatorname{ev}(\mathbf w)
 \in
 \operatorname{span}_{\Bbbk}\operatorname{ev}(\mathscr B^\triangle).
\]

If \(\operatorname{inv}(\mathbf w)=0\), no positive letter precedes a
negative letter.  Hence all negative letters precede all positive
letters, and therefore
\(\mathbf w\in\mathscr B^\triangle\).

Now suppose that
\[
 N:=\operatorname{inv}(\mathbf w)>0
\]
and that the assertion holds for all bad words of inversion number
strictly smaller than \(N\).  The sign sequence of \(\mathbf w\)
contains an adjacent pair
\[
 x_u\in\mathscr S^{1},
 \qquad
 x_{u+1}\in\mathscr S^{-1}.
\]
Apply the appropriate local replacement above.  After omitting all terms
with zero coefficient, this gives
\[
 \operatorname{ev}(\mathbf w)
 =
 \sum_{\mathbf w'}c_{\mathbf w'}\operatorname{ev}(\mathbf w'),
 \qquad
 c_{\mathbf w'}\in\Bbbk^\times.
\]
By the local analysis above, every output word
\(\mathbf w'\) occurring in this sum is bad.

It remains to compare inversion numbers.  If both letters of the chosen
pair survive, their sign pattern changes from
\[
 (+,-)\quad\text{to}\quad(-,+),
\]
so the inversion formed by this pair disappears. The remaining two possible outside signs
contribute no change of the inversion.  Consequently,
\[
 \operatorname{inv}(\mathbf w')
 =
 \operatorname{inv}(\mathbf w)-1
\]
whenever both letters survive.

If at least one letter of the pair is deleted, the chosen inversion
again disappears.  Moreover, deleting a letter can only remove
inversions involving that letter, and the relative order of all
remaining letters is unchanged.  Hence in this case
\[
 \operatorname{inv}(\mathbf w')
 \leq
 \operatorname{inv}(\mathbf w)-1.
\]
Thus every output word in the preceding sum satisfies
\[
 \mathbf w'\in\mathscr B,
 \qquad
 \operatorname{inv}(\mathbf w')<N.
\]
The induction hypothesis therefore gives
\[
 \operatorname{ev}(\mathbf w')
 \in
 \operatorname{span}_{\Bbbk}\operatorname{ev}(\mathscr B^\triangle)
\]
for every such \(\mathbf w'\).  Substituting these expressions into the
straightening formula yields
\[
 \operatorname{ev}(\mathbf w)
 \in
 \operatorname{span}_{\Bbbk}\operatorname{ev}(\mathscr B^\triangle).
\]
This proves
\[
 \operatorname{span}_{\Bbbk}\operatorname{ev}(\mathscr B)
 \subseteq
 \operatorname{span}_{\Bbbk}\operatorname{ev}(\mathscr B^\triangle).
\]
The reverse inclusion follows from
\(\mathscr B^\triangle\subseteq\mathscr B\), and the lemma follows.

\end{proof}

\begin{proposition}[Half-ideal decomposition of \(\mathfrak J\)]
\label[proposition]{prop:endpoint-half-ideal-decomposition}
\begin{equation}
 \mathfrak J
 =\bigoplus_{\lambda\in X\setminus X^{\bullet}}
   V^{-}1_\lambda V^{+}
 \ \oplus
 \bigoplus_{\lambda\in X^{\bullet}}
   \bigl(\mathfrak I^{-}1_\lambda V^{+}
        +V^{-}1_\lambda\mathfrak I^{+}\bigr),
 \label{eq:endpoint-half-ideal-form}
\end{equation}
\begin{equation}
 \mathfrak J
 =\bigoplus_{\nu\in\bigoplus_{i\in I}\mathbb Z\alpha_i}
   \bigl(\mathfrak J\cap(\dot V)_{\nu}\bigr),
 \label{eq:endpoint-half-ideal-grading}
\end{equation}
and, as root-graded \(\Bbbk\)-vector spaces,
\begin{equation}
 \dot V/\mathfrak J
 \cong
 \bigoplus_{\lambda\in X^{\bullet}}
 \bigl(V^{-}/\mathfrak I^{-}\bigr)
 \otimes_{\Bbbk}\Bbbk1_\lambda
 \otimes_{\Bbbk}\bigl(V^{+}/\mathfrak I^{+}\bigr).
 \label{eq:endpoint-quotient-decomposition}
\end{equation}
\end{proposition}

\begin{proof}
By \cref{lem:endpoint-bad-straightening},
\(\mathfrak J=\operatorname{span}_{\Bbbk}
\operatorname{ev}(\mathscr B^{\triangle})\).

\emph{Left side contained in right side.}
Let \(\mathbf w=(x_1,\ldots,x_m;\eta)\in\mathscr B^{\triangle}\) with
splitting index \(p\), and put \(\lambda:=\lambda_p(\mathbf w)\).  By
\eqref{eq:endpoint-word-idempotents},
\[
 \operatorname{ev}(\mathbf w)
 =\bigl(\operatorname{ev}(x_1)\cdots\operatorname{ev}(x_p)\bigr)
  1_\lambda
  \bigl(\operatorname{ev}(x_{p+1})\cdots\operatorname{ev}(x_m)1_\eta\bigr)
 \in V^{-}1_\lambda V^{+} .
\]
If \(\lambda\notin X^{\bullet}\), this lies in the first sum of
\eqref{eq:endpoint-half-ideal-form}.  Assume \(\lambda\in X^{\bullet}\).
If every letter of \(\mathbf w\) were block, then
\eqref{eq:nonblock-color-nonzero} would give
\(\deg(x_j)\in L^{\bullet}\) for all \(j\), hence
\[
 \lambda_j(\mathbf w)\in\lambda+L^{\bullet}\subseteq X^{\bullet}
 \qquad(0\leq j\leq m)
\]
by \cref{lem:endpoint-lattice-basics}(a), contradicting
\(\mathbf w\in\mathscr B\).  So some \(x_j\) is nonblock.  If
\(j\leq p\), then \((x_1,\ldots,x_p)\in\mathscr W^{-1}\) and
\cref{lem:endpoint-half-ideals} places \(\operatorname{ev}(\mathbf w)\)
in \(\mathfrak I^{-}1_\lambda V^{+}\); if \(j>p\), then
\((x_{p+1},\ldots,x_m)\in\mathscr W^{1}\) and
\(\operatorname{ev}(\mathbf w)\in V^{-}1_\lambda\mathfrak I^{+}\).

\emph{Right side contained in left side.}
Let \(\lambda\in X\setminus X^{\bullet}\).  Expanding both halves of
\(V^{-}1_\lambda V^{+}\) in evaluated words gives \(\Bbbk\)-linear
combinations of \(\operatorname{ev}(\mathbf w)\) with
\(\mathbf w\in\mathscr B^{\triangle}\) and
\(\lambda_p(\mathbf w)=\lambda\notin X^{\bullet}\); hence
\(V^{-}1_\lambda V^{+}\subseteq\mathfrak J\).  Let
\(\lambda\in X^{\bullet}\).  By \cref{lem:endpoint-half-ideals},
\(\mathfrak I^{-}\) is spanned by
\(\operatorname{ev}(\mathscr W^{-1})\) and \(V^{+}\) by evaluated
positive words, so \(\mathfrak I^{-}1_\lambda V^{+}\) is spanned by
\(\operatorname{ev}(\mathbf w)\) with
\(\mathbf w\in\mathscr B^{\triangle}\) containing a nonblock letter;
similarly for \(V^{-}1_\lambda\mathfrak I^{+}\).  Both are therefore
contained in \(\mathfrak J\).

\emph{Directness, grading, quotient.}
The summands of \eqref{eq:endpoint-half-ideal-form} lie in pairwise
distinct blocks \(V^{-}1_\lambda V^{+}\) of
\eqref{eq:endpoint-triangular-decomposition}, so the sum is direct.  Each
summand is root graded, the blocks by
\eqref{eq:endpoint-triangular-decomposition} and
\(\mathfrak I^{\pm}\) by \eqref{eq:half-ideal-properties}; this gives
\eqref{eq:endpoint-half-ideal-grading}.
Finally, fix \(\lambda\in X^\bullet\) and let
\[
 p^\pm:V^\pm\twoheadrightarrow V^\pm/\mathfrak I^\pm
\]
be the quotient maps.  Since tensoring over the field \(\Bbbk\) is
exact, the standard quotient--tensor isomorphism gives
\[
 \frac{V^-\otimes_{\Bbbk}\Bbbk1_\lambda\otimes_{\Bbbk}V^+}
 {\mathfrak I^-\otimes_{\Bbbk}\Bbbk1_\lambda\otimes_{\Bbbk}V^+
  +V^-\otimes_{\Bbbk}\Bbbk1_\lambda\otimes_{\Bbbk}\mathfrak I^+}
 \cong
 (V^-/\mathfrak I^-)\otimes_{\Bbbk}\Bbbk1_\lambda
 \otimes_{\Bbbk}(V^+/\mathfrak I^+).
\]
Equivalently,
\[
 \ker\bigl(p^-\otimes\operatorname{id}\otimes p^+\bigr)
 =
 \mathfrak I^-\otimes_{\Bbbk}\Bbbk1_\lambda\otimes_{\Bbbk}V^+
 +
 V^-\otimes_{\Bbbk}\Bbbk1_\lambda\otimes_{\Bbbk}\mathfrak I^+.
\]
Transporting this kernel through the multiplication isomorphism in
\eqref{eq:endpoint-triangular-decomposition} yields a root-graded
isomorphism
\[
 \frac{V^-1_\lambda V^+}
 {\mathfrak I^-1_\lambda V^+
  +V^-1_\lambda\mathfrak I^+}
 \cong
 (V^-/\mathfrak I^-)\otimes_{\Bbbk}\Bbbk1_\lambda
 \otimes_{\Bbbk}(V^+/\mathfrak I^+).
\]
For \(\lambda\notin X^\bullet\),
\eqref{eq:endpoint-half-ideal-form} gives
\[
 V^-1_\lambda V^+\subseteq\mathfrak J,
\]
so the entire \(\lambda\)-block vanishes in \(\dot V/\mathfrak J\).
Since the triangular decomposition in
\eqref{eq:endpoint-triangular-decomposition} is an internal direct sum,
taking the direct sum of the surviving blocks over
\(\lambda\in X^\bullet\) gives
\[
 \dot V/\mathfrak J
 \cong
 \bigoplus_{\lambda\in X^\bullet}
 (V^-/\mathfrak I^-)\otimes_{\Bbbk}\Bbbk1_\lambda
 \otimes_{\Bbbk}(V^+/\mathfrak I^+),
\]
which is \eqref{eq:endpoint-quotient-decomposition}.

\end{proof}

\subsection{Proof of the quotient theorem}
\label{app:endpoint-quotient-proof}

\begin{proof}[Proof of \cref{thm:endpoint-triangular-quotient}]
\emph{Two-sidedness.}
Let
\[
 \mathbf u=(u_1,\ldots,u_p;\theta)\in\mathscr M,
 \qquad
 \mathbf w=(x_1,\ldots,x_m;\eta)\in\mathscr B.
\]
Since the \(1_\lambda\) are orthogonal idempotents,
\[
 \operatorname{ev}(\mathbf u)\operatorname{ev}(\mathbf w)
 =
 \begin{cases}
  \operatorname{ev}
  (u_1,\ldots,u_p,x_1,\ldots,x_m;\eta),
  &\theta=\lambda_0(\mathbf w),\\
  0,&\theta\neq\lambda_0(\mathbf w),
 \end{cases}
\]
and
\[
 \operatorname{ev}(\mathbf w)\operatorname{ev}(\mathbf u)
 =
 \begin{cases}
  \operatorname{ev}
  (x_1,\ldots,x_m,u_1,\ldots,u_p;\theta),
  &\eta=\lambda_0(\mathbf u),\\
  0,&\eta\neq\lambda_0(\mathbf u).
 \end{cases}
\]

Suppose first that the product on the left is nonzero, and denote the
corresponding concatenated word by \(\mathbf v\).  Then
\[
 \lambda_{p+j}(\mathbf v)=\lambda_j(\mathbf w)
 \qquad(0\leq j\leq m).
\]
Thus every letter \(x_j\) of \(\mathbf w\), and every boundary weight
\(\lambda_j(\mathbf w)\), occurs unchanged in \(\mathbf v\).  Hence a
nonblock letter or a forbidden boundary weight witnessing
\(\mathbf w\in\mathscr B\) also witnesses
\(\mathbf v\in\mathscr B\).

Likewise, if the product on the right is nonzero, then its concatenated
word \(\mathbf v\) satisfies
\[
 \lambda_j(\mathbf v)=\lambda_j(\mathbf w)
 \qquad(0\leq j\leq m),
\]
and again all letters and boundary weights of \(\mathbf w\) occur
unchanged in \(\mathbf v\).  Therefore \(\mathbf v\in\mathscr B\) in
this case as well.  Consequently,
\[
 \operatorname{ev}(\mathscr M)\operatorname{ev}(\mathscr B)
 \subseteq
 \{0\}\cup\operatorname{ev}(\mathscr B),
 \qquad
 \operatorname{ev}(\mathscr B)\operatorname{ev}(\mathscr M)
 \subseteq
 \{0\}\cup\operatorname{ev}(\mathscr B).
\]
Since
\[
 \dot V=\operatorname{span}_{\Bbbk}\operatorname{ev}(\mathscr M),
 \qquad
 \mathfrak J
 =\operatorname{span}_{\Bbbk}\operatorname{ev}(\mathscr B),
\]
by \eqref{eq:endpoint-words-span} and the definition of \(\mathfrak J\),
\(\Bbbk\)-bilinearity gives
\[
 \dot V\mathfrak J\subseteq\mathfrak J,
 \qquad
 \mathfrak J\dot V\subseteq\mathfrak J.
\]
Thus \(\mathfrak J\) is a two-sided ideal of \(\dot V\).  Its
root-gradedness was established in
\cref{prop:endpoint-half-ideal-decomposition}; hence
\(\mathcal Q=\dot V/\mathfrak J\) is a root-graded
\(\Bbbk\)-algebra.

\emph{Support.}
Let
\[
 \overline V^\pm:=V^\pm/\mathfrak I^\pm,
 \qquad
 L^\bullet_+:=\bigoplus_{i\in I}\mathbb N\alpha_i^\bullet.
\]
By \cref{lem:endpoint-half-ideals}, \(\overline V^+\) is spanned by
the images of evaluated positive words containing no nonblock letter.
For a letter \(x=(1,i,a)\in\mathscr S^1\setminus
\mathscr S^{\mathrm{nb}}\), one has \(\ell_i\mid a\), and therefore
\[
 \deg(x)=a\alpha_i
 =\frac{a}{\ell_i}\alpha_i^\bullet
 \in\mathbb N\alpha_i^\bullet.
\]
It follows that
\[
 \operatorname{Supp}(\overline V^+)\subseteq L^\bullet_+.
\]
The same argument for negative words gives
\[
 \operatorname{Supp}(\overline V^-)\subseteq-L^\bullet_+.
\]

By \eqref{eq:endpoint-quotient-decomposition},
\[
 \mathcal Q
 \cong
 \bigoplus_{\lambda\in X^\bullet}
 \overline V^-\otimes_{\Bbbk}\Bbbk1_\lambda
 \otimes_{\Bbbk}\overline V^+
\]
as root-graded \(\Bbbk\)-vector spaces and more explicitly,
\[
 \mathcal Q_\nu
 \cong
 \bigoplus_{\lambda\in X^\bullet}
 \ \bigoplus_{\substack{\beta,\gamma\in L^\bullet_+\\
                         \gamma-\beta=\nu}}
 (\overline V^-)_{-\beta}
 \otimes_{\Bbbk}\Bbbk1_\lambda
 \otimes_{\Bbbk}(\overline V^+)_{\gamma}.
\]  Since every \(1_\lambda\) has
root degree \(0\), a nonzero homogeneous component
\((\mathcal Q)_\nu\) can occur only if
\[
 \nu\in-L^\bullet_++L^\bullet_+\subseteq L^\bullet.
\]
Hence
\[
 \bigl\{\nu\mid(\mathcal Q)_\nu\neq0\bigr\}
 \subseteq L^\bullet.
\]

\emph{Survival of \(E_01_0\).}
By \cref{lem:endpoint-lattice-basics}(b),
\[
 \alpha_0=\alpha_0^\bullet\in L^\bullet\subseteq X^\bullet,
\]
while \(0\in X^\bullet\) because
\(\langle h_i,0\rangle=0\) for every \(i\in I\).

By \cref{lem:endpoint-half-ideals}, \(\mathfrak I^-\) is spanned by
evaluations of words in \(\mathscr W^{-1}\).  Every such word is
nonempty and has strictly negative root degree, so
\[
 (\mathfrak I^-)_0=0.
\]
Similarly, every word in \(\mathscr W^1\) contains a nonblock letter
\((1,i,a)\).  By \eqref{eq:nonblock-color-nonzero}, its color satisfies
\(i\in I^\circ\), and hence \(i\neq0\) by
\cref{lem:endpoint-lattice-basics}(b).  Since all letters of the word
have positive degree, its degree has a strictly positive
\(\alpha_i\)-coefficient and therefore cannot equal \(\alpha_0\).
Consequently,
\[
 (\mathfrak I^+)_{\alpha_0}=0.
\]
Together with \eqref{eq:endpoint-E0-nonzero}, these equalities give
\[
 1+\mathfrak I^-\neq0
 \quad\text{in }\overline V^-,
 \qquad
 E_0+\mathfrak I^+\neq0
 \quad\text{in }\overline V^+.
\]

In the \(\lambda=\alpha_0\) summand of
\eqref{eq:endpoint-quotient-decomposition}, the multiplication
isomorphism uses
\[
 1_{\alpha_0}E_0=E_01_0.
\]
Therefore
\[
 E_01_0+\mathfrak J
 \longleftrightarrow
 (1+\mathfrak I^-)\otimes1_{\alpha_0}
 \otimes(E_0+\mathfrak I^+).
\]
All three factors on the right are nonzero, so their pure tensor is
nonzero over the field \(\Bbbk\).  Hence
\[
 E_01_0+\mathfrak J
 \in
 (\mathcal Q)_{\alpha_0}\setminus\{0\},
\]
as required.
\end{proof}

\bigskip
\centerline{\bf Acknowledgments}
The project is partially supported by the Simons Foundation grant no. MP-TSM-00002518.

\vskip 20pt

\noindent{\bf Data Availability Statement.}
All data of this work are included in the manuscript and its supplemental materials.

\bigskip
\noindent
\textsc{Ziyang Jiang}, Department of Mathematics, North Carolina State
University, Raleigh, NC 27695, USA\\
Email address: \href{mailto:zjiang29@ncsu.edu}{\texttt{zjiang29@ncsu.edu}}

\medskip
\noindent
\textsc{Naihuan Jing}, Department of Mathematics, North Carolina State
University, Raleigh, NC 27695, USA\\
Email address: \href{mailto:jing@ncsu.edu}{\texttt{jing@ncsu.edu}}

\end{document}